\documentclass[12 pt]{amsart}

\usepackage{hyperref}
\usepackage{etex}
\usepackage[shortlabels]{enumitem}
\usepackage{amsmath}
\usepackage{amsxtra}
\usepackage{amscd}
\usepackage{amsthm}
\usepackage{amsfonts}
\usepackage{amssymb}
\usepackage{eucal}
\usepackage[all]{xy}
\usepackage{graphicx}
\usepackage{tikz-cd}
\usepackage{mathrsfs}
\usepackage{subfiles}
\usepackage{mathpazo}
\usepackage[colorinlistoftodos, textsize=tiny]{todonotes}
\usepackage{morefloats}
\usepackage{pdfpages}
\usepackage{thm-restate}
\usepackage[utf8]{inputenc}
\usepackage{epigraph}
\usepackage{csquotes}
\usepackage[margin=1.5in]{geometry}
\usepackage{adjustbox}
\usepackage{stmaryrd}
\usepackage{microtype}
\usepackage{scalerel}
\usepackage{stackengine}
\stackMath
\newcommand\reallywidehat[1]{%
\savestack{\tmpbox}{\stretchto{%
  \scaleto{%
    \scalerel*[\widthof{\ensuremath{#1}}]{\kern-.6pt\bigwedge\kern-.6pt}%
    {\rule[-\textheight/2]{1ex}{\textheight}}
  }{\textheight}%
}{0.5ex}}%
\stackon[1pt]{#1}{\tmpbox}%
}
\graphicspath{ {images/} }

\RequirePackage{color}
\definecolor{myred}{rgb}{0.75,0,0}
\definecolor{mygreen}{rgb}{0,0.5,0}
\definecolor{myblue}{rgb}{0,0,0.65}

\usepackage{color}

\usepackage{hyperref}
\hypersetup{citecolor=blue}
\usepackage{tikz}
\usetikzlibrary{matrix,arrows,decorations.pathmorphing}

\theoremstyle{plain}
\newtheorem{theorem}[subsubsection]{Theorem}
\newtheorem{proposition}[subsubsection]{Proposition}
\newtheorem{lemma}[subsubsection]{Lemma}
\newtheorem{corollary}[subsubsection]{Corollary}

\theoremstyle{definition}
\newtheorem{definition}[subsubsection]{Definition}
\newtheorem{remark}[subsubsection]{Remark}
\newtheorem{example}[subsubsection]{Example}

\newtheorem{question}[subsubsection]{Question}

\theoremstyle{remark}
\newtheorem{notation}[subsubsection]{Notation}

\numberwithin{equation}{section}
\newcommand\nc{\newcommand}
\nc\on{\operatorname}
\nc\renc{\renewcommand}
\DeclareMathOperator\rk{rk}

\newcommand*{\sheafhom}{\mathscr{H}\kern -2pt om}
\newcommand*{\sheafend}{\mathscr{E}\kern -1pt nd}

\title{Geometric local systems on very general curves and isomonodromy}
\author{Aaron Landesman, Daniel Litt}
\date{\today}

\begin{document}

\begin{abstract}
We show that the minimum rank of a non-isotrivial local system of geometric origin on a
suitably general $n$-pointed curve of genus $g$ is at least $2\sqrt{g+1}$. We apply this result to resolve conjectures of Esnault-Kerz and Budur-Wang. The main input is an analysis of stability
properties of flat vector bundles under isomonodromic deformations, which
additionally answers questions of Biswas, Heu, and Hurtubise.
\end{abstract}

\maketitle

\section{Introduction}
\label{section:introduction}
\subsection{Overview}
\label{subsection:overview}
We work over the complex numbers $\mathbb{C}$. 
The main result of this paper, \autoref{theorem:very-general-VHS}, is that an
analytically very general $n$-pointed curve of genus $g$ 
(defined in \autoref{definition:general})
does not carry any
non-isotrivial polarizable integral variations of Hodge structure of rank less than $2\sqrt{g+1}$. In particular, an
analytically very general $n$-pointed curve of genus $g$ carries no geometric local
systems of rank less than $2\sqrt{g+1}$ with infinite monodromy, as we show in
\autoref{corollary:geometric-local-systems}. This is a strong restriction on the
topology of smooth proper maps to an analytically very general curve, and, as pointed out by H\'el\`ene Esnault, contradicts conjectures of Esnault-Kerz \cite[Conjecture 1.1]{esnault2021local} and Budur-Wang \cite[Conjecture 10.3.1]{budur2020absolute}, as explained in \autoref{corollary:non-density-of-geometric-local-systems}. 

The above results rely on an analysis of stability properties of isomonodromic deformations of flat vector bundles with regular singularities, and require correcting a number of errors in the literature on this topic.
We next state our main results on stability properties of isomonodromic
deformations of flat vector bundles.
Let $C_0$ be the central fiber of a family of curves $\mathscr{C}\to \Delta$
with $\Delta$ a contractible domain, and let $(E_0, \nabla_0)$ be a vector
bundle with flat connection
on $C_0$. Recall that, loosely speaking, the isomonodromic deformation of $(E_0, \nabla_0)$  is the deformation $(\mathscr{E}, \nabla)$  of $(E_0, \nabla_0)$ to $\mathscr{C}/\Delta$, such that the monodromy of the connection is constant.

In \autoref{corollary:counterexample}, we construct a flat vector bundle on a smooth proper curve over $\mathbb{C}$, whose isomonodromic deformations 
to a nearby curve
are never semistable. 
(See \autoref{definition:nearby} for precise definitions.)
The construction arises from the ``Kodaira-Parshin trick," and contradicts  earlier claimed theorems of Biswas, Heu, and Hurtubise (\cite[Theorem 1.3]{BHH:logarithmic},
\cite[Theorem 1.3]{BHH:irregular}, and
\cite[Theorem 1.2]{BHH:parabolic}), which imply that such a construction is impossible. See \autoref{remark:bhh-error} for a discussion of the errors in those papers. 

As a complement to this example, we show in
\autoref{theorem:hn-constraints-parabolic} that any logarithmic flat vector
bundle admits an isomonodromic deformation to a nearby curve which is \emph{close} to semistable,
in a suitable sense,
and moreover is (parabolically) semistable if the rank is small compared to the genus of the curve.
While our results contradict those of \cite{BHH:logarithmic, BHH:irregular,
BHH:parabolic}, our methods owe those papers a substantial debt. 
Biswas, Heu,
and Hurtubise 
pitch the question of isomonodromically deforming a vector bundle to a
semistable vector bundle (see \autoref{question:bhh})
as an
analogue of Hilbert's 21st problem, also known as the Riemann-Hilbert problem.

The semistability property of \autoref{theorem:hn-constraints-parabolic} is also the main
input to our Hodge-theoretic main results, mentioned above.
The applications to polarizable variations of Hodge structures come from the fact that flat vector bundles underlying polarizable variations are rarely (parabolically) semistable, due to well-known curvature properties of Hodge bundles.

\subsection{Main Hodge-theoretic results}
\label{subsection:hodge-intro}
Results from this subsection, \autoref{subsection:hodge-intro}, as well as the
next, \autoref{subsection:isomonodromy-intro}, will be proven later in the
paper, as detailed in \autoref{subsection:organization}.
For convenience, throughout the paper, out main results will primarily be stated for hyperbolic curves.
\begin{definition}
	\label{definition:hyperbolic}
	Let $C$ be a curve over $\mathbb C$ of genus $g$ and $D \subset C$ a
	reduced effective divisor of degree $n$. 
	Call $(C, D)$ {\em hyperbolic} if $C$ is a smooth proper connected curve
	and
	either $g \geq 2$ and $n\geq 0$, $g = 1$ and $n > 0$, or $g =0$ and $n > 2$.
	We call an $n$-pointed curve $(C, x_1, \ldots, x_n)$ {\em hyperbolic} if
	$(C, x_1 +\cdots + x_n)$ is hyperbolic.
\end{definition}
\begin{remark}
	\label{remark:}
	Equivalently, $(C,D)$ is hyperbolic if and only if it has no infinitesimal automorphisms,
i.e., $H^0(C, T_C(-D)) = 0$.
\end{remark}

We will also work with the following analytic notion of a (very) general general point.

\begin{definition}
	\label{definition:general}
	A property holds for an {\em analytically general} point of a complex orbifold $X$, 
	if there exists a nowhere dense
	closed analytic subset $S \subset X$ so that the property holds
	on $X - S$.
	We say that a property holds for an
	{\em analytically very general} point if, locally on $X$, there exists a countable collection of nowhere dense closed analytic subsets such that the property holds on the complement of their union. If $\mathscr{M}_{g,n}$ is the analytic moduli stack of
$n$-pointed curves of genus $g$, we say that a property holds for an
analytically (very) general $n$-pointed curve if it holds for an analytically
(very) general point of $\mathscr{M}_{g,n}$.
\end{definition}
\begin{remark}\label{remark:}
	From the definition, it may appear that ``analytically very general'' is
	a local notion, while ``analytically general'' is a global notion.
	However, being ``analytically general'' also has the following
	equivalent local definition, which is more similar to the definition of
	``analytically very general'':
	locally on $X$, there exists a nowhere dense closed analytic subset such
	that the property holds on the complement of this subset.
\end{remark}

The main geometric consequence of this work is the following constraint on the
rank of non-isotrivial polarizable variations of Hodge structure (defined in
\autoref{section:hodge-theoretic-preliminaries}) on an analytically very general curve:
\begin{theorem}\label{theorem:very-general-VHS}
	Let $K$ be a number field with ring of integers
$\mathscr{O}_K$. Suppose 
$(C, x_1, \cdots, x_n)$ 
is an analytically very general $n$-pointed
hyperbolic curve of genus $g$, and $\mathbb{V}$ is a $\mathscr{O}_K$-local system on $C\setminus\{x_1, \cdots, x_n\}$ with infinite monodromy. 
Suppose additionally that for each embedding $\iota: \mathscr{O}_K\to \mathbb{C}$, $\mathbb{V}\otimes_{\mathscr{O}_K, \iota}\mathbb{C}$ underlies a polarizable complex variation of Hodge structure. 
Then, $$\on{rk}_{\mathscr{O}_K}(\mathbb{V})\geq 2\sqrt{g+1}.$$
\end{theorem}
\begin{remark}
	\label{remark:}
	Note that a result analogous to \autoref{theorem:very-general-VHS} does
	not hold for variations without an underlying $\mathscr{O}_K$-structure. 
	Indeed, every smooth proper curve of genus at least $2$ admits a polarizable complex  variation of Hodge structure of rank $2$ with infinite monodromy, arising from uniformization (see e.g.~\cite[bottom of p.~870]{simpson:constructing-VHS}).
\end{remark}

Let $X$ be a smooth variety. We say a complex local system $\mathbb{V}$ on $X$ is \emph{of geometric origin} if there exists a dense open $U\subset X$, and a smooth proper morphism $f: Y\to U$ such that $\mathbb{V}|_U$ is a direct summand of $R^if_*\mathbb{C}$ for some $i\geq 0$. As local systems of geometric origin satisfy the hypotheses of \autoref{theorem:very-general-VHS}, we have:
\begin{corollary}\label{corollary:geometric-local-systems}
Let $(C, x_1, \cdots, x_n)$ be an analytically very general hyperbolic $n$-pointed curve of genus $g$. If $\mathbb{V}$ is a local system on $C\setminus\{x_1, \cdots, x_n\}$ of geometric origin and with infinite monodromy, then $\dim_{\mathbb{C}}\mathbb{V}\geq 2\sqrt{g+1}$.
\end{corollary}
We will prove \autoref{corollary:geometric-local-systems} 
in \autoref{subsubsection:geometric-origin-proof}.
As a consequence of \autoref{corollary:geometric-local-systems} 
we obtain the following concrete geometric corollary:
\begin{corollary}
	\label{corollary:abelian-schemes}
	If $(C, x_1, \ldots, x_n)$ is an analytically
very general hyperbolic $n$-pointed genus $g$ curve, then any non-isotrivial
abelian scheme over $C\setminus\{x_1, \cdots, x_n\}$  has relative dimension at least $\sqrt{g+1}$.
	Similarly, any relative smooth proper curve over $C\setminus\{x_1, \cdots, x_n\}$
	has genus at least $\sqrt{g+1}$.
\end{corollary}
We will prove \autoref{corollary:abelian-schemes} 
in \autoref{subsubsection:abelian-scheme-proof}.
In \autoref{proposition:hilbert-modular},
we prove a variant of the above corollary
with a stronger bound on the genus
when the abelian scheme has real multiplication, corresponding to a map from 
$C\setminus\{x_1, \cdots, x_n\}$ to a Hilbert modular stack.

\begin{remark}
	\label{remark:remove-analytic}
	Najmuddin Fakhruddin has pointed out to us that one may replace ``analytically very general" with "very general" 
	in the usual algebraic sense in \autoref{corollary:abelian-schemes}. Indeed, if a very general curve carried a non-isotrivial Abelian scheme of relative dimension less than $\sqrt{g+1}$, the same would be true for an analytically very general curve, by spreading out. 
	We do not know how to analogously strengthen \autoref{theorem:very-general-VHS}---see  \autoref{question:non-abelian-Hodge}.
\end{remark}

\begin{remark}
    It is a well-known conjecture that integral local systems underlying a polarizable variation of Hodge structure are of geometric origin---see e.g.~\cite[Conjecture 12.4]{simpson62hodge} for a precise statement. \autoref{theorem:very-general-VHS} verifies this conjecture for local systems of rank less than $2\sqrt{g+1}$ on a analytically very general $n$-pointed hyperbolic curve of genus $g$, as local systems with finite monodromy arise from geometry.
\end{remark}	
	
We are grateful to H\'el\`ene Esnault for pointing out the following consequence of \autoref{corollary:geometric-local-systems} to us. We let $$\mathscr{M}_{B,r}(C\setminus \{x_1, \cdots, x_n\}):=\on{Hom}(\pi_1(C\setminus\{x_1, \cdots, x_n\}), \on{GL}_r(\mathbb{C}))\sslash \on{GL}_r(\mathbb{C})$$ be the \emph{character variety} parametrizing conjugacy classes of semisimple representations of $\pi_1(C\setminus \{x_1, \cdots, x_n\})$ into $\on{GL}_r(\mathbb{C})$. See e.g.~\cite{sikora2012character} for a useful primer on character varieties.
\begin{corollary}
    \label{corollary:non-density-of-geometric-local-systems}
    Let $(C, x_1, \cdots, x_n)$ be an analytically very general hyperbolic $n$-pointed curve of genus $g$. Then if $1<r<2\sqrt{g+1}$, the local systems of geometric origin are not Zariski-dense in the character variety $\mathscr{M}_{B,r}(C\setminus \{x_1, \cdots, x_n\})$.
\end{corollary}

\begin{remark}
     \autoref{corollary:non-density-of-geometric-local-systems} contradicts conjectures of Esnault-Kerz \cite[Conjecture 1.1]{esnault2021local} and Budur-Wang \cite[Conjecture 10.3.1]{budur2020absolute}, which imply the density of geometric local systems in the character variety of any smooth complex variety.
\end{remark}

We will prove \autoref{corollary:non-density-of-geometric-local-systems} in 
\autoref{subsubsection:non-density-proof}.

In what follows, we say a flat vector bundle has \emph{unitary monodromy} if the associated monodromy representation
$\rho:\pi_1(C)\to \on{GL}_n(\mathbb{C})$
has image with compact closure.
We will deduce the above results from
\autoref{theorem:isomonodromic-deformation-CVHS} below,
using that a discrete subset of the image of a unitary $\rho$ is finite.

\begin{theorem}\label{theorem:isomonodromic-deformation-CVHS}
Let $(C, x_1, \cdots, x_n)$ be an $n$-pointed hyperbolic curve of genus $g$.
Let $({E}, \nabla)$ be a flat vector bundle on
$C$ with $\on{rk}{E}<2\sqrt{g+1}$ and with regular singularities at the $x_i$. If an isomonodromic
deformation of ${(E, \nabla)}$ to an analytically general nearby $n$-pointed curve underlies a polarizable complex variation of Hodge structure, then $({E},\nabla)$ has unitary monodromy.
\end{theorem}

\subsection{Main results on isomonodromic deformations}
\label{subsection:isomonodromy-intro}
As remarked in \autoref{subsection:overview}, the Hodge-theoretic results of
\autoref{subsection:hodge-intro} arise from an analysis of the Harder-Narasimhan filtrations of isomonodromic deformations of flat vector bundles on curves. Our first such result is a counterexample to 
\cite[Theorem 1.3]{BHH:logarithmic},
\cite[Theorem 1.3]{BHH:irregular}, and
\cite[Theorem 1.2]{BHH:parabolic}, which demonstrates that the situation is somewhat more complicated than was previously believed --- there exist irreducible flat vector bundles whose isomonodromic deformations are never semistable. 

Specifically, \cite{BHH:logarithmic} ask the following
question. 
\begin{question}[\protect{\cite[p. 123]{BHH:logarithmic}}]
	\label{question:bhh}
	Let $X$ be a smooth proper curve, and $D\subset X$ a reduced effective divisor. Given a flat vector bundle $(E, \nabla)$ on $X$, with regular singularities along $D$, let $(E', \nabla')$ be the isomonodromic deformation of $(E,\nabla)$ to an analytically general nearby curve $(X', D')$. Is $E'$ semistable?
\end{question}

The main claim of \cite{BHH:logarithmic} is that \autoref{question:bhh} has a positive answer if $(E,\nabla)$ has irreducible monodromy and the genus of $X$ is at least $2$. However, the following results answer \autoref{question:bhh} in the negative, even in this case. See \autoref{remark:bhh-error} for a discussion of the errors in previous claims that \autoref{question:bhh} had a positive answer.

We use 
$\mathscr{M}_{g,n}$ to denote the analytic moduli stack of smooth proper curves with
geometrically connected fibers and $n$ distinct marked points.
\begin{theorem}
	\label{theorem:counterexample}
Let $g\geq 2$ be an integer. There exists a vector bundle with flat connection
$(\mathscr{F}, \nabla)$ on $\mathscr{M}_{g,1}$ such that for each fiber $C$ of
the forgetful morphism $\mathscr{M}_{g,1} \to \mathscr{M}_g$, the restriction of $(\mathscr{F}, \nabla)$ to $C$
\begin{enumerate}
    \item has semisimple monodromy and
    \item is not semistable.
\end{enumerate}
\end{theorem}
%

We also have the following variant, where the vector bundle has irreducible
monodromy, instead of just semisimple monodromy.
\begin{corollary}\label{corollary:counterexample}
Let $C$ be a smooth projective curve of genus at least $2$. There exists an irreducible flat
vector bundle $(E, \nabla)$ on $C$, whose isomonodromic deformations to
a nearby curve are never semistable.
\end{corollary}
\begin{proof}
	The restriction $(\mathscr{F}, \nabla)|_C$ from \autoref{theorem:counterexample} provides a semisimple flat vector bundle, each of whose flat summands has degree zero; by \autoref{theorem:counterexample}(2), its
isomonodromic deformation to a nearby curve is never semistable. Hence one of
the irreducible summands of
$(\mathscr{F}, \nabla)|_C$
satisfies the statement of the corollary.
\end{proof}

In a positive direction, we have have the following result, showing 
that the 
isomonodromic deformation of any semisimple flat vector bundle to an
analytically general nearby curve is close to being semistable, and moreover it is
semistable if the rank is small.

\begin{theorem}
	\label{theorem:hn-constraints}
	Let $(C,D)$ be hyperbolic of genus $g$ and let
	$({E}, \nabla)$ be a flat
	vector bundle on $C$ with regular singularities along $D$,
	and irreducible
monodromy.
Suppose
 $(E',\nabla')$ 
is an isomonodromic deformation
of $({E}, \nabla)$ to an analytically general nearby curve,
with Harder-Narasimhan filtration $0 = (F')^0 \subset (F')^1 \subset \cdots \subset
(F')^m =
E'$. For $1 \leq i \leq m$, let $\mu_i$ denote the slope of
$\on{gr}^{i}_{HN}E' := (F')^i/(F')^{i-1}$.
Then the following two properties hold.
\begin{enumerate}
	\item If $E'$ is not semistable, then for every $0 < i < m$, there
	exists $j < i < k$ with $$\rk \on{gr}^{j+1}_{HN}E'\cdot \rk
	\on{gr}^k_{HN}E'\geq g+1.$$
\item We have $0<\mu_i-\mu_{i+1}\leq 1$ for all $i<m$.
\end{enumerate}
\end{theorem}

In other words, the consecutive associated graded pieces of the generic
Harder-Narasimhan filtration have slope differing by at most one, and, if there
are multiple pieces of the generic Harder-Narasimhan filtration, many of them must have large rank relative to $g$.  This result is due to Bolibruch \cite[Proposition 5.6]{bolibrukh1990riemann} in the case that $g=0$ and $\on{rk}(E)=2$.

\autoref{theorem:hn-constraints} is a special case of
\autoref{theorem:hn-constraints-parabolic} below, where we allow certain
parabolic structures on the vector bundle $E$. These more general results are required for our Hodge-theoretic applications.
\begin{remark}
	\label{remark:}
	\autoref{theorem:hn-constraints} also holds without the hyperbolicity
	assumption, as we will explain. Nevertheless, it is convenient to make the assumption 
	so that curves have no infinitesimal automorphisms. In this case
	isomonodromic deformations are somewhat better behaved, see 
	\cite[p. 518]{Heu:universal-isomonodromic}.

	We now explain the proof of \autoref{theorem:hn-constraints} in the case
	$(C, D)$ is not hyperbolic.
	Suppose $(C, D)$ is not hyperbolic, so either $g = 1, n = 0$ or $g = 0, n \leq 2$. In
this case the fundamental group $\pi_1(C - \{x_1, \ldots, x_n\})$ is abelian.
This implies any irreducible representation of $\pi_1(C - \{x_1, \ldots, x_n\})$
is $1$-dimensional, so the corresponding flat vector bundle is a line bundle. In
this case, $E$ and $E'$ are semistable, so
\autoref{theorem:hn-constraints} still holds.
\end{remark}

As a corollary, we are able to salvage the main theorem of \cite{BHH:logarithmic} for flat vector bundles whose rank is small relative to $g$, using the AM-GM inequality.
The following corollary can be deduced directly from
\autoref{theorem:hn-constraints-parabolic} and AM-GM. 
It is also a special
case of \autoref{cor:stable-parabolic}.
\begin{corollary}\label{cor:stable}
	Let $(C, D)$ be a hyperbolic curve of genus $g$.
Let $({E}, \nabla)$ be an irreducible flat
vector bundle on $C$ with regular singularities along $D$, and suppose that
$\on{rk}(E)<2\sqrt{g+1}$. 
Then an isomonodromic deformation of 
$E$
to an analytically general nearby curve is  semistable. 
\end{corollary}
 As remarked above, \cite{BHH:logarithmic} claims an analogous theorem with no bound on the rank of $E$; our \autoref{corollary:counterexample} implies such a bound is necessary. Our methods are heavily inspired by those of \cite{BHH:logarithmic}, but our
technique requires some new input from Clifford's theorem for vector bundles.

These results appear to be new even for vector bundles with finite monodromy; we give some example applications in this case.
\begin{example}
	\label{example:splitting-type}
	In this example, we describe what 
	\autoref{theorem:hn-constraints}(2) tells us about splitting types
	of certain vector bundles on $\mathbb P^1$.
	Suppose we are given a finite group $G$ and
	a finite $G$-cover $f: X \to \mathbb P^1$.
	Consider the flat vector bundle $E := f_* \mathscr O_X$ on $\mathbb P^1$, with the connection $\nabla$ induced by the exterior derivative $d: \mathscr{O}_X\to \Omega^1_X$,
	which has regular singularities along the branch locus of $f$.
	Let $F$ be a summand of $E$ with irreducible monodromy
	and let $F'$ be an isomonodromic deformation of $F$ to a very general
	nearby pointed genus $0$ curve.
	Since every vector bundle on $\mathbb P^1$ is a sum of line bundles,
	we can write $F' = \mathscr O_{\mathbb P^1}(a_1)^{b_1} \oplus \cdots
	\oplus \mathscr O_{\mathbb P^1}(a_m)^{b_m}$,
	with $a_1 < a_2 < \cdots < a_m$ and $b_i > 0$.
	Then \autoref{theorem:hn-constraints}(1) tells us nothing, but
	\autoref{theorem:hn-constraints}(2) tells us that the $a_i$ are
	consecutive, i.e., $a_{i + 1} = a_i + 1$ for $1 \leq i \leq m-1$.

	Such $F'$ appear as a summand in $f'_* \mathscr O_{X'}$, where $f': X'
	\to \mathbb P^1$ is a general $G$-cover of $\mathbb P^1$.
\end{example}

\begin{example}
	\label{example:tschirnhausen}
	We now give a sample application of \autoref{cor:stable} to
	semistability of Tschirnhausen bundles of finite covers.
	Consider a family $\mathscr X \xrightarrow{\alpha} \mathscr{Z}
	\xrightarrow{\beta} \mathscr Y
	\xrightarrow{\gamma} B$
	where $\mathscr X \to B, \mathscr Y \to B$ and $\mathscr Z \to B$ are smooth proper curves
	with geometrically connected fibers,
	$\beta$ is finite locally free of degree $d$, $\beta\circ \alpha$ is an $S_d$ cover
	which is the Galois closure of $\beta$. 
	Suppose further $\beta \circ \alpha$ is branched
	over a divisor $\mathscr D \subset \mathscr Y$ consisting of $n$ disjoint sections over $B$ so that
	$(\mathscr Y, \mathscr D)$ is a relative hyperbolic curve of genus $g$ over $B$.
	Assume the map $B \to \mathscr M_{g,n}$ induced by $\gamma$
	is dominant.

	We obtain a flat vector bundle $E := (\beta \circ \alpha)_* \mathscr
	O_{\mathscr X}$ on $\mathscr Y$ with regular singularities along
	$\mathscr D$.
	One can decompose $E = \bigoplus\limits_{S_d \text{-irreps }
	\rho} E_\rho^{\dim \rho}$ as a sum of flat bundles with irreducible
	monodromy. Let $F$ denote one such
	summand corresponding to the standard representation of dimension $d-1$. 
	The flat vector bundle $\beta_* \mathscr O_{\mathscr Z}$ decomposes as
	$\mathscr O_{\mathscr Y} \oplus F$.
	The dual $T$ of $F$ is known
	as the Tschirnhausen bundle.
	By \autoref{cor:stable},
	the restriction of $T$ to a general fiber of $\mathscr Y \to B$
	is semistable whenever $d -1 < 2 \sqrt{g+1}$.

	Previous results on stability of $T$ were established in 
	\cite[Theorem 1.5]{deopurkarP:vector-bundles-and-finite-covers}.
	If $h$ denotes the genus of $\mathscr Z \to B$, they proved the
	restriction of $T$ to a general fiber of $\mathscr Y \to B$ is
	semistable whenever $h \geq dg + d(d-1)^2 g$
	\cite[Remark 3.16]{deopurkarP:vector-bundles-and-finite-covers}. After this paper appeared on the arXiv, it was shown by Coskun, Larson, and Vogt \cite{coskun2022stability} that the restriction of $T$ to a general fiber of $\mathscr Y \to B$
	is semistable when $g\geq 1$, and stable if $g\geq 2$.
\end{example}

\subsection{Motivation}
 Our main motivation comes from the following question. Let $f: X\to Y$ be a map
 of algebraic varieties. What are the restrictions on the topology of $f$? Our
 \autoref{corollary:geometric-local-systems} places a very strong restriction on
 the topology of morphisms to an analytically very general curve $C$ of genus
 $g$. For example, it implies that if $f: X\to C$ is a proper morphism with smooth generic fiber and bad reduction at $n$ analytically very general points of $C$, then any non-isotrivial monodromy representation occurring in the cohomology of $X/C$ has dimension at least $2\sqrt{g+1}$.
 
We became interested in this question and its connection to isomonodromy while
trying to understand \cite{BHH:logarithmic}. In that paper Biswas, Heu, and
Hurtubise raise \autoref{question:bhh}, asking whether it is possible to isomonodromically deform irreducible flat vector bundles to achieve semistability, by analogy to Hilbert's 21st problem (also known as the Riemann-Hilbert problem). 

Hilbert's 21st problem, as answered by  Bolibruch \cite{bolibruch1995riemann} (correcting earlier work of Plemelj), poses the question of whether every monodromy representation can be realized by a Fuchsian system. Esnault and Viehweg generalize this question to higher genus in \cite{esnault1999semistable}: they ask when an irreducible representation can be realized as the monodromy of a flat vector bundle $(E, \nabla)$ with regular singularities at infinity, with $E$ semistable.  
 
 In Esnault-Viehweg's formulation, the complex structure on the underlying curve
 is fixed, and the residues of the differential equation at regular singular
 points are modified to achieve semistability. Flipping this around,
 Biswas-Heu-Hurtubise's analogue asks if semistability can be achieved by
 modifying the complex structure and fixing the residues. They claim that this
 is always possible in the logarithmic, parabolic, and irregular settings, in
 \cite{BHH:logarithmic, BHH:parabolic, BHH:irregular}. After discovering the
 Hodge-theoretic counterexample to these claims in
 \autoref{corollary:counterexample}, we proved \autoref{theorem:hn-constraints}
 as an attempt  (1) to understand to what extent Biswas, Heu, and Hurtubise's
 \autoref{question:bhh} has a positive answer, and (2) to apply the cases when there is a positive answer to the analysis of variations of Hodge structure on curves.

\subsection{Idea of proof}
\label{subsection:idea-of-proof}

To prove \autoref{theorem:very-general-VHS}, we first reduce to proving
\autoref{theorem:isomonodromic-deformation-CVHS}, using that discrete compact
spaces are finite.
We then prove \autoref{theorem:isomonodromic-deformation-CVHS} by showing that
any flat vector bundle satisfying the hypotheses of the theorem is forced to be (parabolically) semistable on an analytically general curve, whence the Hodge
filtration consists of a single piece by \autoref{corollary:unstable}.
The polarization then gives a definite Hermitian form preserved by the monodromy, and hence
the monodromy is unitary.
The key issue, which follows from \autoref{theorem:hn-constraints-parabolic},
is therefore to show that low rank flat vector bundles 
are parabolically semistable
on an analytically general curve.

To prove \autoref{theorem:hn-constraints-parabolic}, we assume we have a flat
vector bundle $(E, \nabla)$ on our hyperbolic curve $(C, D)$, and consider an
isomonodromic deformation to a nearby curve.
To this end, we use the deformation theory of this flat vector bundle with its Harder-Narasimhan filtration,
which is governed by a variant of the Atiyah bundle.
We show that if the Harder-Narasimhan filtration does not satisfy the conclusion of \autoref{theorem:hn-constraints-parabolic}, then there is a direction along which we can deform the curve so that
the filtration is destroyed.
Indeed, if the filtration persisted, deformation theory provides us with a map
from $T_C(-D)$ to a certain
parabolically semistable subquotient of $\mathrm{End}(E)$ which vanishes on $H^1$.
Taking the Serre duals gives a semistable coparabolic vector bundle of low rank
and large coparabolic slope
which is not generically globally generated.
In the end, we rule this out by a variant of Clifford's theorem for vector
bundles.

\subsection{Organization of the paper}
\label{subsection:organization}
In \autoref{section:parabolic}, we review background on parabolic bundles.
In \autoref{section:deformation-theory}, we give background on Atiyah bundles,
parabolic Atiyah bundles, and isomonodromic deformations. 
In \autoref{section:hodge-theoretic-preliminaries} we give background on complex variations of Hodge structures and their associated parabolic Higgs bundles. Experts can likely skip these three sections. 
In \autoref{section:counterexample}, we prove \autoref{theorem:counterexample} and \autoref{corollary:counterexample}, providing counterexamples to earlier published claims about semistability of isomonodromic deformations. 
In \autoref{section:hn-filtration} we prove the main results on isomonodromic
deformations, \autoref{theorem:hn-constraints} and \autoref{cor:stable} (and their generalizations \autoref{theorem:hn-constraints-parabolic} and \autoref{cor:stable-parabolic}). This is the technical heart of the paper. 
In \autoref{section:hodge-theoretic-results}, we prove the main consequences for
variations of Hodge structure, \autoref{theorem:very-general-VHS}, \autoref{corollary:geometric-local-systems}, \autoref{corollary:non-density-of-geometric-local-systems}, and \autoref{theorem:isomonodromic-deformation-CVHS}. 
Finally, \autoref{section:questions} lists some questions motivated by our results.

For readers unfamiliar with the theory of parabolic bundles, we suggest first considering the case of \autoref{theorem:very-general-VHS} where $\mathbb{V}$ is assumed to have unipotent monodromy at infinity. In this case one may replace parabolic stability with the usual notion of stability for vector bundles,  simplifying the proof;  in particular, one may use \autoref{cor:stable} in place of \autoref{cor:stable-parabolic}.
\subsection{Acknowledgments}
This material is based upon work supported by the Swedish Research Council under
grant no. 2016-06596 while the authors were in residence at Institut
Mittag-Leffler in Djursholm, Sweden during the fall of 2021. 
Landesman was supported by the National Science
Foundation under Award No. DMS-2102955;
Litt was supported by NSF grant DMS-2001196 and an NSERC grant, “Anabelian methods in arithmetic and algebraic geometry.” 
We would like to thank multiple anonymous referees for helpful comments.
We are grateful for useful conversations with 
Donu Arapura,
Indranil Biswas, 
Marco Boggi, 
Juliette Bruce, 
Anand Deopurkar,
H\'el\`ene Esnault, 
Najmuddin Fakhruddin,
Joe Harris,
Viktoria Heu,
Jacques Hurtubise, 
David Jensen, 
Jean Kieffer,
Matt Kerr,
Eric Larson, 
Haohao Liu,
Anand Patel,
Alexander Petrov,
Andrew Putman,
Will Sawin,
Ravi Vakil,
and Isabel Vogt.

\section{Background on parabolic bundles}
\label{section:parabolic}

We now review some basics on parabolic sheaves and parabolic bundles, primarily following the notation of
\cite[\S2.3]{BHH:parabolic}.
Some useful references include 
\cite[Part 3, \S1]{seshadri:fibres-vectoriels-sur-les-courbes-algebriques},
\cite[\S1 and \S3]{yokogawa:infinitesimal-deformation}, and
\cite[\S2]{bodenY:moduli-spaces-of-parabolic-higgs-bundles}.
Let $C$ be a smooth proper curve and $D\subset C$ a reduced divisor. Loosely speaking, a parabolic bundle is a vector bundle on $C$ together with an additional
filtration of the fibers over points of the given divisor, weighted by an increasing sequence of real numbers in $[0,1)$.

\subsection{Definition of parabolic bundles}
\label{subsubsection:definition-parabolic}

\begin{definition}
	\label{definition:parabolic-bundle}
	Let $E$ be a vector bundle over a curve $C$.
	Let $D = x_1 + \cdots + x_n \subset C$ denote a divisor, with the $x_i$ distinct.
	A {\em quasiparabolic structure} on $E$ over $D$ is a strictly decreasing
	filtration of subspaces
	\begin{align*}
		E_{x_j} = E_j^1 \supsetneq E_j^2 \supsetneq \cdots \supsetneq E_j^{n_j+1}
		= 0
	\end{align*}
	for each $1 \leq j \leq n$.
	A {\em parabolic structure} on $E$ over $D$ is a quasiparabolic
	structure together with
	$n$ sequences of real numbers
	\begin{align*}
		0 \leq \alpha^1_j < \alpha^2_j \cdots < \alpha^{n_j}_j < 1
	\end{align*}
	for $1 \leq j \leq n$.
	A {\em parabolic bundle} is a vector bundle with a parabolic structure.
	The collection $\{\alpha^i_j\}$ are called the {\em weights} of the
	parabolic bundle. We say the parabolic bundle has {\em rational weights} if
	all $\alpha^i_j$ are rational numbers.
	We often notate the data of a parabolic bundle simply as $E_\star$
	instead of $(E, \{E^i_j\}, \{\alpha^i_j\})$.
\end{definition}

\begin{definition}[Parabolic degree and slope]
	\label{definition:}
	Let $E_\star := (E, \{E^i_j\}, \{\alpha^i_j\})$. We define the {\em
	parabolic degree} of $E_\star$ as
\begin{align*}
	\on{par-deg}(E_\star) := \deg(E) + \sum_{j=1}^n \sum_{i=1}^{n_j}
	\alpha^i_j \dim(E_j^i/E_j^{i+1}).
\end{align*}
We define the {\em parabolic slope} as $\mu_\star(E_\star) :=
\on{par-deg}(E_\star)/\rk(E_\star)$.
\end{definition}

\subsection{Definition of parabolic sheaves}
\label{subsubsection:parabolic-sheaf}

In order to later apply Serre duality, we will not only need parabolic bundles, but also so-called
coparabolic bundles. Coparabolic bundles are examples of parabolic sheaves,
which we define next.
The definitions in this subsection follow
\cite{yokogawa:infinitesimal-deformation} and
\cite{bodenY:moduli-spaces-of-parabolic-higgs-bundles}.

\begin{definition}
	\label{definition:filtered-O_X-module}
	Let $X$ be a scheme and $D \subset X$ an effective Cartier divisor.
	Let $\mathbb R$ denote the category whose objects are real numbers with
	a single morphism $i^{\alpha,\beta}:\alpha\to\beta$ if $\alpha \geq \beta$ and no
	morphisms otherwise.
	Let $\mathscr M_X$ denote the category of sheaves of $\mathscr{O}_X$-modules.
A {\em $\mathbb R$-filtered $\mathscr O_X$-module} is a functor $E: \mathbb R \to \mathscr M_X$.
Notationally, we use $E_\star$ to denote the functor $E$, so $E_\alpha := E(\alpha)$.
We write $i_E^{\alpha,\beta} := E(i^{\alpha,\beta})$.
\end{definition}
\begin{example}
	\label{example:}
	Define $E[\alpha]_\star$ as the $\mathbb{R}$-filtered $\mathscr O_X$-module given by
$E[\alpha]_\beta := E_{\alpha + \beta}$
and $i^{\beta,\gamma}_{E[\alpha]} := i_E^{\beta + \alpha, \gamma + \alpha}$.
Let $i^{[\alpha,\beta]}_E : E[\alpha]_\star \to E[\beta]_\star$ denote the
natural transformation whose value on $\gamma$ is $i_E^{\alpha + \gamma, \beta +
\gamma}$.
For $f: E_\star \to F_\star$, we use $f[\alpha] : E[\alpha]_\star \to F[\alpha]_\star$ for
the natural induced map.
By abuse of notation, we will frequently write $E$ in place of $E_0$.
\end{example}
\begin{definition}
	\label{definition:parabolic-sheaf}
A {\em parabolic sheaf} on $X$ with respect to $D$ is
a $\mathbb R$-filtered $\mathscr O_X$-module $E_\star$ equipped with an isomorphism
\begin{align*}
	j_E : E_\star \otimes \mathscr O_X(-D) \overset{\sim}{\to} E[1]_\star
\end{align*}
 such that $i_E^{[1,0]} \circ j_E = \on{id}_{E_\star}
\otimes i_D: E_\star \otimes \mathscr O_X(-D) \to E_\star$, where $i_D: \mathscr{O}_X(-D)\to \mathscr{O}_X$ is the natural inclusion.

A natural transformation of parabolic $\mathscr O_X$-modules $f: E_\star \to
F_\star$ is a {\em parabolic morphism} if
\begin{equation}
	\label{equation:}
	\begin{tikzcd} 
		E_\star \otimes \mathscr O_X(-D) \ar {r}{f\otimes
		\on{id}} \ar {d}{j_E} & F_\star \otimes \mathscr
		O_X(-D) \ar {d}{j_F} \\
		E[1]_\star \ar {r}{f[1]} & F[1]_\star
\end{tikzcd}\end{equation}
commutes.

Let $\on{Hom}(E_\star, F_\star)$ denote the set of parabolic morphisms.
We let $\sheafhom(E_\star, F_\star)$ denote the sheaf of homomorphisms 
defined by taking $\sheafhom(E_\star, F_\star)(U):= \on{Hom}(E_\star|_U,
F_\star|_U).$
We also define a parabolic sheaf
$\sheafhom(E_\star, F_\star)_\star$ by taking
$\sheafhom(E_\star, F_\star)_\alpha := \sheafhom(E_\star, F[\alpha]_\star)$
and the $i_{\sheafhom(E_\star, F_\star)}^{\alpha,\beta}$ for $\alpha \geq \beta$
to be the natural maps
induced by $i_F^{[\alpha, \beta]}: F[\alpha]_\star \to F[\beta]_\star$.

We define $\on{End}(E_\star) := \on{Hom}(E_\star, E_\star)$
and use $\sheafend(E_\star) := \sheafhom(E_\star, E_\star)$.
\end{definition}

\begin{example}
	\label{example:parabolic}
	Any parabolic vector bundle $(E, \{E^i_j\}, \{\alpha^i_j\})$ defines a
	parabolic sheaf as follows. For $0 \leq \alpha < 1$,
	define
	\begin{align*}
	E_\alpha := \cap_{j=1}^n \ker(E \to
	E_{x_j}/E^{\beta(\alpha, j)}_j).
	\end{align*}
	where $\beta(\alpha, j)=\min(i: \alpha^i_j\geq \alpha),$ for $\alpha\leq \max_i(\alpha^i_j)$ and taking $\beta(\alpha, j)=n_j+1$ for $1>\alpha>\max_i(\alpha^i_j)$.
	For $n\in \mathbb{Z}$ set $E_{x+n} := E_x(-nD)$ and take $i^{\alpha,\beta}_E$ as the
	natural inclusions. We will refer to parabolic vector bundles and their associated parabolic sheaves interchangeably.
\end{example}

\begin{remark}
	\label{remark:end-definition}
	It follows from the definition of $\on{End}(E_\star)$ that 
	$\sheafend(E_\star) \subset \sheafend(E)$ is the coherent subsheaf
corresponding to those endomorphisms preserving the quasiparabolic filtration
$\{E^i_j\}$ at each point $x_j$ of $D$.
\end{remark}

\begin{example}
	\label{example:vector-bundle}
	Every vector bundle $E$ defines a parabolic bundle on $X$ with respect
	to $D$ by taking the quasiparabolic structure
	$E_{x_j} = E_j^1 \subset E_{j}^2 = 0$ with $\alpha_j^1 = 0$.
	In turn, this defines a parabolic sheaf by \autoref{example:parabolic}.
	We say such a parabolic bundle has trivial parabolic structure.
\end{example}
\begin{remark}
	\label{remark:map-vector-bundle-to-parabolic}
	If $V$ is a vector bundle on a scheme $X$ and $E_\star$ is a parabolic
	sheaf, we write a map $V \to E_\star$ to denote a map $V_\star \to
	E_\star$, where $V_\star$ is the parabolic sheaf corresponding to $V$ as
	in \autoref{example:vector-bundle}.
\end{remark}

In addition to parabolic vector bundles, we will also require coparabolic vector
bundles, in order to use Serre duality.
The essential idea is that while the parabolic sheaves associated to parabolic vector bundles are unchanged in intervals of
the form $(\alpha^i, \alpha^{i+1}]$, coparabolic vector bundles are unchanged in
intervals of the form $[\alpha^i, \alpha^{i+1})$.
I.e., parabolic vector bundles can be viewed as lower semicontinuous functors, taking
the discrete topology on the set of isomorphism classes of vector bundles, while coparabolic vector
bundles 
can be viewed as upper semicontinuous functors, see \cite[Figure
1]{bodenY:moduli-spaces-of-parabolic-higgs-bundles}.
\begin{definition}[Coparabolic vector bundles,
	\protect{\cite[Definition 2.3]{bodenY:moduli-spaces-of-parabolic-higgs-bundles}}]
	\label{definition:coparabolic-vector-bundle}
	Let $E_\star$ denote a parabolic vector bundle, viewed as a parabolic
	sheaf.
	The associated {\em coparabolic vector bundle}, denoted
	$\widehat{E}_\star$, is the parabolic sheaf defined by
		$$\widehat{E}_\alpha :=
			\on{colim}_{\beta > \alpha} E_{\beta} $$
	The colimit above is the union taken over the inclusions given by
	$i^{\alpha,\beta}_E$.
\end{definition}

\begin{definition}[Coparabolic degree and slope]
	\label{definition:}
	If $F_\star$ is a coparabolic bundle of the form $F_\star = \widehat{E}_\star$,
for $E_\star$ a parabolic vector bundle, then the {\em coparabolic degree} of
$F_\star$ is defined by $\on{copar-deg}(F_\star) := \on{par-deg}(E_\star)$ and the
coparabolic slope of $F_\star$ is defined by $\mu_\star(F_\star) := \mu_\star(E_\star)$.
\end{definition}
\begin{example}
	\label{example:vector-bundle-co}
	Given a vector bundle $V$, one can define an associated parabolic bundle
	$E_\star$ with the trivial parabolic structure, as in
	\autoref{example:vector-bundle}. We call $\widehat{E}_\star$
	as the coparabolic bundle associated with trivial coparabolic structure.
\end{example}

\subsection{Induced subbundles and quotient bundles}

\label{subsubsection:induced-subbundle}
Let $E_\star := (E, \{E^i_j\}, \{\alpha^i_j\})$ be a parabolic bundle.
	Any subbundle $F \subset E$ has an induced parabolic structure $F_\star$ as
	follows, see
\cite[Part 3, \S1.A, Definition 4]{seshadri:fibres-vectoriels-sur-les-courbes-algebriques}.
	The quasiparabolic structure over $x_j$ on $F$ is obtained
	from the filtration $$F_{x_j} = E^1_j \cap F_{x_j} \supset E^2_j \cap
	F_{x_j} \supset \cdots \supset E^{n_j+1}_j \cap F_{x_j} = 0$$
	by removing redundancies. For the weight associated to $F^i_j \subset
	F_{x_j}$ one takes $$\max_{\substack{k, 1 \leq k \leq n_j}} \{ \alpha^k_j : F^i_j = E^k_j \cap
	F_{x_j}\}.$$
	Similarly, any quotient $E \to Q$ with kernel $F$ has an induced parabolic structure
	given as follows.
	The quasiparabolic structure over $x_j$ on $Q$ is obtained from
	$$Q_{x_j} = (E^1_j+F_{x_j})/F_{x_j} \supset (E^2_j+F_{x_j})/F_{x_j}\supset\cdots\supset 
	(E^{n_j+1}_j+F_{x_j})/F_{x_j} = 0$$
	by removing redundancies. For the weight associated to a subspace $Q^i_j \subset
Q_{x_j}$ one takes $$\max_{\substack{k, 1 \leq k \leq n_j}} \{ \alpha^k_j : Q^i_j = (E^k_j + F_{x_j})
/F_{x_j}\}.$$

\subsection{Stability of parabolic and coparabolic bundles}
\label{subsubsection:stability-parabolic}

\begin{definition}[Parabolic stability]
	\label{definition:}
	A parabolic vector bundle $E_\star$ is {\em parabolically semistable}
	(respectively, parabolically stable) if $\mu_\star(F_\star) \leq
\mu_\star(E_\star)$ (respectively $\mu_\star(F_\star)< \mu_\star(E_\star)$) for all parabolic subbundles $F_\star \subset E_\star$ (with the
induced parabolic structure as described above).
\end{definition}
\begin{definition}[Coparabolic stability]
	\label{definition:coparabolic-stability}
	A coparabolic bundle $E_\star$ is {\em coparabolically semistable}
	if $E_\star$ is of the form $\widehat{F}_\star$ for $F$ a semistable
	parabolic bundle.
\end{definition}
\begin{remark}
	\label{remark:}
	Note that parabolic and coparabolic stability are both defined with
	respect to parabolic bundles.
\end{remark}
\begin{remark}
	\label{remark:coparabolic-stability}
	This remark will not be needed in what follows.
	It turns out that if $n > 0$,
	a coparabolic bundle $E_\star=\widehat{F}_\star$ is coparabolically semistable if for any
	injection from a parabolic vector bundle $G_\star \hookrightarrow E_\star$, $\mu_\star(G_\star) <
	\mu_\star(E_\star)$.
	That is, although the definition 
	only gives $\mu_\star(G_\star) \leq \mu_\star(E_\star)$, equality of
	slopes is not possible when $n > 0$.

	The reason for this is that any such map $G_\star \to E_\star$ also
	induces a map $G[-\varepsilon]_\star \to E_\star\to F_\star$ for some sufficiently
	small $\varepsilon > 0$. 
	We then have $\mu_\star(G_\star) < \mu_\star(G[-\varepsilon]_\star) \leq
	\mu_\star(F_\star)=\mu_\star(E_\star)$.
\end{remark}

\begin{lemma}
	\label{lemma:quotient-semistable}
	Suppose $E_\star$ is a parabolic bundle and $F_\star \subset E_\star$ is
	a parabolic subbundle with the induced subbundle structure. If $Q_\star
	= E_\star/F_\star$ then $\on{par-deg}(E_\star) = \on{par-deg}(Q_\star) +
	\on{par-deg}(F_\star)$.
	In particular, a parabolic bundle $E_\star$ is semistable (respectively,
	stable) if and only if for every quotient
	bundle $Q_\star$, $\mu_\star(E_\star) \leq \mu_\star(Q_\star)$, (respectively $<
	\mu_\star(Q_\star)$).
\end{lemma}
\begin{proof}
	The second statement follows from the first, because
	$\frac{\on{par-deg}(F_\star)}{\rk(F_\star)} <
	\frac{\on{par-deg}(E_\star)}{\rk(E_\star)}$
	is equivalent to 
	$\frac{\on{par-deg}(E_\star)- \on{par-deg}(F_\star)}{\rk(E_\star) -
	\rk(F_\star)} >
	\frac{\on{par-deg}(E_\star)}{\rk(E_\star)}$, and similarly where one
	replaces the inequalities with equalities.

	The first statement is stated in
	\cite[Part 3, \S1.A, p. 69, Remark
	3]{seshadri:fibres-vectoriels-sur-les-courbes-algebriques},
	together with the definition of exact sequence of parabolic bundles
	\cite[Part 3, \S1.A, p. 68]{seshadri:fibres-vectoriels-sur-les-courbes-algebriques}.
%
\end{proof}

\subsection{Harder-Narasimhan filtrations}
\label{subsubsection:harder-narasimhan-parabolic}
It is a standard fact that parabolic vector bundles have a Harder-Narasimhan
filtration, and its proof is similar to the construction of Harder-Narasimhan
filtrations of vector bundles, see 
\cite[Part 3, \S1.B, Theorem 8]{seshadri:fibres-vectoriels-sur-les-courbes-algebriques}.
	
\subsection{Serre duality}

If $E_\star$ is a parabolic sheaf on a scheme
$X$, we have $H^0(C, E_\star) := \on{Hom}(\mathscr O_X, E_\star) =
\on{Hom}(\mathscr O_X, E_0) = \Gamma(X, E)$, and one can define the
higher cohomology groups by taking the corresponding right derived functor, as in 
\cite[p. 130]{yokogawa:infinitesimal-deformation}, where, more generally,
$\on{Ext}^i(E_\star, F_\star)$ is defined for parabolic $\mathscr O_X$ modules.
In general we have $H^i(X, E_\star)=H^i(X, E_0)$.

To state Serre duality, we need the notion of parabolic tensor product and
duality.
\begin{definition}
	\label{definition:dualization}
	For $F_\star$ a parabolic sheaf, define $F^\vee_\star :=
	\sheafhom(F_\star, \mathscr O_X)_\star$.
\end{definition}
The following gives a useful alternate description of parabolic dualization.
\begin{lemma}[Parabolic dualization, cf. \protect{\cite[(3.1)]{yokogawa:infinitesimal-deformation}}]
	\label{lemma:parabolic-dual-formula}
	If $F_\star$ arises from a parabolic vector bundle, we have
	$F^\vee_\alpha \simeq (\widehat{F}_{-\alpha})^\vee (-D)$.
\end{lemma}

\begin{definition}[Parabolic tensor product, cf. \protect{\cite[Example
	3.2]{yokogawa:infinitesimal-deformation}}]
	\label{definition:parabolic-tensor}
	Let $\tau: X - D \to X$ denote the inclusion.
	Suppose $E_\star, F_\star$ are both parabolic modules so that each $F_\alpha$
	and $E_\alpha$ is locally free,
	define $(E_\star \otimes F_\star)_\alpha := \sum_{\alpha_1 + \alpha_2 =
	\alpha} E_{\alpha_1} \otimes F_{\alpha_2}$, viewed as a subbundle of
	$\tau_* \tau^*(E \otimes F)$. We take $i^{\alpha,\beta}_{E_\star \otimes F_\star}$ to be the natural inclusion map.
\end{definition}
\begin{remark}
	\label{remark:}
	In \cite[p. 136-137]{yokogawa:infinitesimal-deformation}, the tensor
	product of two arbitrary parabolic sheaves is defined, but the
	definition is more difficult to state, and we will not require this
	greater generality.
\end{remark}

The next lemma states that parabolic tensor products and duals interact in the usual way with degree.
We omit the proof, which is a matter of unwinding definitions.
\begin{lemma}
	\label{lemma:degree-in-duals}
	For $E_\star$ and $F_\star$ two parabolic bundles, $\deg E_\star \otimes
	F_\star= \deg E_\star \rk F_\star + \deg F_\star \rk E_\star$ and $\deg
	E_\star^\vee = - \deg E_\star$.
\end{lemma}

\begin{proposition}[Serre duality]
	\label{proposition:serre-duality}
	Suppose $X$ is a smooth projective $n$-dimensional variety over an
	algebraically closed field $k$ and let $\omega_X$ denote the dualizing
	sheaf on $X$. For all parabolic vector bundles $E_\star$, we have a canonical isomorphism
\begin{align*}
	H^i(X, E_\star) \simeq H^{n-i}(X, \widehat{E_\star^\vee} \otimes
	\omega_X(D))^\vee.
\end{align*}
\end{proposition}
\begin{proof}
	The version in \cite[Proposition
	3.7]{yokogawa:infinitesimal-deformation} states
	$\on{Ext}^i_X(E_\star, F_\star \otimes \omega_X(D)) \simeq
	\on{Ext}^{n-i}_X(F_\star, \widehat{E}_\star)^\vee.$
	Using \cite[Lemma 3.6]{yokogawa:infinitesimal-deformation}, and taking
	$E_\star = \mathscr O_X$, we find
	$H^i(X, F_\star \otimes \omega_X(D)) \simeq H^{n-i}( F_\star^\vee
		\otimes
	\widehat{\mathscr O_X})^\vee \simeq H^{n-i}(X, \widehat{F_\star^\vee})^\vee$.
	Now, taking $E_\star$ as in the statement to be 
	$F_\star \otimes \omega_X(D)$, we find $F_\star^\vee = E_\star^\vee
	\otimes \omega_X(D)$ and so
$H^i(X, E_\star) \simeq H^{n-i}(X, \reallywidehat{E_\star^\vee \otimes
\omega_X(D)})^\vee \simeq 
H^{n-i}(X, \widehat{E_\star^\vee} \otimes
\omega_X(D))^\vee.$
\end{proof}

\section{Background on Atiyah bundles and isomonodromic deformations}
\label{section:deformation-theory}

\subsection{The Atiyah bundle of a filtered vector bundle}
\label{subsection:atiyah-bundle}

We begin by defining the Atiyah bundle.
Let $C$ be a smooth projective curve.
Following 
\cite[16.8.1]{EGAIV.4},
for $E$ a vector bundle on $C$, define $\on{Diff}^1(E,E)$ as follows: for
$U\subset C$ open, $\on{Diff}^1({E}, {E})(U)$ is the set of
$\mathbb{C}$-linear endomorphisms $\tau$ of $E(U)$, such that for each $f\in
\mathscr{O}_C(U), v\in E(U)$, we have that $$\tau_f: v\mapsto
\tau(fv)-f\tau(v)$$ is $\mathscr{O}_C$-linear.
Here $\tau_f$ measures the failure of $\tau$ to be $\mathscr{O}_C$-linear, in
that $\tau_f$ is zero for all $f$ if and only if $\tau$ is $\mathscr{O}_C$-linear.

\begin{definition}[The Atiyah bundle, see
	\protect{\cite[p. 5]{BHH:very-stable}}] 
	\label{definition:atiyah}	
	Let $E$ be a
vector bundle on a curve $C$. 
Define the Atiyah bundle $$\on{At}_C(E)\subset \on{Diff}^1({E}, {E})$$ as the
subsheaf with sections on an open set $U \subset C$ given as follows.
Let $\on{At}_C(E)(U)$ consist of those $\mathbb{C}$-endomorphisms $\tau \in
\on{Diff}^1({E}, {E})(U)$ such that for each $f\in \mathscr{O}_C(U)$, the endomorphism of ${E}$ defined by $$\tau_f: v\mapsto \tau(fv)-f\tau(v)$$ is multiplication by a section $\delta_\tau(f)\in \mathscr{O}_C(U)$. 
\end{definition}

One can also construct Atiyah bundles associated to filtered bundles.
\begin{definition}[Atiyah bundle of a filtered vector bundle]
	\label{definition:atiyah-filtered}
Let $P^\bullet := (0 = P^0 \subset P^1
\subset \cdots \subset P^m = E)$ be a filtration on $E$.
We define $$\on{At}_C(E, P^\bullet)\subset \on{At}_C(E)$$ to be the subsheaf
consisting of those endomorphisms that preserve $P^\bullet$. 
\end{definition}

\begin{remark}
	\label{remark:atiyah-exact-sequence}
	From the definition, (see also \cite[(2.7)]{BHH:logarithmic}), there is a 
short exact sequence 
\begin{equation} \label{equation:non-logarithmic-atiyah-sequence} 0\to
	\sheafend(E, P^\bullet)\overset{\iota}{\to} \on{At}_C(E,
P^\bullet)\overset{\delta}{\to} T_C\to 0,\end{equation} where
$\sheafend(E, P^\bullet)\subset \sheafend(E)$ is the subsheaf of $\mathscr{O}_C$-linear endomorphisms preserving $P^\bullet$, $\iota$ is the evident inclusion, 
and $\delta$ sends a differential operator $\tau$ to the derivation
$$\delta_\tau: f\mapsto \delta_\tau(f)$$ defined in 
\autoref{definition:atiyah}.
\end{remark}

\begin{remark}
    There is an alternate, perhaps more geometric, description of $\on{At}_C(E,
    P^\bullet)$. Namely, the filtration $P^\bullet$ gives a restriction of
    the structure group of ${E}$ to a parabolic subgroup ${P}\subset
    \on{GL}_n$, and hence gives rise to a natural ${P}$-torsor $p:\Pi\to C$
    over $C$, which is a subscheme of the frame bundle of ${E}$ (i.e.~it
    consists of those frames which are compatible with $P^\bullet$). The tangent
    exact sequence $$0\to T_{\Pi/C}\to T_\Pi\to p^*T_C\to 0$$ naturally admits a ${P}$-linearization (for the ${P}$-action on $\Pi$) 
    and hence descends to a short exact sequence on
    $C$, which is precisely \eqref{equation:non-logarithmic-atiyah-sequence}.
\end{remark}

Next, we introduce Atiyah bundles with respect to divisors.
\begin{definition}[Atiyah bundle of a filtered vector bundle with respect to a divisor]
	\label{definition:atiyah-bundle-divisor}
Let $D\subset C$ be a reduced effective divisor.
The Atiyah bundle $\on{At}_{(C,D)}(E,P^\bullet)$ is defined as the preimage 
$$\on{At}_{(C,D)}(E,P^\bullet):=\delta^{-1}(T_C(-D)),$$ where $\delta$ is the map appearing in Sequence \eqref{equation:non-logarithmic-atiyah-sequence} and where $T_C(-D)\hookrightarrow T_C$ is the natural inclusion. 

If $P^\bullet = (0 = P^0 \subset P^1 = E)$ is the trivial filtration, we omit it from the notation, e.g.~we will use the notation $\on{At}_{(C,D)}(E)$ 
in place of $\on{At}_{(C,D)}(E, 0 \subset E)$
when convenient.
\end{definition}
The
following alternate viewpoint on connections will be useful.
\begin{proposition}\label{proposition:atiyah-splittings}
Suppose $E$ is a vector bundle on $C$ and $D \subset C$ a reduced
effective divisor.
    There is a natural bijection between splittings of the Atiyah exact sequence
    \begin{equation}
	    \label{equation:Atiyah-exact-sequence-normal}
	    0\to \sheafend(E, P^\bullet)\to \on{At}_{(C,D)}(E, P^\bullet)\to T_C(-D)\to
    0.
    \end{equation}
	and flat connections on $E$ with
    regular singularities along $D$ and preserving $P^\bullet$, given by
    adjointness. That is, given a connection $$\nabla: E\to E\otimes
    \Omega^1_C(\log D)$$ preserving $P^\bullet$, we may by adjointness view
    $\nabla$ as a map $q^\nabla: T_C(-D)\to
    \sheafend_{\mathbb{C}}(E)$. 
    This map factors through $\on{At}_{(C,D)}(E)$ and
    yields a splitting of \eqref{equation:Atiyah-exact-sequence-normal}.
    Moreover, this correspondence between flat connections and splittings is bijective.
\end{proposition}
\begin{proof}
    This is a matter of unwinding definitions.
\end{proof}
We will pass freely between $\nabla$ and $q^\nabla$ and refer to each as a connection.
\subsection{The parabolic Atiyah bundle}
\label{subsection:parabolic-atiyah}

We next recall the generalization of the Atiyah bundle to the parabolic setting.

	Let $C$ be a curve with reduced divisor $D=x_1+\cdots+x_n$ and let $E_\star$ be a
	parabolic vector bundle on $(C,D)$.
	It will useful to recall the explicit description of $\on{End}(E_\star)$ given in
	\autoref{remark:end-definition}.
	To define the parabolic Atiyah bundle, we also need the following
	definition.

\begin{definition}
	\label{definition:}
	Let
$\on{End}_j(E_\star)$ denote the image of the natural inclusion
$\sheafend(E_\star)_{x_j} \to \sheafend(E)_{x_j}$, viewed as a subspace of $\sheafend(E)_{x_j}$.
\end{definition}

\begin{remark}
	\label{remark:}
	We can write $\sheafend(E)/\sheafend(E_\star)$ as a direct sum of skyscraper
sheaves supported on the $x_j$.
More precisely, for $\on{End}_j(E_\star)$ the image of
$\sheafend(E_\star)_{x_j} \to \sheafend(E)_{x_j}$ in the fiber over $x_j$,
there is an exact sequence
\begin{equation}
	\label{equation:}
	\begin{tikzcd}
		0 \ar {r} & \sheafend(E_\star) \ar {r} & \sheafend(E) \ar {r} &
		\oplus_{j= 1}^n \sheafend(E)_{x_j}/\on{End}_j(E_\star) \ar {r} &
		0.
\end{tikzcd}\end{equation}
\end{remark}

\subsubsection{The residue homomorphism}
\label{subsubsection:residue}
To define the Atiyah bundle associated to $E_\star$, we essentially want to
carve it out from the Atiyah bundle by requiring our differential operators to preserve the quasiparabolic
filtration over each $x_j$. For this, we need a certain homomorphism
${\phi}_j: \on{At}_{(C,D)}(E)_{x_j}
\to \sheafend(E)_{x_j}$,
referred to as a {\em residue homomorphism}.
The residue homomorphism is
defined in, for
example, \cite[(2.9)]{BHH:parabolic}, and we also recall it now.

Given any smooth proper curve $C$ and reduced divisor $D \subset C$ with $x_j$ in the support of $D$,
the residue homomorphism at $x_j$ is a map ${\phi}_j: \on{At}_{(C,D)}(E)_{x_j}
\to \sheafend(E)_{x_j}$ defined as follows.
There is a commutative diagram of vector spaces
\begin{equation}
	\label{equation:}
	\begin{tikzcd}
		0 \ar {r} & \sheafend(E)_{x_j} \ar {r}{\nu_j} \ar {d} &
		\on{At}_{(C,D)}(E)_{x_j} \ar {r} \ar {d}{\alpha_j} & T_C(-D)_{x_j}
		\ar {r} \ar {d}{\beta_j} & 0 \\
		0 \ar {r} & \sheafend(E)_{x_j} \ar {r}{\mu_j} & \on{At}_C(E)_{x_j} \ar
		{r}{\gamma_j} & (T_C)_{x_j} \ar {r} & 0.
\end{tikzcd}\end{equation}
The map $\beta_j$ is induced from the natural inclusion of invertible sheaves, and
hence vanishes.
Therefore, $\gamma_j \circ \alpha_j = 0$, which means $\alpha_j$ factors through
$\sheafend(E)_{x_j}$. This produces the desired map
${\phi}_j: \on{At}_{(C,D)}(E)_{x_j} \to \sheafend(E)_{x_j}$,
satisfying the property that $\mu_j\circ \phi_j  = \alpha_j$.
By definition of $\alpha_j$, the restriction of $\alpha_j$ to $\sheafend(E)_{x_j}$ is the identity.
That is, $\phi_j\circ \nu_j   = \mathrm{id}$.
This gives us a splitting 
$\on{At}_{(C,D)}(E)_{x_j} \simeq  \sheafend(E)_{x_j} \oplus T_C(-D)_{x_j}$.

Now, let $z$ denote a uniformizer of the local ring of $C$ 
at $x_j$ and let $z \frac{\partial
}{\partial z}$ denote the corresponding section of $T_C(-D)$ at $x_j$.
Observe that this is independent of the choice of $z$.
Given $q^\nabla: T_C(-D) \to \on{At}_{(C,D)}(E)$ a connection with regular
singularities,
let $q^\nabla_{x_j}(z \frac{\partial }{\partial z}) \in \on{At}_{(C,D)}(E) \simeq  \sheafend(E)_{x_j} \oplus T_C(-D)_{x_j}$ denote the image of $z
\frac{\partial }{\partial z}$
under $q^\nabla$ restricted to the fiber $x_j$.
Define the {\em residue} $\on{Res}(\nabla)(x_j) \in \sheafend(E)_{x_j}$ as the
projection of 
$q^\nabla_{x_j}(z \frac{\partial }{\partial z}) \in \sheafend(E)_{x_j} \oplus T_C(-D)_{x_j}$ to $\sheafend(E)_{x_j}$.

Having defined the map $\phi_j$ above, we next define the Atiyah bundle
associated to a parabolic bundle.
\begin{definition}[Atiyah bundle of a parabolic bundle]
	\label{definition:parabolic-atiyah-bundle}
Let $D\subset C$ be a reduced effective divisor 
and let $E_\star = (E, \{E^i_j\}, \{\alpha^i_j\})$ be a parabolic bundle on $(C,
D)$.
For $\phi_j$ the residue homomorphism, as defined above in
\autoref{subsubsection:residue},
let $\widehat{\phi}_j: \on{At}_{(C, D)}(E) \to \sheafend(E)_{x_j}/\on{End}_j(E_\star)$ denote the composition
\begin{align*}
	\on{At}_{(C, D)}(E) \to \on{At}_{(C,D)}(E)_{x_j} \xrightarrow{\phi_j}
	\sheafend(E)_{x_j} \to \sheafend(E)_{x_j}/\on{End}_j(E_\star).
\end{align*}
Define $\on{At}_{(C, D)}(E_\star)$ as the coherent subsheaf of $\on{At}_{(C, D)}(E)$ given by
\begin{align*}
	\on{At}_{(C,D)}(E_\star) := \ker\left( \on{At}_{(C,D)}(E)
	\xrightarrow{\oplus_{j=1}^n \widehat{\phi}_j } \oplus_{j=1}^n
\sheafend(E)_{x_j}/\on{End}_j(E_\star) \right).
\end{align*}
Similarly, for $P^\bullet := (0 = P^0 \subset P^1
\subset \cdots \subset P^m = E)$ a filtration on $E$,
we let 
$\on{At}_{(C,D)}(E_\star, P^\bullet) \subset \on{At}_{(C, D)}(E_\star)$
denote the coherent subsheaf
consisting of those endomorphisms that preserve $P^\bullet$. 
\end{definition}

\begin{remark}
	\label{remark:}
Using \autoref{definition:parabolic-atiyah-bundle}
	and \eqref{equation:non-logarithmic-atiyah-sequence},
	we find that $\on{At}_{(C,D)}(E_\star, P^\bullet)$ fits into a short exact sequence
\begin{equation}\label{equation:Atiyah-exact-sequence}
	0\to \sheafend(E_\star, P^\bullet)\to \on{At}_{(C,D)}(E_\star, P^\bullet)\to T_C(-D)\to
    0.
\end{equation}

By comparing \eqref{equation:Atiyah-exact-sequence} for a filtration $P^\bullet$
and the trivial filtration, we obtain 
the short exact sequence
\begin{equation}
	\label{equation:quotient-atiyah-bundles}
	\begin{tikzcd}
		0 \ar {r} &  \on{At}_{(C,D)}(E_\star, P^\bullet) \ar {r} &
		\on{At}_{(C,D)}(E_\star) \ar {r}
		& \sheafend(E_\star)/\sheafend(E_\star, P^\bullet)\ar {r} & 0,
\end{tikzcd}\end{equation}
where $\sheafend(E_\star, P^\bullet) \subset \sheafend(E_\star)$ is the coherent subsheaf
consisting of those endomorphisms preserving the filtration $P^\bullet$.
\end{remark}

\subsection{Parabolic structure and connections}
\label{subsection:atiyah-properties}
We continue  our review of Atiyah bundles by recalling the parabolic bundle
associated to a connection, and a constraint on irreducibility of connections.

\begin{definition}
	\label{definition:associated-parabolic}
	Let $q^\nabla: T_C(-D) \to \on{At}_{(C,D)}(E)$ be a connection on $E$
	with regular singularities along $D$.
	Let $\on{Res}(\nabla)(x_j) \in \sheafend(E)_{x_j}$ denote the residue
	of $\nabla$ at $x_j$, described in \autoref{subsubsection:residue}.
	Suppose the eigenvalues of $\on{Res}(\nabla)(x_j)$ are $\eta_j^1, \ldots,
	\eta^{s_j}_j$. Let $$\lambda^i_j:=\on{Re}(\eta^i_j)-\lfloor \on{Re}(\eta^i_j)\rfloor \in[0,1)$$ denote the fractional part of the real
	part of $\eta^i_j$. Reorder the $\lambda^i_j$ and remove repetitions so that
	\begin{align*}
		0 \leq \lambda^1_j < \lambda^2_j < \cdots < \lambda_j^{n_j} < 1.
	\end{align*}
	Define $E^i_j \subset E_{x_j}$ as the sum of all generalized eigenspaces
	of $\on{Res}(\nabla)(x_j)$
	such that the fractional part of the real part of the associated
	eigenvalue is $\geq \lambda^i_j$.
	The data $(E, \{E^i_j\}, \{\lambda^i_j\})$ is the {\em parabolic bundle
	associated to the connection} $\nabla$.
	By \cite[Lemma 4.1]{BHH:parabolic}, the connection $q^\nabla: T_C(-D) \to
	\on{At}_{(C, D)}(E)$
	factors through $\on{At}_{(C, D)}(E_\star) \subset \on{At}_{(C, D)}(E)$
	and we denote the induced map by $q^\nabla : T_C(-D) \to \on{At}_{(C,
	D)}(E_\star)$ as well.
	If the real part of each $\eta^i_j$ lies in $[0,1)$, then $(E_\star,
	\nabla)$ is
	called the {\em Deligne canonical extension} of $(E|_{C \setminus D},
	\nabla_{C \setminus D})$.
\end{definition}

\begin{proposition}
	\label{proposition:irreduciblity-splitting-condition}
    Let $(E, \nabla: {E}\to {E}\otimes \Omega^1_C(\log D))$ be a flat
    vector bundle on $C$ with regular singularities along $D$, 
    and let $E_\star$ denote the parabolic bundle associated to $\nabla$, as in
    \autoref{definition:associated-parabolic}.
    Suppose the monodromy representation $\rho$ 
    associated to $(E, \nabla)|_{C\setminus D}$ via the Riemann-Hilbert
    correspondence is irreducible. Let $q^\nabla:
    T_C(-D)\to \on{At}_{(C,D)}(E_\star)$ be the corresponding splitting of the Atiyah
    exact sequence via \autoref{proposition:atiyah-splittings} and
    \autoref{definition:associated-parabolic}. Then for any
    nontrivial filtration $P^\bullet$ of $E$, the composition
    $$T_C(-D)\overset{q^\nabla}{\to} \on{At}_{(C,D)}(E_\star){\to}
    \on{At}_{(C,D)}(E_\star)/\on{At}_{(C,D)}(E_\star, P^\bullet)\simeq
    \sheafend(E_\star)/\sheafend(E_\star,P^\bullet)$$ is nonzero.
\end{proposition}
\begin{proof}
Assume not. Then $q^\nabla$ has image in $\on{At}_{(C,D)}(E_\star, P^\bullet)$, and
hence yields a splitting of \eqref{equation:Atiyah-exact-sequence}. 
Using \autoref{proposition:atiyah-splittings},
the corresponding connection with regular singularities on $E$ preserves $P^1$ and hence yields a flat connection on $P^1$ with regular singularities along
$D$, whose monodromy is a sub-representation of the monodromy representation
$\rho$ associated to $({E}, \nabla)|_{C\setminus D}$. 
But this contradicts the
assumption that $\rho$ is irreducible. (See \cite[Proof of Proposition
5.3]{BHH:logarithmic} for a similar argument in the non-parabolic setting.)
\end{proof}
\subsection{Isomonodromic deformations}
\label{subsection:isomonodromic-deformations}

We next recall the notion of  isomonodromic deformation.
We also define the notion of an ``isomonodromic deformation to an analytically
general nearby curve'' which appears in many of our main results.

\begin{notation}
	\label{notation:curve-and-sections}
	Let $\mathscr{C}, \Delta$ be complex manifolds, and let $\pi: \mathscr{C}\to \Delta$ 
be a proper holomorphic submersion with connected fibers of relative dimension
one, with $\Delta$ contractible.  Let $$s_1, \cdots, s_n: \Delta \to
\mathscr{C}$$ be disjoint sections to $\pi$, and let $\mathscr{D}$ be the union
$$\mathscr{D}:=\bigcup_i \on{im}(s_i).$$ Given a point $0\in \Delta$, let
$(C, D) := (\pi^{-1}(0), \pi^{-1}(0) \cap \mathscr D)$ and further assume
$(C, D)$ is hyperbolic.
\end{notation}

\begin{lemma}
	\label{lemma:iso-extension}
	With notation as in \autoref{notation:curve-and-sections}, let $$(E, \nabla: E \to E\otimes \Omega^1_C(\log
	D))$$ be a flat vector bundle on $C$ with regular singularities along
	$D$. Such a logarithmic flat vector bundle extends canonically to a
	logarithmic flat vector bundle $$(\mathscr{E}, \widetilde{\nabla}:
	\mathscr{E}\to \mathscr{E}\otimes \Omega^1_{\mathscr{C}}(\log
\mathscr{D}))$$ on $\mathscr{C}$ with regular singularities along $\mathscr{D}$.
\end{lemma}
\begin{proof}
	This follows from Deligne's work on differential equations
with regular singularities \cite{deligne:regular-singular} and is explained in
\cite[Theorem 3.4]{Heu:universal-isomonodromic}, following work of Malgrange
\cite{malgrange:I, malgrange:II}.
In particular, \cite[Theoreme 2.1]{malgrange:I} explains the case where $C$ has
genus zero, and the general case is similar. 
We now recapitulate the proof.

The restriction
of $({E}, \nabla)$ to $C\setminus D$ is a flat vector bundle and hence
gives rise to a locally constant sheaf of $\mathbb{C}$-vector spaces
$$\mathbb{V}:=\ker(\nabla)$$ on $C\setminus D$. As $\Delta$ is contractible, the
inclusion $$C\setminus D\hookrightarrow \mathscr{C}\setminus \mathscr{D}$$ is a
homotopy equivalence; thus $\mathbb{V}$ extends uniquely (up to canonical
isomorphism) to a local system $\widetilde{\mathbb{V}}$ on $\mathscr{C}\setminus
\mathscr{D}$. 

Manin's local results on extending flat vector bundles across divisors
\cite[Proposition 5.4]{deligne:regular-singular} imply that there is a
canonical extension, which is unique, up to unique isomorphism,
$$(\widetilde{\nabla}:
\mathscr{E}\to \mathscr{E}\otimes \Omega^1_{\mathscr{C}}(\log \mathscr{D}))$$ of
$$(\widetilde{\mathbb V}\otimes_{\mathbb{C}}\mathscr{O}_{\mathscr{C}\setminus
\mathscr{D}}, \on{id}\otimes d)$$ to a flat vector bundle on $\mathscr{C}$ with
regular singularities along $\mathscr{D}$, equipped with an isomorphism
$({\mathscr{E}}, \widetilde{\nabla})|_C\simeq (E,\nabla).$
\end{proof}

Using the above, we are ready to define isomonodromic deformations.
\begin{definition}[Isomonodromic Deformation]
	\label{definition:isomonodromy}
	With notation as in \autoref{notation:curve-and-sections}, let $D=x_1+ \cdots+ x_n$, so that $(C, D)$ is an $n$-pointed hyperbolic curve of genus $g$. Let $(E, \nabla)$ be a flat vector bundle on $C$ with regular singularities at the $x_i$.
We call the extension $(\mathscr{E}, \widetilde{\nabla})$ as in
\autoref{lemma:iso-extension} \emph{the isomonodromic deformation} of $({E}, \nabla)$.
If $\Delta=\mathscr{T}_{g,n}$ is the universal cover of
of the analytic stack $\mathscr{M}_{g,n}$, and $\mathscr{C}\to \Delta$ is the universal curve,
we call the isomonodromic deformation over such $\Delta$  \emph{the universal
isomonodromic deformation}.
\end{definition}
\begin{definition}
	\label{definition:nearby}
	With notation as in \autoref{definition:isomonodromy}, let $\Delta$ be
	the universal cover of $\mathscr{M}_{g,n}$.
We use \emph{an isomonodromic deformation to a nearby curve} to denote the
restriction of $(\mathscr{E}, \widetilde{\nabla})$ to any fiber of $\mathscr{C}
\to \Delta$.
We use 
\emph{an isomonodromic deformation to an analytically general nearby curve}
to denote the restriction of $(\mathscr{E}, \widetilde{\nabla})$ to a
general fiber of $\mathscr{C}
\to \Delta$, i.e., a fiber in the complement of a nowhere dense closed analytic
subset.
\end{definition}

\begin{remark}
	\label{remark:}
	The construction of \autoref{lemma:iso-extension} is functorial: given a commutative diagram 
$$\xymatrix{
D \ar@{^(->}[d] \ar@{^(->}[r] & \mathscr{D} \ar[r]\ar@{^(->}[d] & \mathscr{D}'\ar@{^(->}[d] \\
C \ar@{^(->}[r] \ar[d]& \mathscr{C} \ar[r] \ar[d]^\pi & \mathscr{C}' \ar[d]^{\pi'}\\
0 \ar@{^(->}[r] & \Delta \ar[r] & \Delta'
}$$
and a flat vector bundle $({E}, \nabla)$ on $C$ with regular
singularities along $D$, the isomonodromic deformation over $\Delta'$ pulls back
to the isomonodromic deformation over $\Delta$.
\end{remark}
\begin{example}[Families of families, essentially in \cite{doran:isomonodromic}]
	With notation as in \autoref{notation:curve-and-sections}, suppose $\mathscr{D}=\emptyset$, and let $\tilde h:
\mathscr{X}\to \mathscr{C}$ be a proper holomorphic submersion. Let
$X=h^{-1}(C)$, and let $h=\tilde h|_X$. Then for each $i\geq 0,$ $R^i\tilde
h_*\Omega^\bullet_{dR, \mathscr{X}/\mathscr{C}}$ with its Gauss-Manin connection
is the isomonodromic deformation of $R^ih_*\Omega^\bullet_{dR, X/C}$ with its
Gauss-Manin connection.
\end{example}

\begin{remark}
	\label{remark:parabolic-structure-on-isomonodromic-deformation}
	Using residues of the connection, we were able to associate to $(E,
	\nabla)$ a certain parabolic bundle $E_\star$ in
	\autoref{definition:associated-parabolic}.
	This induces the structure of a relative parabolic bundle on the
	isomonodromic deformation $(\mathscr E, \widetilde{\nabla})$ of $(E,
	\nabla)$, which we denote $\mathscr E_\star$, as explained in
	\cite[\S4.3]{BHH:parabolic}.
	Let $\mathscr C_t$ denote the fiber of $\mathscr C \to \Delta$
	over the point $t \in \Delta$ defining the isomonodromic deformation.
	By \cite[Lemma 4.2]{BHH:parabolic}, the parabolic weights and the
	associated dimensions of the graded parts of the quasiparabolic
	structure corresponding to those weights on $\mathscr E|_{\mathscr C_t}$ are independent
	of the point $t \in \Delta$. Indeed, the parabolic structure on $\mathscr{E}|_{\mathscr{C}_t}$ is exactly the one associated to the natural connection on $\mathscr{E}|_{\mathscr{C}_t}$ obtained by restricting $\widetilde{\nabla}$.
\end{remark}
We next introduce a notion of refinement of a parabolic structure. Loosely speaking, $F_\star$ is refined by $E_\star$ if $E_\star, F_\star$ have the same underlying vector bundle, and the parabolic structure on $F_\star$ is obtained by forgetting part of the parabolic structure on $E_\star$. The most
important case of this which we will use is that the bundle $E_\star$ is a refinement of $E_0$, with its trivial parabolic structure.

\begin{definition}
	\label{definition:refinements}
	Given a parabolic bundle $E_\star$ on $(C,D)$, with $D=x_1+ \cdots+ x_n$ a reduced effective divisor, we may write $$E_\star= (E, \{E^i_j\}_{1 \leq j \leq n,
	1\leq i \leq n_j+1}, \{\alpha^i_j\}_{1 \leq j \leq n,
1\leq i \leq n_j}).$$
Let $D'\subset D$ be an effective divisor, i.e.~there exists $I\subset \{1, \cdots, n\}$ so that $D'=\sum_{j\in I} x_j$. 

We say a parabolic bundle $F_\star$ on $(C, D')$ is {\em refined by} $E_\star$ if for each $j\in I$ there is a subset $\{i_1, \ldots, i_{m_j}\} \subset
\{2, \ldots, n_j\}$ such that there is an isomorphism of parabolic bundles
$$F_\star \simeq  
(E, \{E^i_j\}_{j \in I, i \in \{1, i_1, \ldots, i_{m_j}\} },
\{\alpha^i_j\}_{j \in I, i \in \{1, i_1, \ldots, i_{m_j}\} }).$$
\end{definition}

\begin{example}
	\label{example:refinement}
	Consider the case that our base curve is $X$, $D = x_1 + x_2$, and
	the parabolic bundle $E_\star$ has underlying bundle $E = \mathscr O_X
	\oplus \mathscr O_X$.
	Suppose the parabolic structure of $E_\star$ at $x_1$ is given by
	$E_{x_1} = E_1^1 \subsetneq E_1^2 \subsetneq 0$ with weights $\alpha_1^1
	= 1/3, \alpha_1^2 = 2/3$ and parabolic structure at $x_2$ given by $E_{x_2} = E_2^1 \subsetneq 0$ with weight
	$\alpha_2^1 = 1/2$.

	There are $6 = 3 \cdot 2$ parabolic bundles $F_\star$ refined by $E_\star$.
	Since $F_\star$ is a parabolic bundle with respect to a divisor $D'
	\subset D$, we can take $D' = \emptyset, x_1, x_2,$ or $x_1+
	x_2$. In the first case, $F_\star$ has trivial parabolic structure and
	corresponds to the vector bundle $E_0$. In the case $D = x_1$, either
	the parabolic structure at $x_1$ is the same as that of $E_\star$, or we take 
	$F_{x_1} = E_1^1 \subsetneq 0$ with $\beta_1^1 = 1/3$.
	When $D' = x_2$, $F_\star$ must have the same parabolic structure as
	$E_\star$ at $x_2$. Finally, when $D' = x_1+ x_2$, the parabolic
	structure at $x_2$ must agree with that of $E_\star$ and the parabolic
	structure at $x_1$ can either agree with that of $E_\star$ or be given
	by $E_{x_1} = E_1^1 \subsetneq 0$ with $\beta_1^1 = 1/3$, as in the
	case $D = x_1$.

	Generalizing this, there are $\prod_{j=1}^n (2^{n_j-1}+1)$ vector bundles
	refined by a bundle $E_\star$ with respect to a divisor $D = x_1 +
	\cdots + x_n$.
\end{example}

\begin{remark}
	\label{remark:refinement}
	Continuing with the notation of
	\autoref{remark:parabolic-structure-on-isomonodromic-deformation},
	for any parabolic bundle $F_\star$ refined by $E_\star$,
	there is a corresponding
	isomonodromic deformation of $(F_\star,\nabla)$ over $\Delta$ given by taking the
	parabolic structure from 
	\autoref{remark:parabolic-structure-on-isomonodromic-deformation}
	and forgetting
	part of the parabolic structure on $\mathscr{E}_\star$.
	As an important special case, the trivial parabolic structure on $E$ described in \autoref{example:vector-bundle} is refined by the parabolic bundle  $E_\star$ arising from \autoref{definition:associated-parabolic}.
\end{remark}

\subsection{Deformation theory of isomonodromic deformations}\label{subsection:filtered-bundle-deformations}
We now analyze the infinitesimal deformation theory of isomonodromic
deformations. 
\begin{notation}\label{notation:deformation-theory-iso}
    Let $\on{Art}_{\mathbb{C}}$ be the category of local Artin
$\mathbb{C}$-algebras. Let $C$ be a smooth proper curve over $\mathbb{C}$,
$D\subset C$ a reduced effective divisor with $D=x_1+\cdots+x_n$, and $$({E}, \nabla: {E}\to
{E}\otimes\Omega^1_C(\log D))$$ a flat vector bundle on $C$ with regular 
singularities along $D$. Let $P^\bullet$ be a filtration of $E$.
\end{notation}

\begin{definition}[Deformations of a curve with divisor]
	\label{definition:curve-deformation}
	Let $$\on{Def}_{(C,D)}: \on{Art}_{\mathbb{C}}\to
\on{Set}$$ be the functor sending a local Artin $\mathbb{C}$-algebra
$(A,\mathfrak m, \kappa)$ (so $\mathfrak m$ is the maximal ideal and $\kappa$ is
the residue field) to the
set of flat deformations of $(C,D)$ over $A$.
More precisely, it assigns to $A$ the set of those
$(\mathscr C, \mathscr D, q, f)$ where
$q: \mathscr C \to \on{Spec} A$ is a flat morphism, $\mathscr D \subset \mathscr C$
is a relative Cartier divisor over $\on{Spec} A$ and $f: C \to \mathscr C$ is a map
inducing an isomorphism $C \to \mathscr C \times_{\on{Spec} A} \on{Spec} \kappa$
taking $D$ isomorphically to $\mathscr D \times_{\on{Spec} A} \on{Spec} \kappa$.
\end{definition}
\begin{proposition}\label{proposition:curve-deformation}
    With notation as in \autoref{definition:curve-deformation}, there is a
    canonical and functorial bijection
    $$\on{Def}_{(C,D)}(\mathbb{C}[\varepsilon]/\varepsilon^2)\overset{\sim}{\to} H^1(C, T_C(-D)).$$
\end{proposition}
\begin{proof}
This is standard, see 
\cite[Proposition 3.4.17]{sernesi:deformations-of-algebraic-schemes}.
\end{proof}
We next generalize the above to describe the deformation theory of filtered
vector bundles on curves.
\begin{definition}[Deformations of a parabolic filtered vector bundle]
	\label{definition:}
	Let
$$\on{Def}_{(C,D, {E_\star}, P^\bullet)}:\on{Art}_{\mathbb{C}}\to \on{Set}$$ be the functor
sending $A$ to the set of flat deformations of $(C, D, E_\star, P^\bullet)$ over $A$.
More precisely, it assigns to $A$ the set of
$(\mathscr C, \mathscr D, q, f, \mathscr E, \{\mathscr E^i_j\}, \mathscr{P}^\bullet, \psi)$
where $(\mathscr C, \mathscr D, q, f)$ is a flat deformation of $(C,D)$ over $A$
as in \autoref{definition:curve-deformation}, $\mathscr E$ is a vector bundle on $\mathscr C$, 
$\oplus_j\mathscr E^i_j$ are sub-bundles of $\mathscr{E}|_\mathscr D$,
$\mathscr{P}^\bullet$ is a filtration of $\mathscr{E}$ by sub-bundles, 
and $\psi: f^* (\mathscr E,  \mathscr{P}^\bullet) \to (E,
 P)$ is an isomorphism of filtered vector bundles on $C$ inducing an isomorphism $\mathscr{E}^i_j|_{x_j}\overset{\sim}{\to} E^i_j$ for each $i,j$.
\end{definition}
\begin{proposition}\label{proposition:filtered-bundle-deformation}
    Let $(E_\star, P^\bullet)$ be a filtered parabolic vector bundle on a curve $C$. Let $D\subset C$ be a reduced effective divisor. 
    There is a canonical and functorial bijection $$\on{Def}_{(C, D, E_\star,
    P^\bullet)}(\mathbb{C}[\varepsilon]/\varepsilon^2)\overset{\sim}{\to} H^1(C,
    \on{At}_{(C,D)}(E_\star,P^\bullet)).$$
\end{proposition}
\begin{proof}
	In the non-parabolic case, this is explained in \cite[\S2.2]{BHH:logarithmic}.
	The parabolic case is described in \cite[Lemma 3.2]{BHH:parabolic} for
	the case of a $1$-step filtration $P^\bullet = (0 \subset P^1 \subset
	E)$, and the case of arbitrary length filtrations is analogous.
\end{proof}
\begin{remark}
	\label{remark:}
	If $P^\bullet$ is trivial, we omit it from the notation. In particular,
	$$\on{Def}_{(C, D,
	E_\star)}(\mathbb{C}[\varepsilon]/\varepsilon^2)\overset{\sim}{\to} H^1(C,
	\on{At}_{(C,D)}(E_\star)).$$
\end{remark}

Now, begin with a flat vector bundle $(E, \nabla)$ with regular singularities
along $D \subset C$, and let $E_\star$ denote the associated parabolic bundle defined in
\autoref{definition:associated-parabolic}.
There is an evident natural transformation $$\on{Forget}: \on{Def}_{(C, D,
{E_\star})}\to \on{Def}_{(C,D)}$$ given by forgetting ${E_\star}$. The construction of
isomonodromic deformations yields a section $$\on{iso}: \on{Def}_{(C,D)}\to
\on{Def}_{(C,D,E_\star)}$$ to this map (which depends on $\nabla$), as we now spell
out.

\begin{proposition}
	\label{proposition:connection-h1}
	With notation as above, let $\delta: \on{At}_{(C,D)}(E_\star)\to T_C(-D)$ be the natural quotient map, and
$$q^\nabla: T_C(-D)\to\on{At}_{(C,D)}({E_\star})$$ be the section to $\delta$ 
described in \autoref{definition:associated-parabolic}
arising
from $\nabla$ via \autoref{proposition:atiyah-splittings} and
\cite[Lemma 4.1]{BHH:parabolic}.
    Under the natural identifications
    $$\on{Def}_{(C,D)}(\mathbb{C}[\varepsilon]/\varepsilon^2)\overset{\sim}{\to}
    H^1(C, T_C(-D))$$ and $$\on{Def}_{(C,D,
    {E_\star})}(\mathbb{C}[\varepsilon]/\varepsilon^2)\overset{\sim}{\to} H^1(C,
    \on{At}_{(C,D)}({E_\star}))$$
    arising from \autoref{proposition:curve-deformation} and \autoref{proposition:filtered-bundle-deformation}, the two squares below commute:
    $$\xymatrix@C=.8em{
    \on{Def}_{(C,D, {E_\star})}(\mathbb{C}[\varepsilon]/\varepsilon^2) \ar[r]^-\sim
    \ar[d]_{\on{Forget}}&  H^1(C, \on{At}_{(C,D)}({E_\star})) \ar[d]_{\delta_*} &
    \on{Def}_{(C,D, {E_\star})}(\mathbb{C}[\varepsilon]/\varepsilon^2) \ar[r]^-\sim &
    H^1(C, \on{At}_{(C,D)}({E_\star})) \\
    \on{Def}_{(C,D)}(\mathbb{C}[\varepsilon]/\varepsilon^2) \ar[r]^-\sim & H^1(C,
    T_C(-D))  & \on{Def}_{(C,D)}(\mathbb{C}[\varepsilon]/\varepsilon^2) \ar[r]^-\sim
    \ar[u]_{\on{iso}} & H^1(C, T_C(-D)) \ar[u]_{(q^\nabla)_*}.
    }$$
   \end{proposition}
\begin{proof}
	This is explained in \cite[Lemma 3.1 and Lemma 4.3]{BHH:parabolic}.
	Also see \cite[\S2.2 and \S 4.1]{BHH:logarithmic} for the non-parabolic
	case.
\end{proof}
We recall one additional result describing when a filtration extends to a
deformation.

The proof of the following lemma in the non-parabolic case is explained following the proof of 
\cite[Lemma 3.1]{BHH:logarithmic}.

\begin{lemma}
	\label{lemma:parabolic-extension-obstruction}
	Suppose we are given $(C,D,E_\star,\nabla)$ as in
	\autoref{proposition:connection-h1}
	and a filtration $P^\bullet$ of $E$ as in
	\autoref{notation:deformation-theory-iso}.
	Assume
	further we have a first-order deformation $(\mathscr C, \mathscr D)$ of $(C, D)$
	corresponding to an element $s \in \on{Def}_{(C,D)}(\mathbb{C}[\varepsilon]/\varepsilon^2)\overset{\sim}{\to} H^1(C, T_C(-D))$.
	With $q^\nabla$ as in 
	\autoref{proposition:connection-h1},
	suppose $q^\nabla(s)$
	corresponds to a deformation $(\mathscr C, \mathscr D, \mathscr E_\star)$ of
	$(C,D,E_\star)$
in which $P^\bullet \subset E$ admits an extension to a filtration $\mathscr
P^\bullet$
of $\mathscr E$.
Then $$q^\nabla(s) \in \ker
\left(H^1(C, \on{At}_{(C,D)}(E_\star)) \to H^1(C,
\sheafend(E_\star)/\sheafend(E_\star, P^\bullet))\right).$$
\end{lemma}
\begin{proof}
	Note that the map
	$$H^1(C, \on{At}_{(C,D)}(E_\star)) \to
	H^1(C,\sheafend(E_\star)/\sheafend(E_\star, P^\bullet))$$
is induced by the surjection of sheaves
$\on{At}_{(C,D)}(E_\star)
\to \sheafend(E_\star)/\sheafend(E_\star, P^\bullet)$
from \eqref{equation:quotient-atiyah-bundles}.
In the above situation,
	$q^\nabla(s) \in H^1(C, \on{At}_{(C,D)}(E_\star))$ is in the image of the natural map
	$$H^1(C, \on{At}_{(C,D)}(E_\star, P^\bullet))\to H^1(C,
	\on{At}_{(C,D)}(E_\star))$$ and the composition
	$$H^1(C, \on{At}_{(C,D)}(E_\star, P^\bullet)) \to H^1(C,
	\on{At}_{(C,D)}(E_\star)) \to H^1(C,
\sheafend(E_\star)/\sheafend(E_\star, P^\bullet))$$
vanishes. Indeed, this composition is part of the long exact sequence in cohomology induced by the short exact
sequence of sheaves (\ref{equation:quotient-atiyah-bundles}).
Therefore, 
$q^\nabla(s) \in \ker
\left(H^1(C, \on{At}_{(C,D)}(E_\star)) \to H^1(C,
\sheafend(E_\star)/\sheafend(E_\star, P^\bullet))\right).$
\end{proof}

\section{Hodge-theoretic preliminaries}\label{section:hodge-theoretic-preliminaries}

We briefly recall the definition of a variation of Hodge structure, and some standard positivity and semistability results for the (parabolic Higgs) bundles associated to such variations.
In particular, \autoref{lemma:positive-Hodge-bundle}, which shows the first
filtered piece of the Hodge filtration tends to have positive parabolic degree, is crucial
to our semistability arguments.
The properties in \autoref{prop:basic-facts} will also be used repeatedly in this
paper

\subsection{Complex variations of Hodge structure}
Let $X$ be a smooth irreducible complex variety.
\begin{definition}[Polarizable complex variations of Hodge structure]
	\label{definition:complex-variation}
	A {\em complex variation of Hodge structure} on $X$ is a triple $(V,
	V^{p,q}, D)$, where $V$ is a $C^\infty$ complex vector bundle on $X$,
	$V=\oplus V^{p,q}$ is a direct sum decomposition, and $D$ is a flat
	connection satisfying Griffiths transversality: $$D(V^{p,q})\subset
	A^{1,0}(V^{p,q})\oplus A^{0,1}(V^{p,q})\oplus A^{1,0}(V^{p-1,q+1})\oplus
	A^{0,1}(V^{p+1, q-1}).$$
	A {\em polarization} on $(V, V^{p,q}, D)$ is a flat Hermitian form $\psi$ on $V$ such that the $V^{p,q}$ are orthogonal to one another under $\psi$, and such that $(-1)^p\psi$ is positive definite on each $V^{p,q}$. 
	A {\em polarizable complex variation of Hodge structure} is a complex
	variation of Hodge structure which admits a polarization.
    
    We call the holomorphic flat vector bundle $(E,
    \nabla):=(\ker(D)\otimes_{\mathbb{C}} \mathscr{O}, \on{id}\otimes d)$ the
    {\em holomorphic flat vector bundle associated to the complex variation of
    Hodge structure}. The filtration $F^pV:=\oplus_{j\geq p} V^{j,q}$ of $V$ induces a decreasing \emph{Hodge filtration} $F^\bullet V$ by holomorphic sub-bundles, such that \begin{equation}\label{eqn:griffiths-transversality} \nabla(F^p)\subset F^{p-1}\otimes \Omega^1_X.\end{equation}
    
    If $\mathbb{V}$ is a complex local system on $X$ which is isomorphic to
    $\ker(D)$ for some polarizable complex variation of Hodge structure $(V, V^{p,q}, D)$, we say that $\mathbb{V}$ \emph{underlies a polarizable complex variation of Hodge structure}.
\end{definition}

For the next definition, recall that in, in the case of curves, we defined the
residues of a connection with regular singularities in
\autoref{subsubsection:residue}.
In the case of higher dimensional varieties see 
\cite[p. 53]{deligne:regular-singular}.

\begin{definition}[Deligne canonical extension \protect{\cite[Remarques 5.5(i)]{deligne:regular-singular}}]
	   Let $\overline{X}$ be a smooth projective variety containing $X$ as a dense open subvariety with simple normal crossings complement $Z$.
   Let $(E, \nabla)$ be a flat holomorphic vector bundle on $X$. The \emph{Deligne canonical extension}
     $(\overline{E}, \overline{\nabla}: \overline{E}\to \overline{E}\otimes
     \Omega^1_{\overline{X}}(\log Z))$ of $(E, \nabla)$ to $\overline{X}$ is the
     unique flat vector bundle on $\overline{X}$ with regular singularities
     along $Z$, equipped with an isomorphism $(\overline{E},
     \overline{\nabla})|_X\overset{\sim}{\to}(E, \nabla)$, characterized by the
     property that all eigenvalues of its residues 
     along components of $Z$ 
     have real parts lying in $[0,1)$.
\end{definition}

\begin{definition}[The associated Higgs bundle]
    Let $(E, F^\bullet, \nabla)$ be a holomorphic vector bundle $E$ on a smooth variety $\overline{X}$, with a flat connection $\nabla$ with regular singularities along a simple normal crossings divisor $Z\subset\overline{X}$, and a decreasing filtration $F^\bullet$ by holomorphic sub-bundles satisfying the Griffiths transversality condition \begin{equation}\label{eqn:griffiths-transversality-2} \nabla(F^p)\subset F^{p-1}\otimes \Omega^1_{\overline{X}}(\log Z).\end{equation} The \emph{associated Higgs bundle} is the pair $(\oplus_i \on{gr}^i_{F^\bullet}E, \theta)$, where the \emph{Higgs field} $$\theta:=\bigoplus_i (\theta_i: \on{gr}^i_{F^\bullet}E\to \on{gr}^{i-1}_{F^\bullet}E\otimes\Omega^1_{\overline{X}}(\log Z))$$ is the $\mathscr{O}_{\overline{X}}$-linear map induced by $\nabla.$
    
    The vector bundle $E$ canonically has the structure of a parabolic bundle
    $E_\star$
    (if $X$ is a curve, this structure is described in
    \autoref{definition:associated-parabolic}, and it is
    described in general in \cite[Proposition 5.4]{arapura2019vanishing}).
    This structure induces the structure of a parabolic bundle on $\oplus_i
    \on{gr}^i_{F^\bullet} E_\star$, via
    \autoref{subsubsection:induced-subbundle} as a direct sum of subquotients of $E$,
    preserved by $\theta$. We refer to the pair $(\oplus_i
    \on{gr}^i_{F^\bullet}E_\star, \theta)$ with its parabolic structure as the \emph{parabolic Higgs bundle} associated to the variation of Hodge structure.
\end{definition}
We collect some basic facts about polarizable complex variations of Hodge structure, the canonical extensions thereof, and their associated Higgs bundles:
\begin{proposition}\label{prop:basic-facts}
Let $\overline{X}$ be a smooth projective curve, $Z\subset
\overline{X}$ a reduced divisor, and let
$X=\overline{X}\setminus Z$. 
Let $(V, V^{p,q}, D)$ be a polarizable complex variation of Hodge structure on
$X$, 
and let $(E, F^\bullet,
\nabla)$ be the holomorphic flat vector bundle associated to this variation of
Hodge structure, with its Hodge filtration. Let $(\overline{E}, \overline{\nabla})$ be its Deligne canonical extension.
Let $\overline{E}_\star$ be the parabolic bundle associated to $(\overline{E}, \nabla)$, as defined in
\autoref{definition:associated-parabolic}. 
\begin{enumerate}
    \item The local system $\mathbb{V}:=\ker(\nabla)$ associated to $(E, \nabla)$ is semisimple.
    \item The local system $\mathbb{V}$ may be canonically decomposed as $$\mathbb{V}\simeq \bigoplus_i \mathbb{L}_i\otimes
	    W_i,$$ where the $\mathbb{L}_i$ are pairwise non-isomorphic
	    irreducible complex local systems on $X$, and each $W_i$ is a complex vector
	    space. Each $\mathbb{L}_i$ underlies a polarizable complex
	    variation of Hodge structure, and each $W_i$ carries a complex polarized Hodge structure, both unique up to shifting the grading, and compatible with the variation carried by $\mathbb{V}$.
    \item $\overline{E}_\star$ has parabolic degree zero.
    \item There exists a canonical extension of $F^\bullet$ to $\overline{E}$, such that $(\overline{E}, F^\bullet, \overline{\nabla})$ satisfies the Griffiths transversality condition (\ref{eqn:griffiths-transversality-2}).
    \item The parabolic Higgs bundle $(\oplus_i
	    \on{gr}^i_{F^\bullet}\overline{E}_\star, \theta)$ associated
	    to $(\overline{E}_\star, F^\bullet, \overline{\nabla})$ is
	    parabolically polystable of degree zero. 
	    That is, there exist a collection of parabolic vector bundles
	    $E^j_\star$ with
	    $\deg(E^j_\star) = 0$ and maps
	    $\theta^j : E^j_\star \to E^j_\star \otimes \Omega^1_{\overline{X}}(\log Z)$ so that
	    both
	    $(\oplus_i\on{gr}^i_{F^\bullet}\overline{E}_\star,
	    \theta)=\oplus_j (E^j_\star, \theta^j)$
	    and for any
	    $\theta^j$-stable proper sub-bundle $H_\star \subset
	    E^j_\star$ with its induced parabolic structure, $\on{par-deg}
	    H_\star < 0$.
\end{enumerate}
\end{proposition}
\begin{proof}
	The proof of 
(1) is explained in \cite[1.11-1.12]{deligne1987theoreme} and (2) is
\cite[1.13]{deligne1987theoreme}. The proof of (3) follows from
\cite[B.3]{esnault1986logarithmic} 
while (4) is explained in e.g.~\cite[Section 7]{brunebarbe2017semi}. 
Finally,
(5) is due to Simpson \cite[Theorem 5]{simpson1990harmonic}. 
See the discussion in the introduction of \cite{arapura2019vanishing} for a nice summary of this and related topics.
\end{proof}

The next lemma is crucial in the proof of our main result
\autoref{theorem:isomonodromic-deformation-CVHS} since the positivity it
gives for $F^i \overline{E}_\star$ will contradict our later results on
semistability, unless the Hodge filtration has only a single part.
The connection to semistability is spelled out below in
\autoref{corollary:unstable}.
\begin{lemma}\label{lemma:positive-Hodge-bundle}
    Let $C$ be a smooth proper curve, $Z\subset C$ a reduced effective divisor,
    and $(V, V^{p,q}, D)$ a  a polarizable complex variation of Hodge structure
    on $C\setminus Z$. Let $(\overline{E}_\star, F^\bullet, \nabla)$ be the Deligne canonical
    extension of the associated flat holomorphic vector bundle to $C$ with its
    canonical parabolic structure.  Let $i$ be maximal such that
    $F^i\overline{E}_\star$ is non-trivial, where $F^\bullet$ is the Hodge
    filtration, and suppose that the Higgs field $$\theta_i:
    F^i\overline{E}_\star\to
    \on{gr}^{i-1}_{F^{\bullet}}\overline{E}_\star\otimes \Omega^1_C(\log Z)$$ is
    non-zero. Then $F^i\overline{E}_\star$ has positive parabolic degree.
\end{lemma}
\begin{remark}
	\label{remark:}
	\autoref{lemma:positive-Hodge-bundle} is essentially due to Griffiths in the case of real variations of
	Hodge structure on a
	smooth proper curve. In that case it follows from the curvature formula \cite[Theorem
	5.2]{griffiths1970periods}, and is observed there in some special cases
	\cite[Corollary 7.10]{griffiths1970periods}. For a more precise
	reference, see \cite[Corollary 2.2]{peters2000arakelov}, which
	immediately implies the claim for real variations of Hodge structure on a smooth proper curve. 
	However, as we were unable to find a precise reference in the case of complex
	variations on a quasi-projective curve, we now give a simple proof.
	A similar argument is given in the proof of 
	\cite[Theorem 3.8]{esnaultK:d-modules-and-finite-monodromy}.
\end{remark}
\begin{proof}[Proof of \autoref{lemma:positive-Hodge-bundle}]
The parabolic vector bundle $F^i\overline{E}_\star$ with the zero Higgs field is a quotient of
the parabolic Higgs bundle  $(\oplus_i \on{gr}^i_{F^\bullet}\overline{E}_\star, \theta)$, and
hence has non-negative parabolic degree by the fact that the latter is polystable, hence semistable, of
degree zero, by \autoref{prop:basic-facts}(5). It has degree zero if and only if
it is a direct summand of $(\oplus_i \on{gr}^i_{F^\bullet}\overline{E}_\star, \theta)$ by polystability; but this is ruled out by the nonvanishing of $\theta_i$.
\end{proof}
\begin{remark}
    By the construction of the Higgs field $\theta$, the condition that
    $\theta_i$ is non-zero in \autoref{lemma:positive-Hodge-bundle} is
    equivalent to the statement that $F^i\overline{E}_\star$ is not preserved by
    $\overline{\nabla}$. For example, it is automatically nonzero if
    $(\overline{E}_\star, \overline{\nabla})$ has irreducible monodromy and
    $F^i\overline{E}_\star$ is a proper sub parabolic bundle of
    $\overline{E}_\star$.
\end{remark}

We now spell out how \autoref{lemma:positive-Hodge-bundle}
relates to semistability.
\begin{corollary}\label{corollary:unstable}
Let $(\overline{E}_\star, F^\bullet, \nabla)$ be as in
\autoref{lemma:positive-Hodge-bundle}. Then the parabolic bundle
$\overline{E}_\star$ is not semistable.
\end{corollary}
\begin{proof}
    The parabolic vector bundle $\overline{E}_\star$ has parabolic degree zero
    by \autoref{prop:basic-facts}(3). But by
    \autoref{lemma:positive-Hodge-bundle}, $F^i\overline{E}_\star$ has positive degree, and is hence a destabilizing subsheaf.
\end{proof}

\section{Variations of Hodge structure and the Kodaira-Parshin trick}
\label{section:counterexample}

In this section we find that
variations of Hodge structure on $\mathscr{M}_{g,1}$ with monodromy which is
``big" in a suitable sense provide examples of flat vector bundles on curves
whose isomonodromic deformation to a nearby curve is never semistable. We then produce such variations of
Hodge structure via the Kodaira-Parshin trick. 
This will be used to prove
\autoref{theorem:counterexample} and contradicts earlier claimed theorems of
Biswas, Heu, and Hurtubise 
\cite[Theorem 1.3]{BHH:logarithmic},
\cite[Theorem 1.3]{BHH:irregular}, and
\cite[Theorem 1.2]{BHH:parabolic},
as described further in \autoref{remark:bhh-error}.

In \autoref{section:hodge-theoretic-results}, we will use that
variations of Hodge structure with suitably large monodromy yield flat vector
bundles which do not have isomonodromic deformations to semistable bundles.
This will be used to analyze variations of Hodge structure on an analytically very general curve.
\subsection{Construction of the counterexample}
We now set up the proof of \autoref{theorem:counterexample} and
\autoref{corollary:counterexample}. We will construct a variation of Hodge
structure over the analytic moduli stack $\mathscr{M}_{g,1}$ whose restriction to each fiber of the forgetful map $\mathscr{M}_{g,1}\to \mathscr{M}_g$ satisfies the hypotheses of \autoref{lemma:positive-Hodge-bundle}. We will do this via the Kodaira-Parshin trick
(see
\cite[Proposition 7]{parshin:algebraic-curves-over-function-fields-i}
and also
\cite[Proposition 7.1]{lawrenceV:diophantine-problems}), 
which produces a family of curves over $\mathscr{M}_{g,1}$ which is non-isotrivial when restricted to each fiber of the natural forgetful map $\mathscr{M}_{g,1}\to \mathscr{M}_g$. 
We give a proof appealing to \cite{cataneseD:answer-to-a-question-by-fujita},
but one can also prove it using \autoref{prop:basic-facts}  and
\autoref{corollary:unstable}, as we mention in
\autoref{remark:alternate-hodge-proof}
Because we had difficulty finding a suitable reference, we now present a version
of the Kodaira-Parshin trick in families.

\begin{proposition}[Kodaira-Parshin trick]
	\label{proposition:kodaira-parshin}
	Let $Y$ denote a Riemann surface of genus $g\geq 1$ with a point $p \in Y$ and
	let $Y^{\circ} := Y - p$. Choose a basepoint $y \in Y^{\circ}$.
	Suppose $G$ is a finite center-free group with a surjection
	$\phi: \pi_1(Y^{\circ},y) \twoheadrightarrow G$ which sends the loop around the puncture $p \in Y$ to 
	a non-identity element of $G$.
	Then there exists a smooth proper relative dimension $1$ map of
	analytic stacks $f: \mathscr{X} \to
	\mathscr{M}_{g,1}$ so that the fiber over a geometric point $[(C,p)] \in
	\mathscr{M}_{g,1}$ is a finite disjoint union of $G$-covers of $C$
	ramified at $p$.
\end{proposition}
We will prove this below in \autoref{subsubsection:proof-kodaira-parshin}.
\begin{remark}
	\label{remark:}
	In the finite disjoint union of $G$-covers appearing at the end of the
	statement of \autoref{proposition:kodaira-parshin},
	we can explicitly identify the finite set of $G$-covers.
	Namely, suppose $h \in \pi_1(\mathscr M_{g,2})$, viewed as an
	automorphism of the fundamental group of a $2$-pointed genus $g$ curve $\pi_1(Y^{\circ},y)$. (See \autoref{subsubsection:setup} below for an explanation of the action of $\pi_1(\mathscr{M}_{g,2})$ on $\pi_1(Y^\circ, y)$.)
	There is one cover associated to each map
	of the form $\phi_h: \pi_1(Y^{\circ},y) \to G$, with $\phi_h(g) :=
	\phi(hgh^{-1})$,
	modulo the following equivalence relation:
	we identify $\phi_h \sim \phi_{g}$ if they
	are conjugate, i.e., if there exists $m \in G$ with $\phi_h(s) = m
	\phi_g(s)m^{-1}$ for all $s \in \pi_1(Y^{\circ},y)$.
\end{remark}

\subsubsection{Setup to prove \autoref{proposition:kodaira-parshin}}
\label{subsubsection:setup}
Let $$\pi:\mathscr{M}_{g, 2}\to \mathscr{M}_{g,1}$$ be the natural forgetful map.
Let $x\in \mathscr{M}_{g,2}$ be a point. Let $\bar{x} := \pi(x) \in
\mathscr{M}_{g,1}$ and let $C^\circ\subset \mathscr{M}_{g,2}$
denote the fiber of $\pi^{-1}(\bar{x})$.

Note that $C^\circ$ is the complement of a point in a smooth proper connected curve of genus $g$. 
There is a natural short exact sequence
\begin{align}
1\to \pi_1(C^\circ, x) \to \pi_1(\mathscr M_{g,2}, x) \to \pi_1(\mathscr
M_{g,1}, \bar x)\to 1
\end{align}
associated to the map $\mathscr{M}_{g,2} \to \mathscr{M}_{g,1}$ with fiber
$C^{\circ}$.
We may obtain this from the Birman exact sequence \cite[Theorem
4.6]{farbM:a-primer} for mapping class groups,
after identifying the fundamental group for $\mathscr M_{g,n}$ with the mapping
class group of an $n$-times punctured genus $g$ surface.
(The case $n = 0$ follows from contractibility of the universal cover of
$\mathscr M_g$ \cite[Theorem 10.6]{farbM:a-primer} with covering group given by
the mapping class group,
and the case of general $n$ can be deduced from the Birman exact sequence
\cite[Theorem 4.6]{farbM:a-primer}.)

Let $G$ be a center-free finite group and suppose further there
is a surjection $$\gamma: \pi_1(C^\circ, x)\twoheadrightarrow G.$$
We assume that $\gamma$ takes the conjugacy class of the loop around the puncture of $C^\circ$ to a non-identity conjugacy class of $G$.

Define $\Gamma\subset \pi_1(\mathscr{M}_{g,2}, x)$ as the set of $h \in
\pi_1(\mathscr{M}_{g,2},x)$ such that there exists $\widetilde{\gamma}(h)\in G$ with $$\gamma(hgh^{-1})=\widetilde{\gamma}(h)\gamma(g)\widetilde{\gamma}(h)^{-1}$$ for all $g\in \pi_1(C^\circ, x).$ 

\begin{lemma}
	\label{lemma:center-free-group}
	Keeping notation as in \autoref{subsubsection:setup},
	the map $\gamma$ determines a well-defined surjective homomorphism 
	\begin{align*}
		\widetilde{\gamma}: \Gamma  & \rightarrow G\\
		h & \mapsto \widetilde{\gamma}(h).
	\end{align*}
	Moreover, $\Gamma \subset\pi_1(\mathscr{M}_{g,2},x)$ has finite index.
\end{lemma}
\begin{proof}
	We first claim that $\Gamma$ contains $\pi_1(C^\circ, x)$ and surjects
	onto $G$. Indeed, for $h\in
\pi_1(C^\circ, x)$, one may take $\widetilde{\gamma}(h)=\gamma(h)$.
Therefore, the surjectivity of $\gamma$ implies that $\widetilde{\gamma}$ is
also surjective.

	Next, we claim that for each $h$, $\widetilde{\gamma}(h)$ is uniquely determined. 
	Indeed, suppose $\widetilde{\gamma}(h)$ may be either
	$\alpha$ and $\beta$. Then we would have $\alpha \gamma(g) \alpha^{-1} =
	\beta \gamma(g) \beta^{-1}$.
	Since $\gamma$ is surjective, as shown above, we find $\alpha
	\beta^{-1}$ lies in the center of $G$, and therefore is trivial. So
	$\alpha = \beta$.	

	The uniqueness of $\widetilde{\gamma}(h)$ just established shows that $\widetilde{\gamma}$ determines a well-defined map.
This is moreover a homomorphism by the above established uniqueness, because we
then obtain $\widetilde{\gamma}(h) \widetilde{\gamma}(h') =
\widetilde{\gamma}(hh')$.   

Finally, we claim $\Gamma$ has finite index in $\pi_1(\mathscr{M}_{g,2}, x)$.
To see this, observe that there is an action of $\pi_1(\mathscr{M}_{g,2}, x)$
on the set of surjective homomorphisms $\pi_1(C^{\circ},x) \to G$ sending $\phi:
\pi_1(C^{\circ},x) \to G$ to the map $\phi^h(g) := \phi(hgh^{-1})$.
By definition, we have $h \in \Gamma$ if and only if $\gamma^h$ is conjugate to $\gamma$.
In particular, $\Gamma$ contains the stabilizer of $\gamma$ under the action of 
$\pi_1(\mathscr{M}_{g,2}, x)$.
But this stabilizer has finite index in 
$\pi_1(\mathscr{M}_{g,2}, x)$ because 
$G$ is finite and $\pi_1(C^\circ, x)$ is finitely generated. 
Therefore,
there are only
finitely many 
homomorphisms $\pi_1(C^\circ, x) \to G$,
and, in particular, finitely many such surjective homomorphisms.

\end{proof}

\subsubsection{}\label{subsubsection:proof-kodaira-parshin}\begin{proof}[Proof of \autoref{proposition:kodaira-parshin}]
Let $\widetilde{\Gamma}$ be the kernel of the map $\widetilde{\gamma}$ from
\autoref{lemma:center-free-group}.
The subgroup $\widetilde{\Gamma}$ corresponds to a finite \'etale cover
$\mathscr X^\circ \to \mathscr M_{g,2}$. 
Observe that $\mathscr M_{g,2} \subset \mathscr{M}_{g,1} \times_{\mathscr M_g} \mathscr{M}_{g,1}$
can be viewed as a dense open substack, and
let $\mathscr{X}$ be the normalization
of $\mathscr{M}_{g,1} \times_{\mathscr M_g} \mathscr{M}_{g,1}$ in the function
field of $\mathscr{X}^\circ,$ forming the following cartesian diagram
\begin{equation}
	\label{equation:}
	\begin{tikzcd} 
		\mathscr X^\circ \ar {r} \ar {d} & \mathscr X \ar {d} \\
		\mathscr M_{g,2} \ar {r} & \mathscr{M}_{g,1} \times_{\mathscr
		M_g} \mathscr{M}_{g,1}.
\end{tikzcd}\end{equation}
Restricting the natural map $\mathscr{X}\to\mathscr{M}_{g,1} \times_{\mathscr M_g} \mathscr{M}_{g,1}$ to a fiber $C$ of the universal curve $\mathscr{M}_{g,1} \times_{\mathscr M_g} \mathscr{M}_{g,1} \to \mathscr M_{g,1}$ yields a finite disjoint union of $G$-covers of $C$, ramified only over the tautological marked point of $C$.
We then take our desired relative curve $f: \mathscr X \to \mathscr{M}_{g,1} \times_{\mathscr M_g}
\mathscr{M}_{g,1} \to \mathscr M_{g,1}$ as the resulting
composition.
\end{proof}

In order to use \autoref{proposition:kodaira-parshin}, we will need to know
there are groups $G$ satisfying its hypotheses. We now provide such an example.
\begin{example}
	\label{example:kodaira-parshin}
	As a concrete example of a group $G$ to which 
	\autoref{proposition:kodaira-parshin} applies, we can take $G = S_3$ to be the
	symmetric group on three letters and identify $\pi_1(Y^\circ, y)$
with the free group on the generators $a_1, \ldots, a_g, b_1, \ldots, b_g$. The group $\pi_1(Y, y)$
is generated by $a_1, \ldots, a_g, b_1, \ldots, b_g$ with the relation
$\prod_{i=1}^g \left[ a_i, b_i \right]$.
Consider the surjection $\phi: \pi_1(C^\circ, y)\twoheadrightarrow S_3$
sending $a_1\mapsto (12), b_1 \mapsto (13)$ and sending $a_i \mapsto
\mathrm{id}, b_i
\mapsto \mathrm{id}$ for $i > 1$.
The loop around the puncture maps to 
$\phi(\prod_{i=1}^g \left[ a_i, b_i \right]) = (12)(13)(12)^{-1}(13)^{-1} =
(123) \neq \mathrm{id}$.
\end{example}

\subsubsection{}
\label{subsubsection:proof-counterexample}
\begin{proof}[Proof of \autoref{theorem:counterexample}]
Let $f: \mathscr X \to \mathscr{M}_{g,1}$ denote the map from
\autoref{proposition:kodaira-parshin}. 
Concretely, we can take $G = S_3$ and the map $\phi$ as in 
\autoref{proposition:kodaira-parshin} to be that given in
\autoref{example:kodaira-parshin}.
Define the local system $\mathbb V := R^1 f_* \mathbb C$ on $\mathscr M_{g,1}$,
and define $\mathscr{F}$ to be the vector bundle $\mathbb{V}\otimes
\mathscr{O}$. Note that $\mathscr{F}$ admits a natural (Gauss-Manin) connection $\on{id}\otimes d$. The local system $\mathbb{V}$ evidently underlies a variation of Hodge structure.

 Let $C$ be a
 fiber of the forgetful morphism $\mathscr{M}_{g,1}\to \mathscr{M}_{g}$.
 Let $X := f^{-1}(C) \subset \mathscr X$. We claim that the flat vector bundle $(\mathscr{F}, \nabla)$ satisfies the
conditions of \autoref{theorem:counterexample}, i.e.~it has semisimple monodromy and $\mathscr{F}|_C$ is not semistable.

We first check that $(\mathscr{F}, \nabla)|_C$ has semisimple monodromy. This is
true for any flat vector bundle arising from the Gauss-Manin connection on the
cohomology of a family of smooth proper varieties, by work of Deligne
\cite[Th\'eor\`eme 4.2.6]{deligne:hodge-ii}.

We now check that $(\mathscr{F}, \nabla)|_C$ is not semistable.  
By \cite[Theorem
4]{cataneseD:answer-to-a-question-by-fujita},
if $X \to C$ is not isotrivial,
$f_* \omega_f$ is a destabilizing subsheaf of $\mathscr F$.
It remains to show $X \to C$ is not isotrivial.
The fiber of $f|_X$ over a point $x\in C$ is a finite disjoint union of
finite covers of $C$, branched only over $x$. These fibers must vary in moduli as $x$ varies,
as there are only finitely many non-constant maps between any
two curves over of genus at least $2$, 
by de Franchis' theorem. 
(See \cite{de1913teorema} or \cite[Corollary 3, p.~75]{samuel1966lectures}, for
example.)
\end{proof}
\begin{remark}
	\label{remark:alternate-hodge-proof}
	We can also give a somewhat more involved proof of
	\autoref{theorem:counterexample} using \autoref{corollary:unstable} in
	place of \cite[Theorem
	4]{cataneseD:answer-to-a-question-by-fujita}, as we now explain. This argument inspired the Hodge-theoretic results \autoref{theorem:very-general-VHS} and \autoref{corollary:geometric-local-systems}, proven in \autoref{section:hodge-theoretic-results}.

	With notation as in the proof of 
	\autoref{theorem:counterexample},
	$\mathscr{F}$ has degree $0$ since it admits a flat connection, by
	\autoref{prop:basic-facts}(3).
	Therefore, it suffices to show $\mathscr{F}$ has a subsheaf of positive
degree.
The Hodge filtration exhibits $F^1\mathscr F\simeq (f|_X)_* \omega_{(f|_X)}$ as a subsheaf
of $\mathscr F$,
which is destabilizing by
\autoref{corollary:unstable}
once we verify that
$\delta: F^1\mathscr F \to F^0 \mathscr F/ F^1 \mathscr F \simeq R^1 (f|_X)_* \mathscr{O}_{X}$
is nonzero.

We now check $\delta$ is nonzero.
Locally around a point of $C$, $\delta$ 
can be identified with the derivative of the period map
\cite[Theorem 5.3.4]{carlsonMP:period-mappings}
sending a curve corresponding to a fiber of 
$f|_X : X \to C$
to the corresponding Hodge structure on its first cohomology. To show this derivative is not identically zero it suffices to show that the period map is non-constant.
More concretely, by the Torelli theorem, we only need to check
$f|_X: X \to C$ is not isotrivial.
This follows by de Franchis' theorem, as explained in the proof of
\autoref{theorem:counterexample}.
\end{remark}

\begin{remark}
	\label{remark:bhh-error}
	As noted prior to its statement,
	\autoref{theorem:counterexample} contradicts
\cite[Theorem 1.3]{BHH:logarithmic},
\cite[Theorem 1.3]{BHH:irregular}, and
\cite[Theorem 1.2]{BHH:parabolic}, which claim, for example, that any irreducible flat vector bundle on a smooth proper  curve of genus at least $2$ admits a semistable isomonodromic deformation. We now explain the gaps in the proofs of
those results. The error in \cite[Theorem 1.3]{BHH:logarithmic} occurs in
\cite[Proposition 4.3]{BHH:logarithmic}; the issue is that the map denoted
$f^*\nabla$ in diagram (4.14) does not in general exist. The proof works
correctly if $G=\on{GL}_2$. An identical error occurs in \cite[Proposition
4.4]{BHH:irregular}. A different argument is given in \cite{BHH:parabolic}.
There, the error occurs in the proof of \cite[Proposition 5.1]{BHH:parabolic}, 
in which the large diagram claimed to be commutative does not in general commute.
\end{remark}

\section{Analysis of Harder-Narasimhan filtration}
\label{section:hn-filtration}

\subsection{Main results on isomonodromic deformations} In this section, we prove the following theorem and corollary, generalizing \autoref{theorem:hn-constraints}.
For the statement of this theorem, recall the notion of a refinement of a
parabolic bundle, introduced in \autoref{definition:refinements}.
Loosely speaking, $E_\star$ refines $F_\star$ if $F_\star$ is essentially obtained
from $E_\star$ by forgetting some of the parabolic structure.
The two most
important cases for us will be when $F_\star= E_\star$
or $F_\star = E_0$ with respect to the empty divisor, so there is trivial
parabolic structure at every point.

\begin{theorem}
	\label{theorem:hn-constraints-parabolic}
	Let $(C,D)$ be hyperbolic of genus $g$ with $D = x_1 + \cdots + x_n$, and let
	$({E}, \nabla)$ be a flat
vector bundle on $C$ with regular singularities along $D$ and irreducible
monodromy. Let $E_\star$ denote a parabolic structure on $E$ refined by the parabolic structure on $E$ associated to $\nabla$.
Suppose
 $(E_\star',\nabla')$ 
 is an isomonodromic deformation (which exists by \autoref{remark:refinement})
of $({E_\star}, \nabla)$ to an analytically general nearby curve,
with Harder-Narasimhan filtration $0 = (F_\star')^0 \subset (F_\star')^1 \subset \cdots \subset
(F_\star')^m =
E_\star'$. For $1 \leq i \leq m$, let $\mu_i$ denote the slope of
$\on{gr}^{i}_{HN}E_\star' := (F'_\star)^i/(F_\star')^{i-1}$.
Then the following two properties hold.
\begin{enumerate}
	\item If $E_\star'$ is not parabolically semistable, then for every $0 < i < m$, there
	exists $j < i < k$ with $$\rk \on{gr}^{j+1}_{HN}E_\star'\cdot \rk
	\on{gr}^k_{HN}E_\star'\geq g+1.$$
\item We have $0<\mu_i-\mu_{i+1}\leq 1$ for all $i<m$. 
\end{enumerate}
\end{theorem}
\begin{corollary}\label{cor:stable-parabolic}
	Let $(C, D)$ be a hyperbolic curve of genus $g$.
Let $({E}, \nabla)$ be a flat
vector bundle on $C$ with irreducible monodromy and with regular singularities 
along $D$, and suppose that
$\on{rk}(E)<2\sqrt{g+1}$. 
Then, for $E_\star$ any parabolic structure refined by the parabolic structure on $E$ associated to $\nabla$,
the isomonodromic deformation of $E_\star$
to an analytically general nearby curve is parabolically semistable. 
\end{corollary}
The proofs are given in \autoref{subsection:proof-hn} and
\autoref{subsubsection:corollary-proof}.
This latter corollary salvages the main theorem of \cite{BHH:parabolic} in the
case of vector bundles of low rank when $E_\star$ is given the parabolic
structure associated to $\nabla$; it salvages the main theorem of
\cite{BHH:logarithmic} in low rank when $E_\star$ is given the trivial parabolic structure (in which case it agrees with \autoref{cor:stable}).

Crucial in this section will be the notion of generic global generation.

\begin{definition}[Generic global generation]
	\label{definition:}
	A vector bundle $V$ is {\em generically globally generated} if the
evaluation map $H^0(C, V) \otimes \mathscr O_C \to V$ does not factor through a proper
subbundle of $V$, i.e.~if the cokernel of this map is torsion.

We call a parabolic sheaf $E_\star$ {\em generically globally generated} if $E =
E_0$ is a vector bundle which is generically globally generated.
\end{definition}

The basic idea of the proof of \autoref{theorem:hn-constraints-parabolic} will be to show that any counterexample to
\autoref{theorem:hn-constraints-parabolic} will produce a certain semistable parabolic vector bundle of high
slope which is not generically globally generated. 
In order to see why this failure of generic global generation
leads to a contradiction, we will need some facts about (generic) global generation of vector bundles on curves, arising from Clifford's theorem.

\subsection{Preliminary results on high slope bundles with many sections}
\label{subsection:high-slope-lemmas}
We start with a bound on the dimension of the space of global sections of a vector bundle 
whose Harder-Narasimhan polygon has slopes between $0$ and $2g$.
\begin{lemma}
	\label{lemma:cohomology-bound}
	Suppose $V$ is a vector bundle on a smooth proper curve $C$ with Harder-Narasimhan filtration $0 = N^0 \subset N^1 \subset \cdots \subset N^m = V$. Suppose moreover that for each $i$, the slope of $\on{gr}^i_{N} V=N^i/N^{i-1}$ 
	satisfies $$0\leq \mu(\on{gr}^i_{N} V):=\frac{\deg(\on{gr}^i_{N}V)}{\on{rk}(\on{gr}^i_{N}V)}\leq 2g.$$ 
	Then $\dim H^0(C, V) \leq \frac{\deg V}{2} + \rk V$.
\end{lemma}
\begin{proof}
	For convenience set $W_i := \on{gr}^i_{N}V=N^i/N^{i-1}$.
Suppose $W_1, \ldots, W_k$ have slopes $> 2g-2$ and $W_{k+1}, \ldots,
	W_m$ have slopes $\leq 2g-2$.

	Using Clifford's theorem for vector bundles 
	\cite[Theorem 2.1]{brambila-pazGN:geography-of-brill-noether-loci}, 
	for $i > k$,	
	we have 
	\begin{align*}
	\dim H^0(C, W_i)
	\leq \frac{\deg W_i}{2} + \rk W_i.
	\end{align*}
	Also, for $i \leq k$, since $W_i$ are semistable, there are no maps $W_i \to
	\omega_C$.
	Therefore, $H^1(C, W_i) = 0$ when $i \leq k$.
	It follows from Riemann Roch that 
	\begin{align*}
	\dim H^0(C, W_i) = \deg W_i + (1-g) \rk W_i
	\end{align*}
	for $i \leq k$.
	Summing over $i$, we get
	\begin{align*}
		\dim H^0(C, W) &\leq \sum_{i=1}^m \dim H^0(C, W_i) \\
		&\leq \sum_{i=1}^k (\deg W_i + (1-g) \rk W_i)
+ \sum_{i=k+1}^m (\frac{\deg
	W_i}{2} + \rk W_i) \\
		&= \sum_{i=1}^m (\frac{\deg W_i}{2} + \rk W_i)
		+ \sum_{i=1}^k (\frac{\deg W_i}{2} -g \rk W_i) \\
		&= \frac{\deg W}{2} + \rk W+ \sum_{i=1}^k (\frac{\deg W_i}{2} -g \rk
		W_i).
	\end{align*}
	To conclude, it is enough to show 
	$\frac{\deg W_i}{2} -g \rk W_i \le 0$.
	However, since we were assuming the slope $\mu(W_i) \leq 2g$, we find
	$\deg W_i \leq 2g \rk W_i$ and so $\frac{\deg W_i}{2} \leq g \rk W_i$,
	as desired.
\end{proof}

The following lemma is a well known criterion for global generation, which we
spell out for completeness.
\begin{lemma}
	\label{lemma:global-generation}
	Let $V$ be a semistable vector bundle on a smooth proper curve $C$, such that the slope of $V$ satisfies $\mu(V)> 2g - 1$. Then
	$V$ is globally generated.
\end{lemma}
\begin{proof}
	Let $p\in C$ be a point. It suffices to show $V|_p$ is generated by global sections of $V$. Indeed, $V(-p)$ is a semistable bundle with slope $\mu(V(-p)) > 2g-2$. Hence
	 $H^1(C, V(-p)) = 0$, as any map $V(-p) \to \omega_C$ would be
	destabilizing.
	Since $H^1(C, V(-p)) = 0$, the sequence
	\begin{equation}
		\label{equation:}
		\begin{tikzcd}
			0 \ar {r} & V(-p) \ar {r} & V \ar {r} & V|_p \ar {r}
			& 0 
	\end{tikzcd}\end{equation}
	is exact on global sections, so
	$H^0(C, V) \otimes \mathscr O \to V \to V|_p$ is surjective, as desired.
\end{proof}

The following result will be key to the proof of
\autoref{theorem:hn-constraints-parabolic}, via
\autoref{proposition:generic-parabolic-global-generation},
as it places a constraint on the rank of a
vector bundle which is not generically globally generated.

\begin{lemma}
	\label{lemma:abstract-lower-bound-on-rank}
	Suppose $V$ is a vector bundle on a smooth proper curve $C$ with 
	$\mu(V) \geq 2g - 2$ (respectively, $> 2g-2$.)
	Assume further $U \subset V$ is a proper subbundle 
	$\delta := h^0(C, V) - h^0(C, U)$,
	and either $U = 0$ or else both
	\begin{enumerate}
		\item 	$\mu(U) \leq \mu(V)$, and
		\item each graded piece 
	$\on{gr}^i_{\on{HN}}U$ of the Harder Narasimhan filtration of $U$
	satisfies $0 \leq \mu(\on{gr}^i_{\on{HN}}U) \leq 2g$.
	\end{enumerate}
	Then, $\rk V \geq g (\rk V - \rk U)-\delta$ (respectively, $> g (\rk V - \rk
	U)-\delta$).
	In particular, if 
	$h^0(C, U)  = h^0(C, V)$, $\rk V \geq g$ (respectively, $\geq g+1$).
\end{lemma}
\begin{proof}
	In the case $U = 0$, the inequality 
	$\rk V \geq g (\rk V - \rk U)-\delta$ (respectively, $\rk V> g (\rk V - \rk
	U)-\delta$)
	is equivalent to
	$h^0(C, V) \geq (g-1) \rk V$ (respectively $h^0(C, V) > (g-1) \rk V$).
	This holds by Riemann-Roch.

	We now assume $U \neq 0$.
	Applying \autoref{lemma:cohomology-bound}, we conclude
	\begin{align*}
	H^0(C,U) \leq \frac{\deg
	U}{2} + \rk U.
	\end{align*}
	Using Riemann-Roch and the definition of $\delta$,
	\begin{align*}
	\dim H^0(C,U) + \delta = \dim H^0(C, V) \geq \deg V + (1-g)
	\on{rk}
	V.
	\end{align*}

	Combining the above gives
	\begin{align}
		\label{equation:section-bound}
		\deg V + (1-g) \rk V \leq \frac{\deg U}{2} + \rk U + \delta.
	\end{align}

	To simplify notation, we use $c := \rk V - \rk U$ to denote the corank
	of $U$ in $V$.
	Rewriting \eqref{equation:section-bound}, and using $\rk U = \rk V - c$ and $\mu(U) \leq
	\mu(V)$ gives
	\begin{align*}
		\mu(V) \rk (V) + (1-g) \rk V \leq \frac{\mu(U) \rk U}{2} + \rk U
		+ \delta
		\leq \frac{\mu(V)}{2} (\rk V - c) + \rk V - c + \delta.
	\end{align*}
	Rearranging the terms,
	and multiplying both sides by $2$, we obtain
	\begin{align*}
		(\mu(V)+2) \cdot c \leq (2g-\mu(V)) \rk V + 2 \delta.
	\end{align*}
	Since $2g-2 \leq \mu(V)$, we find 
	$2g \leq \mu(V) + 2$ and $2g - \mu(V) \leq 2$, implying
	\begin{align*}
		2g \cdot c \leq (\mu(V) + 2)c \leq (2g-\mu(V)) \rk V +2\delta
		\leq 2 \rk V +2\delta.
	\end{align*}
	Therefore, $\rk V \geq gc - \delta$. 
	In particular, if $\delta =0$, $\rk V \geq g$ as $c \geq 1$.
	In the case $2g-2 < \mu(V)$, we similarly find $\rk V > gc - \delta$
	and $\rk V \geq g + 1$ when $\delta = 0$.
\end{proof}

\subsection{A constraint on global generation of parabolic bundles}

In this subsection, we show that semistable parabolic bundles with large slope which are not generically globally generated cannot have
small rank.
This is accomplished in
\autoref{proposition:generic-parabolic-global-generation}.
Although it is a special case of 
\autoref{proposition:generic-parabolic-global-generation},
we start off by stating and proving the 
following special, yet pivotal, case, in order to convey the main idea and orient the reader. We call it ``the non-GGG lemma," where GGG stands for ``generically globally generated."

\begin{proposition}[The non-GGG lemma]
	\label{proposition:lower-bound-on-rank}
	Suppose $V$ is a nonzero semistable vector bundle on a smooth proper curve $C$
	which is not generically globally generated.
	\begin{enumerate}
		\item[(a)] If $\mu(V) > 2g - 2$, then $\rk V \geq g+1$.
		\item[(b)] If $\mu(V) = 2g-2$, then $\rk V \geq g$.
	\end{enumerate}
\end{proposition}
\begin{proof}
	The statement is trivial when $g =0$, so we assume $g \geq 0$.
	Let $U\subset V$ be the saturation of the image of the evaluation map $$H^0(C, V)\otimes \mathscr{O}_C\to V.$$
	We aim to apply
	\autoref{lemma:abstract-lower-bound-on-rank}.
	If $V$ is not generically globally generated, $U
	\subset V$ is a proper sub-bundle of $V$, with $H^0(C,U) \to H^0(C, V)$
	an isomorphism. Hence, we will be done by the final statement of
	\autoref{lemma:abstract-lower-bound-on-rank}, once we verify
	hypotheses $(1)$ and $(2)$ of
	\autoref{lemma:abstract-lower-bound-on-rank}.

	Semistability of $V$ implies $\mu(U) \leq \mu(V)$, verifying $(1)$.

	We conclude by checking hypothesis $(2)$.
	Using \autoref{lemma:global-generation},
	we may assume $2g-2 \leq \mu(V) \leq 2g - 1$. 
	Since $\mu(V) \leq 2g-1$ and $V$ is semistable, each
	graded piece $\on{gr}^i_{\on{HN}}U$ of the Harder-Narasimhan filtration of $U$ must have slope
	at most $2g-1$. Let $j$ be maximal such that $\on{gr}^j_{\on{HN}}U$ is non-zero. Since $U$ is generically globally generated, $\on{gr}^j_{\on{HN}}U$ has a global section, and therefore has
	non-negative slope. By the definition of the Harder-Narasimhan filtration, the same is true for $\on{gr}^i_{\on{HN}}U$ for all $i$.
	This verifies the final hypothesis $(2)$ of 
	\autoref{lemma:abstract-lower-bound-on-rank}, so we conclude $\rk V \geq
	g + 1$ in case $(a)$ and $\rk V\geq g$ in case $(b)$.
\end{proof}

We next wish to generalize \autoref{proposition:lower-bound-on-rank} to the parabolic setting.
\begin{remark}
	\label{remark:}
	The main difficulty in generalizing to the parabolic setting will be that
the graded parts of the Harder-Narasimhan filtration of the bundle
$U$ appearing in the proof of \autoref{proposition:lower-bound-on-rank} need no
longer have slope bounded above by $2g-1$.
We will get around this second issue in the proof of
\autoref{proposition:generic-parabolic-global-generation}
by quotienting both $U$ and $V$ by the
part of the Harder-Narasimhan filtration with slope more than $2g-2$, and
applying the ensuing argument to the resulting quotients.
\end{remark}

Before taking up this generalization, we record a couple of preliminary lemmas which will allow us to
understand generic global generation of
parabolic and coparabolic bundles.

\begin{lemma}
	\label{lemma:pushforward-degree}
Suppose $W_\star= (W,\{ W^i_j\}, \{\alpha^i_j\})$ is a parabolic bundle on $C$ with respect to a reduced divisor $D=x_1+\cdots+x_n$.
	Let $\alpha := \max_{i,j} \alpha^i_j$.	
	Then 
	$\mu_\star(W_\star) - \mu(W) \leq n\alpha$.
	Further, equality holds if and only if all $\alpha^i_j$ are equal to
	$\alpha$.
\end{lemma}
\begin{proof}
	By definition, this difference $\mu_\star(W_\star) - \mu(W)$ is
	$$\sum_{j=1}^n \sum_{i=1}^{n_j} \alpha^i_j \frac{\dim
	(W^i_j/W_j^{i+1})}{\rk W}.$$
	To verify the inequality, splitting the contribution from each point, 
	it suffices to show 
	$$\sum_{i=1}^{n_j} \alpha^i_j \frac{\dim
	(W^i_j/W_j^{i+1})}{\rk W} \leq \alpha.$$
	Indeed, this holds because, for all $j$,
	\begin{align*}
		\sum_{i=1}^{n_j} \alpha^i_j \frac{\dim
	(W^i_j/W_j^{i+1})}{\rk W} \leq \sum_{i=1}^{n_j} \alpha \frac{\dim
	(W^i_j/W_j^{i+1})}{\rk W} = \alpha \sum_{i=1}^{n_j}  \frac{\dim
	(W^i_j/W_j^{i+1})}{\rk W} = \alpha.
	\end{align*}
	Finally, equality holds in the above inequality for all $j$ if and only if all
	$\alpha^i_j$ are equal to $\alpha$.
\end{proof}

Recall from 
\autoref{definition:coparabolic-stability} that a coparabolic bundle $\widehat{E}_\star$ is defined to be
semistable if $E_\star$ is semistable.
\begin{lemma}
	\label{lemma:quotient-slope}
	Let $V_\star$ be a semistable coparabolic vector
	bundle or a semistable parabolic bundle on a curve $C$ with respect to a
	reduced divisor $D=x_1+\cdots +x_n$ and $\mu_\star(V_\star)=r+n$.
	Then any vector bundle $Q$ arising as a quotient of $V$ satisfies
	$\mu(Q)\geq r$. 
	Moreover, $\mu(Q) > r$ holds in the parabolic case if $n > 0$.
\end{lemma}
\begin{proof}
	First we deal with the case that $V_\star$ is a parabolic vector bundle.
	By \autoref{lemma:quotient-semistable}, any 
	parabolic quotient of $V_\star$ 
	(with the induced quotient structure of
	\autoref{subsubsection:induced-subbundle})	
	has parabolic slope at least $r+n$.
	Therefore, to complete the parabolic case, 
	it suffices to show that for any parabolic bundle $W_\star$ on $X$,
	$\mu_\star(W_\star) - \mu(W) \leq n$, with a strict inequality if $n >
	0$.
	The $n = 0$ case is trivial while the $n > 0$ case
	follows from \autoref{lemma:pushforward-degree} as $\alpha^i_j < 1$ for
	all $i, j$.

	Now, suppose $V_\star$ is a coparabolic bundle of the form $\widehat{W}_\star$ for $W_\star =
	(W, \{W^i_j\}, \{\alpha^i_j\})$ a parabolic bundle.
	For any $\varepsilon > 0$, there is a map $W[\varepsilon]_\star \to
	\widehat{W}_\star$ and hence a map $W[\varepsilon]_0 \to Q$.
	Let $Q^\varepsilon$ denote the image of this map, which we may endow
	with the associated quotient parabolic structure to obtain a quotient
	bundle
	$W[\varepsilon]_\star \to Q^\varepsilon_\star$.
	This implies 
	$\mu_\star(Q^\varepsilon_\star) \geq \mu_\star(W[\varepsilon]_\star) = r + n - \varepsilon$ 
	for all $\varepsilon > 0$, by \autoref{lemma:quotient-semistable}.
	By \autoref{lemma:pushforward-degree}, it follows that
        $\mu_\star(Q^\varepsilon_\star) - \mu(Q^\varepsilon) \leq n$, so
	$\mu(Q^\varepsilon) \geq r-\varepsilon$. Now $Q^\varepsilon\to Q$ has
	torsion cokernel, so $\mu(Q)\geq \mu(Q^\varepsilon)\geq r-\varepsilon$
	for all $\varepsilon>0$, giving the result.
\end{proof}

We also need the following fairly standard lemma, which has little to do with
parabolic bundles.
\begin{lemma}
	\label{lemma:cohomology-vanish-for-large-slope}
	Suppose $V$ is a vector bundle on a smooth proper curve $C$ so that each
	graded part of the Harder-Narasimhan filtration of $V$ has slope more
	than $2g - 2$. Then $H^1(C, V) = 0$.
\end{lemma}
\begin{proof}
	Let $0= N^0 \subset N^1 \subset \cdots \subset N^s = V$ denote the
	Harder-Narasimhan filtration.
	From the exact sequence
	\begin{equation}
		\label{equation:}
		\begin{tikzcd}
			H^1(C, N^{i-1}) \ar {r} & H^1(C, N^i) \ar {r} & H^1(C,
			N^i/N^{i-1}) \ar {r} & 0 
	\end{tikzcd}\end{equation}
	by induction on $i$, it is enough to show $H^1(C, N^i/N^{i-1}) = 0$ for
	every $i \leq s$.
	Since $N^i/N^{i-1}$ is semistable of slope more than $2g-2$, there are
	no nonzero maps $N^i/N^{i-1} \to \omega_C$, and so 
	$H^1(C, N^i/N^{i-1}) = 0$.
\end{proof}

Combining the above, we are able to verify the parabolic analog of
\autoref{proposition:lower-bound-on-rank}, namely \autoref{proposition:generic-parabolic-global-generation} below.
In this paper, we will only require the case $\delta = 0, c = 1$
of \autoref{proposition:generic-parabolic-global-generation}
so we encourage the reader to focus on that case.
We include the more general version, as the proof is nearly the same, 
and will be useful in future work.

\begin{proposition}
	\label{proposition:generic-parabolic-global-generation}
	Suppose $C$ is a smooth proper connected genus $g$ curve and
	$E_\star = (E, \{E^i_j\}, \{\alpha^i_j\})$ is a nonzero parabolic bundle $C$
	with respect to $D = x_1 + \cdots + x_n$.
		Suppose $\widehat{E}_\star$ is coparabolically semistable. 
		Let $U \subset \widehat E_0$ be a (non-parabolic) subbundle with
		$c := \rk E_0 - \rk U$
		and
		$\delta := h^0(C, \widehat{E}_0) - h^0(C, U)$.
	\begin{enumerate}
		\item[(I)] If $\mu_\star(\widehat{E}_\star)> 2g-2 + n$, 
		then $\rk E >
			gc - \delta$.
		\item[(II)] If $\mu_\star(\widehat{E}_\star)= 2g-2+n$, then $\rk
			E \geq gc - \delta$.
	\end{enumerate}
	In particular, if $\widehat{E}_\star$
	fails to be generically globally generated, and 
	$\mu_\star(\widehat{E}_\star)> 2g-2 + n$, then $\rk E \geq g	+ 1$.
\end{proposition}
\begin{proof}
	We first check the cases $g = 0$ and $g = 1$.
	The statement is trivial when $g = 0$.
	Case $(II)$ is trivial when $g = 1$, because
	$\rk E \geq c$. Similarly case $(I)$ is trivial when $\delta >
	0$.
	It remains to verify case $(I)$ when $g = 1$ and $\delta = 0$. 
	In this case, it is enough to show that $\rk U > 0$, 
	and since $H^0(C, U) = H^0(C, \widehat{E}_0)$,
	it is enough to show
	$H^0(C, \widehat{E}_0)\neq 0$.
	By \autoref{lemma:quotient-slope}, $\mu(\widehat{E}_0) > 0$,
	so Riemann-Roch implies $H^0(C, \widehat{E}_0) \neq 0$.

	We now assume $g \geq 2$.
	Let $V_\star := \widehat{E}_\star$ denote the given coparabolic bundle.
	As a first step, we reduce to the case that $U$ is generically globally
	generated.
	Indeed, let $U'$ denote the saturation of the image
	$H^0(C, U) \otimes \mathscr O_C \to U \to V$. 
	Since $h^0(C, U') \geq h^0(C, U)$ and $\rk V -
	\rk U' \geq \rk V - \rk U$, the result holds for $U$ if it holds for
	$U'$.
	We may therefore assume $U$ is generically globally generated.

	Let $0=N^0 \subset N^1 \subset \cdots \subset N^s = U$ denote the
	Harder-Narasimhan filtration of $U$.
	Let $t$ be the minimal index so that $\mu(N^{t+1}/N^{t}) \leq 2g-2$.
	If no such index exists, take $t = s$.
	We will show that in fact $\rk V/N^t > gc$ in case $(I)$ of the statement
	of this proposition
	and $\rk V/N^t \geq gc$ in case $(II)$.
	We will do so by verifying the hypotheses of
	\autoref{lemma:abstract-lower-bound-on-rank} applied to the subbundle
	$U/N^t \subset V/N^t$.
	To this end, using \autoref{lemma:quotient-slope},
	we find
	\begin{align}
		\label{equation:quotient-slope-bound}
		\mu(V/N^t) > 2g - 2 \text{ in case $(I)$ and }
	\mu(V/N^t) \geq 2g - 2 \text{ in case $(II)$.}
	\end{align}

	We next verify 
	$h^0(C, V)- h^0(C, U) = h^0(C,V/N^t) - h^0(C, U/N^t)$,
which implies that present value of $\delta$, $h^0(C, V)- h^0(C, U)$, agrees with the value of $\delta$ we will use
in our application of \autoref{lemma:abstract-lower-bound-on-rank},
	$h^0(C,V/N^t) - h^0(C, U/N^t)$.
	Indeed, $H^1(C, N^t) = 0$ by
	\autoref{lemma:cohomology-vanish-for-large-slope}.
	Therefore, 
	$h^0(C, U/N^t) = h^0(C, U) - h^0(C, N^t)$ and 
	$h^0(C, V/N^t) = h^0(C, V) - h^0(C, N^t)$. 
	Hence,
	\begin{align*}
		h^0(C, V/N^t)- h^0(C,U/N^t) = h^0(C, V) - h^0(C, U) = \delta.
	\end{align*}

	In the case $t = s$, so $N^t = U$, we have $U/N^t = 0$, and we have verified the hypotheses of
	\autoref{lemma:abstract-lower-bound-on-rank},
	so we now assume $N^t \neq U$.
	It remains to check hypotheses $(1)$ and $(2)$ of 
	\autoref{lemma:abstract-lower-bound-on-rank}.
	We first check $(2)$. Each 
	graded piece of the Harder-Narasimhan filtration of $U/N^t$
	has slope $\leq 2g - 2$ by construction of $N^t$, which verifies
	the upper bound in $(2)$.
	We claim the lower bound in $(2)$ follows from generic global generation of $U$.
	Indeed, recall we are assuming $U$ is generically globally generated, via
	the reduction made near the beginning of the proof.
	Therefore, 
	$H^0(C, U/N^{s-1})
	= H^0(C, N^s/N^{s-1}) \neq 0$,
	as any quotient bundle of a generically globally generated bundle is
	generically globally generated.
	This implies $\mu(N^s/N^{s-1}) \geq 0$, and so $\mu(N^j/N^{j-1}) \geq 0$ for
	all $1 \leq j \leq t$.

	Finally, we verify $(1)$, again in the case $t< s$, i.e., $N^t \neq U$.
	In the previous paragraph, we showed each graded piece of the
	Harder-Narasimhan filtration of $U/N^t$ has slope at most $2g -2$, so
	we also obtain $\mu(U/N^t) \leq 2g - 2$.
	On the other hand, we have already verified that $\mu(V/N^t) \geq 2g-2$ 
	in \eqref{equation:quotient-slope-bound}, above.
	Hence,
	$\mu(U/N^t) \leq 2g-2 \leq \mu(V/N^t)$,	
	verifying $(1)$.

	Applying \autoref{lemma:abstract-lower-bound-on-rank} to the vector bundle
	$V/N_t$ with the subbundle $U/N_t$ shows 
	\begin{align*}
	    \rk V/N_t &\geq g(\rk V/N_t -
	\rk U/N_t) - \delta
	= g(\rk V -
	\rk U)-\delta
	=gc - \delta,
	\end{align*}
	where the inequality is strict in case $(I)$ because $\mu(V/N^t) > 2g -
	2$.
	This proves cases $(I)$ and $(II)$ since $V = \widehat{E}_0$. 
	
	The final statement
	holds by taking $U$ to be the saturation of the image of
	$H^0(C, \widehat{E}_0) \otimes \mathscr O_C \to \widehat{E}_0$.
	In this case, $H^0(U) = H^0(C, \widehat{E}_0)$, $\delta = 0$,
	and $c \geq 1$
	when $\widehat{E}_0$ is not generically globally generated. So we get
	$\rk E \geq g + 1$.
\end{proof}

\subsection{Reduction for the proof of \autoref{theorem:hn-constraints-parabolic}}
\label{subsection:proving-hn}
We next prove some results in preparation for the proof of
\autoref{theorem:hn-constraints-parabolic}. 
Reviewing the idea of the proof, described in \autoref{subsection:idea-of-proof},
may be helpful.
%
%
\begin{notation}
	\label{notation:graded-E}
	Let $(C, D)$ be a hyperbolic curve.
Let $(E, \nabla)$ be a flat vector bundle on $C$ with regular singularities along $D$, whose
associated monodromy representation is irreducible. 
Let $E_\star$ be a parabolic structure on $E$ refined by the canonical parabolic structure associated to $\nabla$, as in \autoref{definition:associated-parabolic} and
\autoref{definition:refinements}. Let $q^\nabla: T_C(-D) \to \on{At}_{(C,D)}(E_\star)$ be the splitting of the Atiyah exact sequence associated to $\nabla$ and described in \autoref{definition:associated-parabolic}.
Let
$N^\bullet_\star$, given by $0 = N^0_\star \subset N^1_\star \subset \cdots
\subset N^m_\star = E_\star$, be a
nontrivial filtration of $E_\star$ by parabolic subbundles (with the induced parabolic structure). In particular, note $m > 1$.
Let $\on{gr}^i_{N_\star}(E_\star) := N_\star^i/N_\star^{i-1}$ denote the
quotient parabolic bundle with the induced parabolic quotient structure as described in
\autoref{subsubsection:induced-subbundle}.
The parabolic bundle $\sheafend(E_\star)_\star /\sheafend(E_\star, N^\bullet_\star)_\star$ has a filtration 
	by sheaves whose
associated graded sheaf is of the form
\begin{align*}
	\oplus_{1 \leq i< j \leq n} \sheafhom(\on{gr}^i_{N_\star}(E_\star),
	\on{gr}^{j}_{N_\star}(E_\star))_\star.
\end{align*}
For $i<j$ define $E_\star^{i,j} :=
\sheafhom(\on{gr}^i_{N_\star}(E_\star),\on{gr}^{j}_{N_\star}(E_\star))_\star.$

Let $\Delta=\mathscr{T}_{g,n}$ be the universal cover of the analytic stack
$\mathscr{M}_{g,n}$, and let $(\mathscr{C}, \mathscr{D})$ be the universal
marked curve over $\mathscr{T}_{g,n}$. Let $0\in \Delta$ be such that
$(\mathscr{C}, \mathscr{D})_0$ is isomorphic to $(C,D)$; fix such an
isomorphism. Let $(\mathscr{E_\star}, \widetilde{\nabla})$ be the universal
isomonodromic deformation of $(E_\star, \nabla)$ to $\mathscr{C}$.
\end{notation}
We will later take the filtration $N_\star^\bullet$ to be the Harder-Narasimhan
filtration (cf. \autoref{subsubsection:harder-narasimhan-parabolic}) of $E_\star$.

By \autoref{proposition:irreduciblity-splitting-condition},
the map $q^\nabla$ yields a non-zero map 
\begin{align}\label{q-nabla-map}
	T_C(-D)\overset{q^\nabla}\longrightarrow
\on{At}_{(C,D)}(E_\star) \to
\sheafend(E_\star)/\sheafend(E_\star, N_\star^\bullet).
\end{align}

We now observe that if $N_\star^\bullet$ extends to the universal isomonodromic deformation of
$(C,D, E_\star)$ on the first-order neighborhood of $(C,D)$, the induced map on first cohomology must vanish. The reader may find it useful to consider the case where the Harder-Narasimhan filtration of $E_\star$ has only two steps.

\begin{lemma}
	\label{lemma:h1-map}
	Retain notation as in \autoref{notation:graded-E}.
If the filtration $N_\star^\bullet$ extends to a filtration on the restriction of
$(\mathscr E_\star,\widetilde{\nabla})$ to the first-order neighborhood of $(C,D)=(\mathscr{C}, \mathscr{D})_0\subset (\mathscr{C}, \mathscr{D})$, 
then the composite map
\begin{align*}
	H^1(C,T_C(-D))
\overset{(q^\nabla)_*}\longrightarrow H^1(C, \on{At}_{(C,D)}(E_\star))
\to H^1(C, \sheafend(E_\star)_\star/\sheafend(E_\star, N^\bullet_\star)_\star).
\end{align*}
induced by (\ref{q-nabla-map}) is identically zero.
\end{lemma}
\begin{proof}
By \autoref{proposition:connection-h1}
the map $$(q^\nabla)_*: H^1(C,T_C(-D))\to
H^1(C, \on{At}_{(C,D)}(E_\star))$$
induced by the connection
sends a first-order deformation of the pointed curve $(C,D)$
to the corresponding first-order deformation of the triple $(C,D,E_\star)$
obtained from isomonodromically deforming the connection $\nabla$.
But given a first-order deformation $(\widetilde{C}, \widetilde{D} ,
\widetilde{E}_\star)$
of $(C, D, E_\star)$
such that $N_\star^\bullet \subset E_\star$ admits an extension $\widetilde N_\star^\bullet$ to
$\widetilde E_\star$, 
the corresponding element of
$H^1(C, \on{At}_{(C,D)}(E_\star))$ maps to $0$ in $H^1(C,
\sheafend(E_\star)_\star/\sheafend(E_\star, N^\bullet_\star)_\star)$, 
by \autoref{lemma:parabolic-extension-obstruction}. The assumption is precisely that this is true for all elements of $H^1(C, \on{At}_{(C,D)}(E_\star))$ in the image of $(q^\nabla)_*$.
\end{proof}

We now analyze the parabolic bundles $E_\star^{i,j}:=\sheafhom(\on{gr}^i_{N_\star}(E_\star),
\on{gr}^j_{N_\star}(E_\star))_\star,$ for $i<j.$

\begin{lemma}
	\label{lemma:nonzero-induced-connection}
	With notation as in \autoref{notation:graded-E},
	for every $0 < i < m$, there exists $j,k$ with $j < i$ and $k \geq
	i + 1$ so that the nonzero map $T_C(-D)\to
	\sheafend(E_\star)_\star/\sheafend(E_\star, N^\bullet_\star)_\star$ induces a nonzero map
$\psi_{j+1, k}: T_C(-D)\to E_\star^{j+1,k}$.
\end{lemma}
\begin{proof}
	First, recall the non-zero map
$T_C(-D)\to
\sheafend(E_\star)/\sheafend(E_\star,N_\star^\bullet)$ induced by $q^\nabla$, produced  in
	\autoref{proposition:irreduciblity-splitting-condition}.
Let $j$ be maximal such that $\nabla(N^j)\subset N^i\otimes \Omega^1_C(D)$. 
Note that $j< i$ as the monodromy of $(E,\nabla)|_{C\setminus D}$ is irreducible, so $N^i$ is
not a proper flat subbundle of $(E, \nabla)$, implying $\nabla(N^i) \not \subset
N^i\otimes \Omega^1_C(D)$.
Let $k$ be minimal such that $\nabla(N^{j+1})\subset N^k\otimes\Omega^1(D)$. 
Note that $k\geq i+1$ by the definition of $j$. 
By construction, the connection induces a nonzero $\mathscr{O}_C$-linear map of parabolic bundles
\begin{align*}
	N_\star^{j+1}/N_\star^j  \to (N_\star^k/N_\star^{i}) \otimes
\Omega_C^1(D)\to (N_\star^k/N_\star^{k-1}) \otimes
\Omega_C^1(D), 
\end{align*}
or equivalently a nonzero map
\[\psi_{j+1, k}: T_C(-D) \to \sheafhom(\on{gr}^{j+1}_{N_\star}(E_\star),
\on{gr}^k_{N_\star}(E_\star))_\star = E_\star^{j+1,k}.\qedhere\]
\end{proof}

We have shown that for each $i$, there exist $j<i< k$, and a non-zero map $$T_C(-D)\to E_\star^{j+1,k} =
\sheafhom(\on{gr}^{j+1}_{N_\star}(E_\star),
\on{gr}^{k}_{N_\star}(E_\star)).$$ 

We next refine \autoref{lemma:h1-map} by showing that if its hypotheses are satisfied and if in addition $N_\star^\bullet$ is the Harder-Narasimhan filtration of $E_\star$, the  map on $H^1$ 
induced by 
$\psi_{j+1, k}: T_C(-D)\to E_\star^{j+1,k}$
must also vanish.

%
%
\begin{lemma}
	\label{proposition:filtered-h1-map}
	Use notation as in \autoref{notation:graded-E}. Suppose in addition that $N_\star^\bullet$ is the Harder-Narasimhan filtration of $E_\star$. Fix $i$ with $0<i<m$ and let $j,k,$ and $$\psi_{j+1, k}: T_C(-D)\to E_\star^{j+1,k}$$ be the data constructed in \autoref{lemma:nonzero-induced-connection}. If the filtration $N_\star^\bullet$ extends to a filtration on the restriction of
$(\mathscr E_\star,\widetilde{\nabla} )$ to the first-order neighborhood of $(C,D)=(\mathscr{C}, \mathscr{D})_0\subset (\mathscr{C}, \mathscr{D})$, 
then the map
$H^1(C, T_C(-D))\to H^1(C, E_\star^{j+1,k})$
	induced by $\psi_{j+1, k}$
vanishes.
\end{lemma}
For the proof, we will need the following two  lemmas.
\begin{lemma}
	\label{lemma:hn-sections-vanishing}
	Suppose $F_\star, G_\star$ are parabolic vector bundles on a smooth
	proper curve connected $C$ with respect to a divisor $D$
	such that the Harder Narasimhan filtrations $0 = N^0_\star \subset N^1_\star
	\subset \cdots \subset N^u_\star= F_\star$ and 
	$0 = M^0_\star \subset M^1_\star
	\subset \cdots \subset M^v_\star= G_\star$
	satisfy
 $\mu_\star(\on{gr}_{N_\star^\bullet}^i(F_\star)) >
\mu_\star(\on{gr}_{M_\star^\bullet}^j(G_\star))$
for any $1 \leq i \leq u$ and $1 \leq j \leq v$.
Then $H^0(C, F_\star^\vee \otimes G_\star) = 0$.
\end{lemma}
\begin{proof}
Observe that $F_\star^\vee \otimes G_\star$ has a filtration whose associated graded pieces are
	$\on{gr}_{N_\star^\bullet}^i(F_\star)^\vee \otimes
	\on{gr}_{M_\star^\bullet}^j(G_\star)$,
	so it is enough to show the latter have vanishing $H^0$.

	An nonzero element of $H^0(C, \on{gr}_{N_\star^\bullet}^i(F_\star)^\vee \otimes
\on{gr}_{M_\star^\bullet}^j(G_\star))$ would yield a nonzero map
	$\on{gr}_{N_\star^\bullet}^i(F_\star) \to
	\on{gr}_{M_\star^\bullet}^j(G_\star)$.
	The saturation of the image of this map would be a parabolic bundle
$H_\star$. By semistability, $\mu_\star(\on{gr}_{N_\star^\bullet}^i(F_\star)) \leq
\mu_\star(H_\star) \leq \mu_\star(\on{gr}_{M_\star^\bullet}^j(G_\star))$,
	contradicting the assumption that 
	$\mu_\star(\on{gr}_{N_\star^\bullet}^i(F_\star)) >
	\mu_\star(\on{gr}_{M_\star^\bullet}^j(G_\star))$.
\end{proof}
\begin{lemma}
	\label{lemma:h1-injection}
	Suppose $$0\to F_\star \to G_\star \to H_\star\to 0$$ is a short exact sequence of
	parabolic sheaves on a smooth proper connected curve $C$ with respect to a divisor
	$D$ and we are given a map $E \to F_\star$
	inducing the $0$ map $H^1(C, E) \to H^1(C, F_\star) \to H^1(C,
	G_\star)$.
	If additionally $H^0(C, H_\star) = 0$ then $H^1(C, E) \to H^1(C,F_\star)$ vanishes.
\end{lemma}
\begin{proof}
	The vanishing of $H^0(C, H_\star)$ yields an injection $H^1(C, F_\star)
	\to H^1(C, G_\star)$. If the composition $H^1(C, E) \to H^1(C, F_\star)
	\to H^1(C, G_\star)$ vanishes then so does $H^1(C, E) \to H^1(C,
	F_\star)$.
\end{proof}

We now proceed with the proof of \autoref{proposition:filtered-h1-map}.

\begin{proof}[Proof of \autoref{proposition:filtered-h1-map}]
	The proof is a diagram chase.
	Recall that the nonzero map in \autoref{lemma:nonzero-induced-connection} was
	constructed by beginning with the map $T_C(-D) \to
	\sheafend(E_\star)_\star/\sheafend(E_\star, N_\star^\bullet)_\star$, and then showing the induced map
	$T_C(-D) \to
	\sheafhom(N^{j+1}_\star, E_\star/N^{k-1})_\star$
	factors through a map
	$T_C(-D) \to
	\sheafhom(N^{j+1}_\star, N^k_\star/N^{k-1})_\star$,
	which in turn factors through
	$T_C(-D) \to
	\sheafhom(N^{j+1}_\star/N^j_\star, N^k_\star/N^{k-1})_\star$.

	We will show each of the above three maps vanishes on $H^1$.
	By \autoref{lemma:h1-map}, 
	$T_C(-D) \to
	\sheafend(E_\star)_\star$
	vanishes on $H^1$.
	Next, the injection of parabolic sheaves $N^{j+1}_\star \to E_\star$ induces a
surjection of parabolic sheaves $\sheafend(E_\star)_\star\twoheadrightarrow
\sheafhom(N_\star^{j+1}, E_\star)_\star$. 
From this, we obtain
a surjection
$$\sheafend(E_\star)_\star /\sheafend(E_\star,N^\bullet_\star)_\star \twoheadrightarrow \sheafhom(N_\star^{j+1},
E_\star/N_\star^{k-1})_\star.$$ 
Note that $H^2$ of parabolic sheaves vanishes on a curve, because the same is true
for usual sheaves. We therefore obtain a surjection
$$H^1(C,\sheafend(E_\star)_\star/\sheafend(E_\star,N^\bullet_\star)_\star)\twoheadrightarrow H^1(C,\sheafhom(N_\star^{j+1},
E_\star/N_\star^{k-1})_\star).$$ Thus the composition $T_C(-D)\to
\sheafend(E_\star)_\star/\sheafend(E_\star,N^\bullet_\star)_\star\to
\sheafhom(N_\star^{j+1}, E_\star/N_\star^{k-1})_\star$ induces
the zero map on $H^1$, because the first map does by \autoref{lemma:h1-map}. 

We next show the natural map $T_C(-D)\to \sheafhom(\on{gr}_{N_\star}^{j+1}(E_\star),
E_\star/N_\star^{k-1})_\star$ to be described below, vanishes on $H^1$.
Using \autoref{lemma:hn-sections-vanishing}, we find that
$$H^0(C, \sheafhom(N_\star^{j}, E_\star/N_\star^{k-1})_\star)=0.$$
Therefore, applying \autoref{lemma:h1-injection} to the short exact sequence $$0\to \sheafhom(\on{gr}_{N_\star}^{j+1}(E_\star),
E_\star/N_\star^{k-1})_\star \to \sheafhom(N_\star^{j+1},
E_\star/N_\star^{k-1})_\star \to
\sheafhom(N_\star^{j}, E_\star/N_\star^{k-1})_\star \to 0$$ 
and the map $f: T_C(-D) \to \sheafhom(\on{gr}_{N_\star}^{j+1}(E_\star),
E_\star/N_\star^{k-1})_\star$
shows that $f$ induces the $0$ map on $H^1$.

We conclude by showing the map $$\psi_{j+1, k}: T_C(-D)\to
\sheafhom(\on{gr}^{j+1}_{N_\star}(E_\star), \on{gr}^k_{N_\star}(E_\star))_\star=E_\star^{j+1,k}$$
vanishes on $H^1$.
Again by \autoref{lemma:hn-sections-vanishing},
$$H^0(C, \sheafhom(\on{gr}_{N_\star}^{j+1}(E_\star),
E_\star/N_\star^k)_\star) = 0.$$
Applying \autoref{lemma:h1-injection} to the exact sequence 
\begin{equation}
\begin{tikzpicture}[baseline= (a).base]
\node[scale=.7] (a) at (0,0){
				\begin{tikzcd}
		0 \ar {r} &  \sheafhom(\on{gr}^{j+1}_{N_\star}(E_\star),
\on{gr}^k_{N_\star}(E_\star))_\star  \ar {r} & \sheafhom(\on{gr}_{N_\star}^{j+1}(E_\star),
E_\star/N_\star^{k-1})_\star  \ar {r} & \sheafhom(\on{gr}_{N_\star}^{j+1}(E_\star),
E_\star/N_\star^k)_\star \ar {r} & 0 
				 \end{tikzcd}   
};
\end{tikzpicture}
\end{equation}
and the map $\psi_{j+1, k}$ shows that $\psi_{j+1,k}$ vanishes on $H^1$, as
desired.
\end{proof}

We now show that if the map $H^1(C, T_C(-D)) \to H^1(C, E_\star^{j+1,k})$
vanishes, we will be able to produce a coparabolic  bundle which is not generically
globally generated.
We will later apply \autoref{proposition:generic-parabolic-global-generation}
to this bundle
to obtain \autoref{theorem:hn-constraints-parabolic}(1).

\begin{lemma}
	\label{lemma:not-generically-globally-generated}
	With notation as in \autoref{notation:graded-E},
	suppose the map 
	$H^1(C, T_C(-D)) \to 
	H^1(C, E_\star^{j+1,k})$ (induced by the non-zero map $\psi_{j+1, k}: T_C(-D)\to E_\star^{j+1,k}$ of \autoref{lemma:nonzero-induced-connection}) vanishes.
Then the coparabolic bundle $\reallywidehat{(E_\star^{j+1,k})^\vee} \otimes \omega_C(D)$ is not generically globally
	generated.
\end{lemma}
\begin{proof}
	Since
	$\psi_{j+1, k}: T_C(-D) \to E_\star^{j+1,k}$ is nonzero,
we obtain a nonzero Serre dual map
\begin{align}
	\label{equation:serre-dual-map}
	\reallywidehat{(E_\star^{j+1,k})^\vee}
	\otimes \omega_C(D) \to \omega_C \otimes \omega_C(D),
\end{align}
which induces the $0$ map
\begin{align*}
	H^0(C,\reallywidehat{(E_\star^{j+1,k})^\vee} \otimes \omega_C(D) )
	\to H^0(C, \omega_C\otimes \omega_C(D))
\end{align*}
by Serre duality \autoref{proposition:serre-duality} and
\autoref{proposition:filtered-h1-map}.
In particular, 
$\reallywidehat{(E_\star^{j+1,k})^\vee} \otimes \omega_C(D)$ is not
generically globally generated. Indeed, any global section must lie in the kernel
of \eqref{equation:serre-dual-map}, which has corank one in 
$\reallywidehat{(E_0^{j+1,k})^\vee} \otimes \omega_C(D)$.
\end{proof}

This concludes our setup for proving
\autoref{theorem:hn-constraints-parabolic}(1).
To prove 
\autoref{theorem:hn-constraints-parabolic}(2), we will need the following
generalization of \autoref{lemma:global-generation} to the coparabolic setting.
\begin{lemma}
	\label{lemma:coparabolic-global-generation}
	If $V_\star$ is a semistable coparabolic bundle with respect to a
	divisor $D = x_1 + \cdots + x_n$, of coparabolic slope
	$\mu_\star(V_\star) > 2g - 1 + n$, then $V_\star$ is globally
	generated.
\end{lemma}
\begin{proof}
	Let $p\in C$ be a point. 
	Suppose $V_\star = \widehat{E}_\star$ for $E_\star$ a parabolic bundle.	
	It suffices to show $V_\star$ is generated
	by global sections at $p$. Indeed,
	$V_\star(-p)$ is a semistable coparabolic bundle with coparabolic slope
	$\mu_\star(V_\star(-p)) > 2g-2+n$. 
	Therefore, $\mu_\star(E_\star^\vee(p) \otimes \omega_C(D)) < 0$,
	implying $H^0(C, E_\star^\vee(p) \otimes \omega_C(D)) = 0$,
	as any global section would be destabilizing.
	By Serre duality \autoref{proposition:serre-duality}, (taking
		$E_\star^\vee(p) \otimes \omega_C(D)$ in place of the parabolic
	vector bundle $E_\star$ in \autoref{proposition:serre-duality}),
	$H^1(C, V_\star(-p)) = H^0(C, E_\star^\vee(p) \otimes \omega_C(D))^\vee = 0$.
	As in \autoref{lemma:global-generation},
	the long exact sequence on cohomology associated $V(-p) \to V \to V|_p$
	implies $H^0(C, V) \otimes \mathscr O \to V \to V|_p$ is surjective, as
	desired.
\end{proof}

\subsubsection{}
We now prove \autoref{theorem:hn-constraints-parabolic}.
\label{subsection:proof-hn}
\begin{proof}[Proof of \autoref{theorem:hn-constraints-parabolic}]
We use notation as in \autoref{notation:graded-E}.  We aim first to show that if $(E_\star',
\nabla')$ is the isomonodromic deformation of $(E_\star, \nabla)$ to
an analytically
general nearby curve $C'$, and if $E'_*$ is not semistable, then for every $i$ there are some $j < i < k$ with 
$\rk \on{gr}^{j+1}_{HN}E_\star'\cdot \rk \on{gr}^k_{HN}E_\star'\geq g+1.$

By \cite[Lemma 5.1]{BHH:irregular},
the locus of bundles in a family $\mathscr E_\star$ on $\mathscr C \to \Delta$
which are not semistable form a closed analytic subset,
and if a general member is not semistable, then, after passing to an open subset
of $\Delta$, there is a filtration on $\mathscr E_\star$ restricting to the
Harder-Narasimhan filtration on each fiber. Thus after replacing $(C,D)$ with an
analytically general nearby curve $(C',D')$, and replacing $(E_\star,
\nabla)$ with
the restriction $(E_\star', \nabla')$ of the isomonodromic deformation to $(C',D')$,
we may assume the Harder-Narasimhan filtration $HN^\bullet$ of $E_\star'$ extends to a filtration of $\mathscr{E_\star}$ on a first-order neighborhood of $C'$. We set $(E'_\star)^{i,j} :=
\sheafhom(\on{gr}^i_{HN}(E'_\star),\on{gr}^{j}_{HN}(E'_\star))_\star$; note that $(E'_\star)^{i,j}$ is semistable by the definition of the Harder-Narasimhan filtration.

We next verify that for every $0 < i < m$, there is some $j < i < k$ for which
$\reallywidehat{((E'_\star)^{j+1,k})^\vee} \otimes \omega_C(D)$ is not generically globally generated.
By \autoref{lemma:nonzero-induced-connection}, 
for every $0 < i < m$, 
there is some $j < i$ and $k \geq i+1$ so that the map $T_{C'}(-D') \to
\sheafend(E_\star')/\sheafend_{HN^\bullet}(E_\star')$
induces a nonzero map
$$T_{C'}(-D') \to (E'_\star)^{j+1,k}:=\sheafhom(\on{gr}^{j+1}_{HN}E_\star', \on{gr}^k_{HN} E_\star')_\star.$$
By \autoref{proposition:filtered-h1-map},
$H^1(C', T_{C'}(-D')) \to 
H^1(C', (E'_\star)^{j+1,k})$ vanishes.
Hence, by
\autoref{lemma:not-generically-globally-generated},
$\reallywidehat{((E'_\star)^{j+1,k})^\vee} \otimes \omega_C(D)$
is not generically globally generated. Note that 
$\reallywidehat{((E'_\star)^{j+1,k})^\vee} \otimes \omega_C(D)$
has slope $> 2g - 2 + n$ by \autoref{lemma:degree-in-duals}.

We are finally in a position to prove \autoref{theorem:hn-constraints-parabolic}(1).
It follows from \autoref{proposition:generic-parabolic-global-generation} that $$\rk
\on{gr}^{j+1}_{HN}E_\star'\cdot \rk \on{gr}^k_{HN}E_\star'= \rk
\reallywidehat{((E'_\star)^{j+1,k})^\vee}
\otimes \omega_C(D) \geq g+1.$$ Thus \autoref{theorem:hn-constraints-parabolic}(1) holds.

We now conclude by verifying \autoref{theorem:hn-constraints-parabolic}(2).
By \autoref{lemma:coparabolic-global-generation},
\begin{align*}
\reallywidehat{((E')_\star^{j+1,k})^\vee} \otimes \omega_C(D) =
	\reallywidehat{\sheafhom(\on{gr}^{j+1}_{HN}E_\star',
	\on{gr}^{k}_{HN} E_\star')^\vee} \otimes \omega_C(D)
\end{align*}
must have slope at most $2g-1+n$, since it is not generically globally generated.
As $\sheafhom(\on{gr}^{j+1}_{HN}E_\star', \on{gr}^{k}_{HN} E_\star')$ has
negative parabolic slope by the
definition of the Harder-Narasimhan filtration and
\autoref{lemma:degree-in-duals}, we find
\begin{align*}
	-1 \leq \mu_\star(\sheafhom(\on{gr}^{j+1}_{HN}(E_\star'),
	\on{gr}^{k}_{HN}(E_\star'))) < 0.
\end{align*}

Using \autoref{lemma:degree-in-duals},
the parabolic slope of a tensor product of parabolic vector bundles is the sum of their
parabolic slopes, and taking duals negates parabolic slope.
Therefore,
\begin{align*}
	0 < \mu_\star(\on{gr}^{j+1}_{HN}({E}_\star')) -
	\mu_\star(\on{gr}^{k}_{HN}({E}_\star'))
	\leq 1.
\end{align*}

Since 
\begin{align*}
	\mu_\star(\on{gr}^{j+1}_{HN}({E}_\star')) \geq
\mu_\star(\on{gr}^{i}_{HN}({E}_\star')) \geq 
\mu_\star(\on{gr}^{i+1}_{HN}({E}_\star')) \geq
\mu_\star(\on{gr}^{k}_{HN}({E}_\star')),
\end{align*}
we also conclude 
\[
	0 < \mu_\star(\on{gr}^{i}_{HN}({E}_\star')) -
	\mu_\star(\on{gr}^{i+1}_{HN}({E}_\star')) \leq 1.\qedhere
\]
\end{proof}

\subsubsection{}
\label{subsubsection:corollary-proof}
We now prove \autoref{cor:stable-parabolic}, using
\autoref{theorem:hn-constraints-parabolic}
and the AM-GM inequality.
\begin{proof}[Proof of \autoref{cor:stable-parabolic}]

If $(E_\star', \nabla')$ is an isomonodromic deformation of $(E_\star,\nabla)$ to an
	analytically general nearby curve which is not semistable, it follows from
	\autoref{theorem:hn-constraints-parabolic}(1) that for each $i$, there will be $j,k$ with
	$j<i<k$ so that the Harder-Narasimhan filtration $HN$ of $E_\star'$ satisfies
$\rk \on{gr}^{j+1}_{HN}E_\star'\cdot \rk
	\on{gr}^k_{HN}E_\star'\geq g+1.$
	Since $\rk \on{gr}^{j+1}_{HN}E_\star' + \rk \on{gr}^k_{HN}E_\star' \leq \rk E_\star' = \rk
	E_\star$,
	it follows from the arithmetic mean-geometric mean inequality that
	$$g+1\leq \rk \on{gr}^{j+1}_{HN}E_\star'\cdot \rk \on{gr}^k_{HN}E_\star'\leq \left(\frac{\rk
E_\star}{2} \right)^2.$$
So
$\rk E_\star \geq 2 \sqrt{g+1}$ as desired.
\end{proof}

\subsubsection{}
\label{subsection:proof-intro-stable}
We conclude by proving \autoref{theorem:hn-constraints} and
\autoref{cor:stable}.
\begin{proof}[{Proof of \autoref{theorem:hn-constraints} and \autoref{cor:stable}}]
	These follow immediately from \autoref{theorem:hn-constraints-parabolic} and \autoref{cor:stable-parabolic} respectively, taking
	$E_\star$ to have the trivial parabolic structure.
\end{proof}

\section{Variations of Hodge structure on an analytically general curve}\label{section:hodge-theoretic-results}
We prove \autoref{theorem:isomonodromic-deformation-CVHS} in
\autoref{subsection:iso-def-proof},
\autoref{theorem:very-general-VHS} in \autoref{subsection:very-general-vhs}, 
\autoref{corollary:abelian-schemes} in
\autoref{subsubsection:geometric-origin-proof},
and
\autoref{corollary:geometric-local-systems} in
\autoref{subsubsection:abelian-scheme-proof}.
Finally, we prove an additional application concerning maps from very general
curves to Hilbert modular varieties in \autoref{subsection:very-general-vhs}.

\subsection{The proof of \autoref{theorem:isomonodromic-deformation-CVHS}}
\label{subsection:iso-def-proof}
We start with the following lemma, which gives a useful criterion for showing a
monodromy representation is unitary.
\begin{lemma}
	\label{lemma:unitary}
	Suppose $(C, x_1, \ldots, x_n)$ is an $n$-pointed hyperbolic curve and $(E,
	\nabla)$ is a flat vector bundle on $C\setminus\{ x_1, \ldots, x_n\}$
	underlying a polarizable complex variation of Hodge structure. Let
	$(\overline{E}_\star, \overline{\nabla})$ be the Deligne canonical extension of $(E, \nabla)$ to a flat vector bundle on $C$ with regular singularities at the $x_i$, with the parabolic structure associated to $\overline{\nabla}$.
	If $\overline{E}_\star$ is semistable, then the representation of $\pi_1(C\setminus\{x_1, \cdots, x_n\})$ 
	associated to $(E, \nabla)$ is unitary.
\end{lemma}
\begin{remark}
	\label{remark:simpson-reference}
	As pointed out by a referee, \autoref{lemma:unitary} can also be deduced from the results of
	\cite{simpson:iterated-destabilizing-modifications} if $D=\emptyset$, see \cite[\S5.4]{simpson:iterated-destabilizing-modifications}. 
	 
	Presumably a parabolic version of Simpson's results would imply  \autoref{lemma:unitary} in general, but as far as we know such results have not yet appeared in the literature. 
\end{remark}

\begin{proof}[Proof of \autoref{lemma:unitary}]
	By \autoref{prop:basic-facts}(2), we may write
	$\mathbb{V}:=\ker(\nabla)$ as $$\mathbb{V}:=\bigoplus_i
	\mathbb{L}_i\otimes W_i$$ where the $\mathbb{L}_i$ each have irreducible
	monodromy and also underlie polarizable variations of Hodge structure,
	and the $W_i$ are constant complex Hodge structures.
	It suffices to show the representation associated to each $\mathbb L_i$ has unitary monodromy.
	We may therefore reduce to the case that $\mathbb{V} = \mathbb L_i$ and assume
	that $(E,\nabla)$ has irreducible monodromy.

Let $i$ be maximal such that $F^i\overline{E}_\star$ is non-zero. 
Since $\overline{E}_\star$ is semistable, it follows from
\autoref{corollary:unstable} that
that the natural map 
\begin{align*}
	F^i\overline{E}_\star \to
	F^{i-1}\overline{E}_\star/F^i\overline{E}_\star\otimes \omega_C
\end{align*}
induced by the connection is zero, i.e.~the connection preserves
$F^i\overline{E}_\star$. By irreducibility
of the monodromy of $(E, \nabla)$, we must have that $F^i\overline{E}_\star$ equals
$\overline{E}_\star.$
But in this case $({E}, \nabla)$ is unitary, as the monodromy preserves the polarization, a definite Hermitian form.
\end{proof}

\subsubsection{}
We now recall  the setup of
\autoref{theorem:isomonodromic-deformation-CVHS}. Let
$(C, x_1, \cdots, x_n)$ 
be an $n$-pointed
hyperbolic curve 
of genus $g$.
Let $({E}, \nabla)$ be a flat vector bundle on
$C$ with $\on{rk}{E}<2\sqrt{g+1}$ such that $(E,\nabla)$ has regular singularities at the $x_i$. Our goal is to show that if an isomonodromic
deformation of ${(E, \nabla)}$ to an analytically general nearby $n$-pointed curve underlies a polarizable complex variation of Hodge structure, then $({E},\nabla)$ has unitary monodromy.

\begin{proof}[Proof of \autoref{theorem:isomonodromic-deformation-CVHS}]
As the hypothesis is about the restriction of $(E, \nabla)$ to $C\setminus \{x_1, \cdots, x_n\}$, we may without loss of generality assume $(E,\nabla)$ is the Deligne canonical extension of $(E, \nabla)|_{C\setminus \{x_1, \cdots, x_n\}}$. As local systems underlying complex polarizable variations of Hodge structure are semisimple by \autoref{prop:basic-facts}(2), we may moreover assume that $(E, \nabla)|_{C\setminus \{x_1, \cdots, x_n\}}$ has irreducible monodromy.

Endow $E$ with the parabolic structure associated to $\nabla$, and denote the
corresponding parabolic bundle by $E_\star$.
After replacing $(C, x_1, \cdots, x_n)$ with an analytically general nearby
curve, and $(E,\nabla)$ with its isomonodromic deformation to this curve, we may
assume by \autoref{cor:stable-parabolic} that $E_\star$ is semistable.	Thus $(E, \nabla)$ has unitary monodromy by \autoref{lemma:unitary}.
\end{proof}

\subsection{The proof of \autoref{theorem:very-general-VHS}} 
\label{subsection:very-general-vhs}
The proof of \autoref{theorem:very-general-VHS} follows  
from the integrality assumption and the following lemma; this lemma is well-known and appears e.g.~in the course of the proofs of \cite[Proposition 4.2.1.3]{katz:pcurvature-and-hodge} or \cite[Proposition 8.2]{esnault2020rigid}, but we have included a proof in the interests of self-containedness.

\begin{lemma}
	\label{lemma:integral-and-unitary-implies-finite}
	Suppose $\Gamma$ is a group,
	$K$ is a number field, and $\rho$ is
	a representation $$\rho: \Gamma\to \on{GL}_m(\mathscr{O}_K).$$
	If for each embedding $\iota: K\hookrightarrow \mathbb{C}$ the 
	representation $\rho\otimes_{\mathscr{O}_K, \iota} \mathbb{C}$ is unitary, then
	$\rho$ has finite image.
\end{lemma}
\begin{proof}
Indeed, for $\iota: K\hookrightarrow\mathbb{C}$ an embedding, let $$\rho_\iota:
\Gamma\to \on{GL}_m(\mathbb{C})$$ be the corresponding representation $\rho\otimes_{\mathscr{O}_K, \iota} \mathbb{C}$. First,
$$\prod_\iota \rho_\iota: \Gamma\to \prod_\iota \on{GL}_m(\mathbb{C})$$ has
image contained in a compact set by the definition of being unitary. 
Moreover, the image of $$\mathscr{O}_K\hookrightarrow \prod_\iota \mathbb{C}$$
is discrete, since the difference of any two distinct elements of
$\mathscr{O}_K$ has norm at least $1$.
Hence the image of $\prod_\iota \rho_\iota$ is discrete and has compact closure, and is therefore finite.
\end{proof}

Let $K$ be a number field with ring of integers
$\mathscr{O}_K$. Let $(C, x_1, \cdots, x_n)$ be an analytically general
hyperbolic $n$-pointed curve of genus $g$, and let $\mathbb{V}$ be a $\mathscr{O}_K$-local system on $C\setminus\{x_1, \cdots, x_n\}$ with infinite monodromy.
 Suppose that for each embedding $\iota: \mathscr{O}_K\hookrightarrow
 \mathbb{C}$, $\mathbb{V}\otimes_{\mathscr{O}_K, \iota}\mathbb{C}$ underlies a
 polarizable complex variation of Hodge structure. Our goal is to prove
 \autoref{theorem:very-general-VHS}, which states that $$\on{rk}_{\mathscr{O}_K}(\mathbb{V})\geq 2\sqrt{g+1}.$$

\begin{proof}[Proof of \autoref{theorem:very-general-VHS}]
	We use 
	$\mathscr{T}_{g,n}$ to denote the universal cover of $\mathscr{M}_{g,n}$.
For a fixed representation
$\rho: \pi_1(C\setminus\{x_1, \cdots, x_n\}))\to \on{GL}_m(\mathscr{O}_K)$
let $T_\rho$ denote the set of $[(C', x_1', \cdots, x_n')]\in
\mathscr{T}_{g,n}$,
for which the associated $\mathscr O_K$-local system $\mathbb V$ has the
following property:
for each embedding $\iota: \mathscr{O}_K\hookrightarrow \mathbb{C}$,
$\mathbb{V}\otimes_{\mathscr{O}_K, \iota}\mathbb{C}$ underlies a polarizable
complex variation of Hodge structure on $C' \setminus \{x_1', \ldots, x_n'\}$.
Let $M_\rho$ denote the image of $T_\rho$ under the covering map
$\mathscr{T}_{g,n} \to \mathscr M_{g,n}$.

Our goal is to show that an analytically very general point of $\mathscr M_{g,n}$ lies in the complement of
the union of the $M_\rho,$ where $\rho$ ranges over the set of representations
of $\pi_1(C\setminus\{x_1, \cdots, x_n\})\to \on{GL}_r(\mathscr{O}_K)$, with
infinite image, for $K$ a number field, and $r<2\sqrt{g+1}$.
Since there are only countably many such representations $\rho$,
it is enough to show that an analytically very general point lies in the complement
of $M_\rho$ for fixed $\rho$.
Since $(C, x_1, \ldots, x_n)$ is hyperbolic, 
$\mathscr{T}_{g,n} \to \mathscr M_{g,n}$
is a covering space of countable degree, and so the image of a closed analytic
set is locally contained in a countable union of closed analytic subsets.
It therefore suffices to show that for any $\rho$ with infinite monodromy and
rank $< 2 \sqrt{g+1}$, $T_\rho$ is contained in a closed analytic subset
of $\mathscr T_{g,n}$.

We now show such $T_\rho$ as above are contained in a closed analytic subset of
$\mathscr{T}_{g,n}$.
Indeed, suppose $\mathbb V$ is the local system associated to $\rho$ on some curve $C\setminus\{x_1, \cdots, x_n\}$, and that for each embedding $\iota: \mathscr{O}_K\hookrightarrow \mathbb{C}$, $\mathbb{V}\otimes_{\mathscr{O}_K, \iota} \mathbb{C}$, underlies a polarizable complex variation of Hodge structure. 
It is enough to show this complex polarizable variation of Hodge structure does
not extend to an analytically general nearby curve.
Indeed, if it did,
\autoref{theorem:isomonodromic-deformation-CVHS} implies
$\prod_{\iota: \mathscr O_K \to \mathbb C} \rho_\iota$
has unitary monodromy,
and
\autoref{lemma:integral-and-unitary-implies-finite} implies its monodromy is
finite.
This contradicts our assumption that $\rho$ has infinite monodromy. \end{proof}

\subsection{}\label{subsubsection:geometric-origin-proof}
\begin{proof}[Proof of \autoref{corollary:geometric-local-systems}]
Let $g\geq 1$ be an integer and let $(C, x_1, \cdots, x_n)$ be an analytically
general hyperbolic $n$-pointed curve of genus $g$. Let $U\subset C\setminus\{x_1, \cdots, x_n\}$ be a dense Zariski-open subset. Let $f: Y\to U$ be a smooth proper morphism, $i\geq 0$ an integer, and suppose $\mathbb{V}$ is a complex local system on $(C, x_1, \cdots, x_n)$ with infinite monodromy such that $\mathbb{V}|_U$ is a summand of $R^if_*\mathbb{C}$. Then we wish to show that $\dim_{\mathbb{C}}\mathbb{V}\geq 2\sqrt{g+1}$.

It suffices to show that $\mathbb{V}$ satisfies the hypotheses of
\autoref{theorem:very-general-VHS}. The existence of an
$\mathscr{O}_K$-structure follows from the fact that $R^if_*\mathbb{C}$ has a
$\mathbb{Z}$-structure. 
Let $\mathbb{W}$ be the corresponding $\mathscr{O}_K$-local system. All that remains is to verify that for each embedding
$\iota:\mathscr{O}_K\hookrightarrow \mathbb{C}$, the corresponding complex local
system $\mathbb{W}_\iota:=\mathbb{W}\otimes_{\mathscr{O}_K, \iota} \mathbb{C}$
underlies a polarizable complex variation of Hodge structure. But  each such
embedding yields a summand $\mathbb{W}_\iota|_U$ of $R^if_*\mathbb{C}$,
Galois-conjugate to the original embedding $\mathbb{W}|_U\subset
\mathbb{V}|_U\subset R^if_*\mathbb{C}.$ Any such summand underlies a polarizable
complex variation of Hodge structure, by \autoref{prop:basic-facts}(2). Now
$\mathbb{W}_\iota|_U$ extends from $U$ to $C\setminus \{x_1, \cdots, x_n\}$ and
so the same is true for the corresponding complex polarizable variation of Hodge
structure by \cite[Corollary 4.11]{schmid1973variation}, which completes the proof.
\end{proof}

\subsection{}
\label{subsubsection:abelian-scheme-proof}
\begin{proof}[Proof of \autoref{corollary:abelian-schemes}]
	The statement for relative curves $h: S \to C \setminus \{x_1, \ldots, x_n\}$ reduces to the statement for abelian varieties
	upon taking the relative Jacobian $\on{Pic}^0_h$, as the Torelli theorem implies that
	$h$ is isotrivial if $\on{Pic}^0_h$ is.
	Hence, it suffices to show that any abelian scheme $f: A \to C \setminus\{x_1,
	\ldots, x_n\}$ of relative dimension $r < \sqrt{g+1}$ is isotrivial.
	If $r < \sqrt{g+1}$ then $\rk R^1 f_* \mathbb C = 2r < 2 \sqrt{g+1}$, so $R^1 f_* \mathbb C$ has
	finite monodromy by \autoref{corollary:geometric-local-systems}.
	To show $f$ is isotrivial, it is enough to show $f$ is trivial after a finite base change. 
	After passing to a finite \'etale cover of
	$C \setminus\{x_1,\ldots, x_n\}$, we may therefore assume
	$R^1 f_* \mathbb C$ has trivial monodromy.
	In this case, the theorem of the fixed part 
	\cite[Corollaire 4.1.2]{deligne:hodge-ii}
	implies that the first part of the Hodge filtration $F^1(R^1 f_* \mathbb
	C) \subset R^1 f_* \mathbb C$ is also a trivial local system.
	Since $A$ is a quotient of 
	$F^1(R^1 f_* \mathbb C)^\vee$ by the constant local system
	$(R^1 f_* \mathbb Z)^\vee$,
	(on fibers this local system is identified with the first homology of $f$
	via Poincar\'e duality,)
	$f$ is also trivial.
\end{proof}

\subsection{}\label{subsubsection:non-density-proof}
We next prove that local systems of geometric origin with low rank are not
Zariski dense in the character variety of an analytically very general genus $g$
curve.
Recall that we use $\mathscr{M}_{B,r}(X)$ for the character variety
parametrizing conjugacy classes of semisimple representations $$\rho:
\pi_1(C\setminus\{x_1, \cdots, x_n\})\to \on{GL}_r(\mathbb{C}).$$

\begin{proof}[Proof of \autoref{corollary:non-density-of-geometric-local-systems}]
Let $(C, x_1, \cdots, x_n)$ be an analytically very general hyperbolic
$n$-pointed curve of genus $g$, and fix an integer $r$ with $1<r<2\sqrt{g+1}$. 
Note that we must have $g \neq 0$ because there are no integers $r$ with $1 < r
< 2 \sqrt{0 + 1}$.
By \autoref{corollary:geometric-local-systems}, the points of
$\mathscr{M}_{B,r}(C\setminus\{x_1, \cdots, x_n\})$ of geometric origin
correspond to representations $\rho$ with finite image whenever $r < 2
\sqrt{g+1}$. 
We wish to show these finite image representations are not Zariski dense in
$\mathscr{M}_{B,r}(C\setminus\{x_1, \cdots, x_n\})$.
This follows from
by \autoref{lemma:finite-image-not-dense}.
upon choosing a topological identification $C\setminus\{x_1, \cdots, x_n\} \simeq
\Sigma_{g,n}$. 
\end{proof}

\begin{lemma}
	\label{lemma:finite-image-not-dense}
	Let $\Sigma_{g,n}$ be a topological $n$-punctured genus $g$ surface with
	basepoint $p \in \Sigma_{g,n}$. For $r > 1$, the set of representations $\rho: \pi_1(\Sigma_{g,n},p) \to
	\on{GL}_r(\mathbb C)$ with finite image are not Zariski dense in the character
	variety $\mathscr{M}_{B,r}(\Sigma_{g,n}, p)$.
\end{lemma}
\begin{proof}

 For $s\in \mathbb{Z}_{>0}$, let $V_s\subset
 \mathscr{M}_{B,r}(\pi_1(\Sigma_{g,n},p))$ be the closed subvariety
	 consisting of those representations $\rho$ such that $$\on{Tr} [\rho(x)^s,
	 \rho(y)^s]=\on{Tr}(\on{id})$$ for all $x,y\in \pi_1(\Sigma_{g,n},p)$.

By Jordan's theorem (\cite[p.~91]{jordan1878memoire} or \cite[Theorem
36.13]{curtis1966representation}) on finite subgroups of $\on{GL}_r(\mathbb{C})$,
there is some constant $m_r$ such that for each finite subgroup $G\subset
\on{GL}_r(\mathbb{C})$, there exists an abelian normal subgroup $H\subset G$ of index
at most $m_r$. Hence $V_{(m_r!)}$ contains all representations with finite
image. Thus it suffices to show $V_s$ is not all of
$\mathscr{M}_{B,r}(\Sigma_{g,n})$ for any $s>0$. 

We now write $$\pi_1(\Sigma_{g,n},p)=\left\langle a_1, \cdots, a_g, b_1,\cdots,
b_g, \gamma_1, \cdots, \gamma_n \mid \prod_{i=1}^g [a_i,b_i]\cdot \prod_{j=1}^n
\gamma_j\right\rangle$$ for the standard presentation of the fundamental group.
If $g\geq 2$, let $$\rho: \pi_1(\Sigma_{g,n},p)\to \on{GL}_2(\mathbb{C})$$ be the representation defined by $$a_1 \mapsto \begin{pmatrix} 1 & 1 \\ 0 & 1 \end{pmatrix}$$
$$a_2 \mapsto \begin{pmatrix} 1 & 0 \\ 1 & 1 \end{pmatrix}$$
and sending all other generators to the identity.  If $g=1$, hyperbolicity of
$\Sigma_{g,n}$ implies $n>0$. Then $\pi_1(\Sigma_{g,n},x)$ is free on $n+1$ generators, and we may set $\rho$ to be a representation sending two of the generators to the matrices above, and all other generators to the identity. 
Then $\rho$ does not lie in $V_s$ for any $s>0$ by direct computation. This completes the case $r = 2$. For the case that $r > 2$, the representation $\rho\oplus \on{triv}^{\oplus r-2}$ lies outside of $V_s$ for all $s$, for the same reason.
\end{proof}

\subsection{An application to Hilbert modular varieties}
\label{subsection:hilbert-modular}

Let $K$ be a totally real number field of degree $h$ over $\mathbb Q$, $\mathscr O_K$ its ring of integers, and 
use $\mathscr X_K$ to denote the Hilbert modular stack parameterizing
$h$-dimensional principally polarized abelian varieties $A$ with an injection
$\mathscr O_K \hookrightarrow \on{End}(A)$.
Using \autoref{corollary:abelian-schemes}, we find there are no nonconstant maps
from an analytically very general $n$-pointed curve of genus $g$
to the moduli stack of principally polarized abelian varieties of dimension $h$
whenever $h \leq \sqrt{g+1}$, or equivalently when $g > h^2$.
We now show that the analogous statement for Hilbert modular stacks holds
whenever $g \geq 1$.
In particular, this improved bound is independent of $h$.
We thank Alexander Petrov for pointing out the following application.

\begin{proposition}
	\label{proposition:hilbert-modular}
	Let $\mathscr X_K$ be the Hilbert modular stack associated to a totally
	real field $K$.
	Let $(C, x_1, \ldots, x_n)$ be an analytically very general hyperbolic $n$-pointed
	curve of genus $g \geq 1$.
	Any map $\phi: C\setminus \{x_1, \ldots, x_n\} \to \mathscr X_K$ is constant.
\end{proposition}
\begin{proof}
	We first use the map $\phi$ to produce a rank $2$ integral local system
	on $C \setminus \{x_1, \ldots, x_n\}$.
	A map $\phi: C\setminus \{x_1, \ldots, x_n\} \to \mathscr X_K$ yields an
	abelian scheme $f: A \to C\setminus \{x_1, \ldots, x_n\}$ with an $\mathscr
	O_K$-action. In particular, $R^1 f_* \mathbb Z$ is a 
	complex PVHS of rank $2h$ with an	$\mathscr O_K$-action.
	If $K$ has degree $h$, then this rank $2h$ local system may not be a
	$\mathscr O_K$-local system, as its fibers may only be locally free, as opposed to free.
	However, after passing to a finite extension $K'$ of $K$, such as the
	Hilbert class field of $K$, any such rank $2$ locally free module
	becomes a free module,
	and so $R^1 f_* \mathbb Z
	\otimes_{\mathscr O_K} \mathscr O_{K'}$ is an $\mathscr O_{K'}$-local system of rank $2$.
	
	We now show the map $\phi$ must be constant.
	By construction,
	for any embedding $\mathscr O_{K'} \to \mathbb C$, the local system
$R^1 f_* \mathbb Z\otimes_{\mathscr O_K} \mathscr O_{K'} \otimes_{\mathscr
O_{K'}} \mathbb C$ underlies a $\mathbb{C}$-PVHS. Hence by \autoref{theorem:very-general-VHS}, $R^1 f_* \mathbb Z\otimes_{\mathscr O_K} \mathscr O_{K'}$ has finite monodromy, as 
\begin{align*}
2 = \rk_{\mathscr O_K'}R^1 f_* \mathbb Z\otimes_{\mathscr O_K} \mathscr O_{K'} < 2 \sqrt{g+1}.
\end{align*}
Hence, the same holds for
$R^1 f_* \mathbb Z$. 
This implies the map $\phi$ is constant using the theorem of the fixed part
\cite[Corollaire 4.1.2]{deligne:hodge-ii} (see the proof
of \autoref{corollary:abelian-schemes} for a similar argument).
\end{proof}

\section{Questions}\label{section:questions}

We conclude with some open questions related to our results.
\subsection{Bounds}
	
\begin{question}\label{question:bounds}
Is the bound of $2\sqrt{g+1}$ appearing in \autoref{cor:stable},
\autoref{theorem:very-general-VHS} and 
\autoref{theorem:isomonodromic-deformation-CVHS} sharp? If not, can one
explicitly construct low rank geometric variations of Hodge structure with
infinite monodromy on a general curve or general $n$-pointed curve?
Do there exist counterexamples to the above results if one replaces $2
\sqrt{g+1}$ by a linear function of $g$?
\end{question}
We have no reason to believe the bound is sharp. The Kodaira-Parshin trick (as
used in \autoref{section:counterexample}, for example) is one source of
variations of Hodge structure on $\mathscr{M}_{g,n}$ of rank bounded in terms of
$g, n$, but it is not the only one. For example, we expect that the representations constructed in
in \cite{koberda2016quotients} are cohomologically rigid and hence underlie
integral variations of Hodge structure by \cite[Theorem
1.1]{esnault2018cohomologically} and \cite[Theorem 3]{simpson1992higgs}.
Assuming Simpson's motivicity conjecture (\cite[Conjecture,
p.~9]{simpson1992higgs}) these constructions should be geometric in nature, though this is not clear from the construction. Of course it would be extremely interesting to prove that these representations arise from algebraic geometry.

The representations constructed in \cite{koberda2016quotients} have rank
growing exponentially in $g$ \cite[Corollary 4.5]{koberda2016quotients}.
It is natural, given our results, to ask if one can use their methods to produce representations of smaller rank.

We also raise a related question about bounds on maps to the moduli space of
curves.
\begin{question}
	\label{question:}
	Fix an integer $g\geq 2$.
	What is the smallest integer $h \geq 2$ for which the generic genus $g$ curve,
	i.e., the generic fiber of 
	$\mathscr M_{g,1} \to \mathscr M_g$,
	has a
	non-constant map to
	$\mathscr M_h$?
\end{question}
\begin{remark}
	\label{remark:}
	Since a map $C \to \mathscr{M}_h$ corresponds to a family of
smooth curves of genus $h$ over $C$, by considering the associated family of
Jacobians, it follows from \autoref{corollary:abelian-schemes} that $h \geq
\sqrt{g+1}$.
The Kodaira-Parshin trick 
\autoref{proposition:kodaira-parshin} 
does not a priori apply to construct maps from the generic curve to $\mathscr{M}_h$, because as written it produces disconnected covers.
But one can apply a variant where one takes a cover defined by a characteristic quotient of the fundamental group
to show there is some (fairly large) value of $h$ for which the generic genus
$g$ curve has a non-constant map to $\mathscr M_h$.
See \cite[Theorem
1.4]{mcmullen:from-dynamics-on-surfaces-to-rational-points-on-curves} for more
details.
\end{remark}

\subsection{Non-abelian Hodge loci}
Let $(C, x_1, \cdots, x_n)$ be an $n$-pointed curve, $\mathbb{V}$ a
$\mathbb{Z}$-local system on $C\setminus\{x_1, \cdots, x_n\}$ with
quasi-unipotent local monodromy around the $x_i$, and let $(\mathscr{E}, \nabla)$ be the associated flat vector bundle. We refer to the locus $H_{\mathbb{V}}$ in $\mathscr{T}_{g,n}$ where the corresponding isomonodromic deformation of $(\mathscr{E}, \nabla)$ underlies a polarizable variation of Hodge structure as an \emph{non-abelian Hodge locus}. By analogy to the famous result on algebraicity of Hodge loci of Cattani-Deligne-Kaplan \cite{cattani1995locus}, it is natural to ask:
\begin{question}[Compare to {\cite[Conjecture 12.3]{simpson62hodge}}] \label{question:non-abelian-Hodge}
Let $Z$ be an irreducible component of $H_{\mathbb{V}}$. Is the image of $Z$ in $\mathscr{M}_{g,n}$ algebraic? 
\end{question}
This would follow if all $\mathbb{Z}$-local systems which underlie polarizable variations of Hodge structure arise from geometry, which is perhaps a folk conjecture (and is conjectured explicitly in \cite[Conjecture 12.4]{simpson62hodge}). Just as \cite{cattani1995locus} provides evidence for the Hodge conjecture, a positive answer to \autoref{question:non-abelian-Hodge} would provide evidence for this conjecture.

When we refer to an analytically very general curve we mean in the sense
of \autoref{definition:general}.
A positive answer to \autoref{question:non-abelian-Hodge} would allow us to replace this with the usual algebraic notion of a very general curve in \autoref{theorem:very-general-VHS}. 
It seems plausible that one can make this replacement in \autoref{corollary:geometric-local-systems} without requiring input from \autoref{question:non-abelian-Hodge}, using the main result of \cite{cattani1995locus}. 
\bibliographystyle{alpha}
\bibliography{bibliography-isomonodromy}

\end{document}